\documentclass[11pt]{article}

\usepackage[hidelinks]{hyperref}
\hypersetup{
  colorlinks   = true, %Colours links instead of ugly boxes
  urlcolor     = blue, %Colour for external hyperlinks
  linkcolor    = blue, %Colour of internal links
  citecolor   = red %Colour of citations
}

\usepackage[english]{babel}
\usepackage{amsmath,amsthm,amssymb}
\usepackage{graphicx}
\usepackage{bbm} 

\newcommand{\sy}{\boldsymbol{\Psi}}
\newcommand{\py}{\boldsymbol{\Phi}}
\newcommand{\T}{\mathbb{T}^3}

\usepackage[mathscr]{eucal}
\usepackage{xcolor}
\usepackage[page,toc]{appendix}
\usepackage[margin=1in]{geometry} 
\usepackage{fancyhdr,enumerate,mathtools,mathrsfs,bm,graphicx}
\usepackage[numbers, super]{natbib}
\usepackage{url}

\usepackage{multicol}
\newcommand{\N}{\mathbbm{N}}									%basic sets

\newcommand{\R}{\mathbb{R}}

\newcommand{\vertiii}[1]{{\left\vert\kern-0.25ex\left\vert\kern-0.25ex\left\vert #1 
    \right\vert\kern-0.25ex\right\vert\kern-0.25ex\right\vert}}

\newcommand{\overbar}[1]{\mkern 1.5mu\overline{\mkern-1.5mu#1\mkern-1.5mu}\mkern 1.5mu}
\newcommand{\inner}[2]{\left\langle #1, #2 \right\rangle}
\DeclarePairedDelimiter\abs{\lvert}{\rvert}					%absolute value
\DeclarePairedDelimiter\norm{\lVert}{\rVert}				%norm

\newtheorem{theorem}{Theorem}[subsection]
\newtheorem{corollary}{Corollary}[theorem]

\newtheorem{lemma}[theorem]{Lemma}
\newtheorem{proposition}[theorem]{Proposition}
\newtheorem{remark}{Remark}
\newtheorem{definition}[theorem]{Definition}

\newtheorem{assumption}[theorem]{Assumption}

\begin{document}
	\title{Stochastic Calculus in Infinite Dimensions and SPDEs}
	\author{Daniel Goodair}
	\maketitle
	\setcitestyle{numbers}
	%\begin{abstract}
	 %   These notes aim to take the reader from an elementary understanding of functional analysis and probability theory to a robust construction of the stochastic integral in Hilbert Spaces. We consider integrals driven at first by real valued martingales and later by Cylindrical Brownian Motion, introducing this concept and expanding into a basic set up for Stochastic Partial Differential Equations (SPDEs). The framework that we establish facilitates an exceedingly broad class of SPDEs and noise structures, in which we build upon standard SDE theory and rigorously deduce a conversion between the Stratonovich and It\^{o} Forms. The study of Stratonovich equations is largely motivated by the stochastic variational principle of SALT [\cite{holm2015variational}] for fluid dynamics, and we discuss an application of the framework to a Navier-Stokes Equation with Stochastic Lie Transport as seen in [\cite{goodair2023existence}]. Moreover we prove a fundamental existence and uniqueness result (which to the best of our knowledge, is not present in the literature) for SPDEs evolving in a finite dimensional Hilbert Space driven by Cylindrical Brownian Motion, assuming the analogy to the Lipschitz and linear growth conditions for the standard existence and uniqueness theory for SDEs. 
	%\end{abstract}
	
\tableofcontents
\thispagestyle{empty}
\newpage

\setcounter{page}{1}
\addcontentsline{toc}{section}{Introduction}
\section*{Introduction}

Stochastic Differential Equations have rich applications in physics and finance, for example in Langevin Equations modelling the movement of a particle in space [\cite{coffey2012langevin}] or the Black-Scholes options pricing model for the dynamics of the price of a stock [\cite{black1973pricing}]. These are applications of classical It\^{o} Calculus, where we construct the integral of a process taking values in Euclidean space. Whilst this theory is adequate in such applications, mathematical models for physical phenomena far exceed those for the position of a particle. These notes ponder stochastic equations modelling a function of both space and time, such as the velocity or temperature of a particle. It is, therefore, necessary to define the stochastic integral \begin{equation} \label{stochint} \int_0^
t \sy_s dW_s \end{equation} for $\sy:\Omega \times [0,\infty) \times \R^n \rightarrow \R^d$ and $W$ a standard real valued Brownian Motion. We regard $\sy$ not as a pointwise defined function but rather an element of a function space, which is our motivating context for stochastic integration of Hilbert Space valued processes. This is the sense in which we mean `Stochastic Calculus in Infinite Dimensions', or at least partly, as the field of stochastic modelling has also expanded into infinite dimensional \textit{driving processes}. Our notes assume only a rudimentary understanding of functional analysis and real valued stochastic calculus, with which we explore the following three main areas of this text. These areas are ordered in increasing complexity and novelty. 

\begin{itemize}
    \item In Section \ref{section 1} we present a `classical' construction of the It\^{o} stochastic integral, for processes evolving in a Hilbert Space. This occurs first for a one dimensional driving Brownian Motion, before generalisations to other one dimensional martingales and further to \textit{Cylindrical Brownian Motion}. Our construction is direct and designed to be familiar to a reader who has undertaken the real valued study so well explicated in the likes of [\cite{karatzas1991brownian}, \cite{oksendal2013stochastic}]. In defining our infinite dimensional Brownian Motion, that is the Cylindrical Brownian Motion, we cover the fundamentals of martingale theory in Hilbert Spaces broadly by finite dimensional projections along with the real valued theory. The hope is again that this approach is accessible to a reader with background in the real valued theory. The Cylindrical Brownian Motion is the only infinite dimensional driving process that we integrate with respect to; whilst we present a background on general $Q-$Cylindrical Processes which could be viable integrators, limiting ourselves to Cylindrical Brownian Motion enables the integral to be established as a straightforwards limit of the integrals against finite dimensional Brownian Motions. In particular we avoid the operator theoretic technicalities necessary in the general case, noting that our construction is sufficient for the framework and applications that follow.

    \item Section \ref{section 2} details a framework for the study of stochastic partial differential equations (SPDEs), which are evolution equations involving integration of the form introduced in the previous section. Through this framework we define notions of solutions for an abstract SPDE, motivated in particular by the recent attention given to transport type noise (where the stochastic integral is dependent on the gradient of the solution). This attention comes largely from the seminal work [\cite{holm2015variational}], in which Holm establishes a new class of stochastic equations driven by \textit{Stratonovich} transport type noise which serve as fluid dynamics models by adding uncertainty in the transport of fluid parcels to reflect the unresolved scales. A rigorous mathematical understanding of these equations appears perilous for three key reasons. The first is the typical nonlinearity of fluid equations, rendering the more established linear theory insufficient. The second is the gradient dependency in the noise, taking us beyond the most general `variational frameworks' seen in the literature as these are posed for a noise operator which is bounded on some Hilbert Space. The third is the Stratonovich integration, which we are likely to only understand as a corrected It\^{o} integral, yet this conversion is highly nontrivial for a noise which is not bounded on a Hilbert Space.\\

    More precisely, we present a framework which shares its spirit with the variational approach to SPDEs pioneered by Pardoux in the 1970s and now best represented in the more recent books [\cite{pardoux2021stochastic}, \cite{prevot2007concise}]. This classical framework considers an evolution equation with respect to a Gelfand Triple, say $V \xhookrightarrow{} H \xhookrightarrow{} V^*$, where solutions have paths which are square integrable in $V$, continuous in $H$ and satisfying an identity in $V^*$. Recalling our motivation of fluid equations, the prototypical example in this framework is the Navier-Stokes Equation. Whilst analytically weak solutions fit this framework seamlessly, analytically strong solutions fit to the spaces $W^{2,2} \xhookrightarrow{} W^{1,2} \xhookrightarrow{} L^2$ which prompts our shift to a triplet of embedded Hilbert Spaces without any necessary duality structure. Furthermore to include a Stratonovich transport type noise we in fact introduce a fourth Hilbert Space, necessary as the It\^{o}-Stratonovich correction costs a derivative for transport type noise. We can then properly define weak, strong, local and maximal solutions of nonlinear PDEs with Stratonovich transport noise, amongst other more familiar stochastic perturbations. We believe that presenting the technical details, in such generality, of these notions here facilitates the rigorous and free analysis of the equations in future works.\\
    
    Some details of the proof of an energy equality in this setting are also presented. This is well understood in the typical variational framework, for which we again refer to [\cite{prevot2007concise}], but we take care in addressing subtle changes. The first is the loss of the duality structure, though for this result we do assume a bilinear form relation which behaves similarly. The second is that we conduct the proof for local solutions, necessary for our motivating class of equations, so it is important for us to explicitly address how the localisation affects the proof. Similarly the final key change is that we do not assume any integrability over the probability space of our processes, demanding again another source of localisation which we find worthy of detailing in these notes. The section rounds out with a simple statement of the general It\^{o} Formula, referring to [\cite{da2014stochastic}] for a proof, and a demonstration that the infinite dimensional noise can be reduced to something one dimensional if it is constant multiplicative in each direction.

\item Section \ref{section 3} concludes these notes and contains more advanced and novel techniques in the existence theory for nonlinear SPDEs. The beauty of the classical variational approach comes from the existence results, which certainly cannot be matched as elegantly in a framework built for 3D Navier-Stokes Equations and related stochastic fluid models. Instead we focus on techniques that can be used in this direction, centred around the \textit{Galerkin Method} in which finite dimensional approximations of the SPDE are considered and some properties are used to deduce their limit. Immediately then an existence result for the finite dimensional equations is required, more precisely for where the Hilbert Space in which the equation evolves is finite dimensional but the driving Brownian Motion is still infinite dimensional. We assume standard Lipschitz and linear growth conditions, and to the best of our knowledge this result is not present in the literature. There are two predominant ways to deduce the existence of a limit of the finite dimensional approximations, which we detail now.\\

The first is through \textit{tightness}, which is the stochastic route to relative compactness arguments used in PDE theory. The idea is that from tightness we can deduce relative compactness of the laws of the processes over some suitable function space, at which point Skorohod's Representation Theorem enables the deduction of a limiting process almost surely on a new probability space. Criteria to deduce tightness in relevant function spaces are thus of great significance, and our criteria comes largely from the works of [\cite{aldous1978stopping}, \cite{jakubowski1986skorokhod}, \cite{rockner2022well}]. The second is through a Cauchy type argument in the relevant spaces, difficult to execute in the case of local solutions but recently this has been overcome to great effect due to Glatt-Holtz and Ziane in [\cite{glatt2012local}] and extended by the author [\cite{goodair2023navier}]. We defer a greater discussion of this highly technical result to Section \ref{section 3}.\\

In addition, we do give a complete set of assumptions necessary to deduce the existence of maximal solutions in this framework. The result is taken from the author's work [\cite{goodair2023existence}] so no proof is given here, but the methods discussed throughout these notes are used in the result. The purpose of including this is not for the reader to understand why it is true from this work, but rather to provide an idea of what kind of impositions one must make to deduce the existence of solutions to the highly non-trivial SPDEs motivating these notes. Of course a reference to the full proof is given for the interested reader. Our exposition wraps up with applications of the framework and results to the Navier-Stokes Equation under Stochastic Advection by Lie Transport, which is the aforementioned perturbation scheme developed in [\cite{holm2015variational}]. We hope that this conveys how the framework developed here is sufficient to properly understand such challenging SPDEs, and demonstrates concrete use of the techniques surveyed.

\end{itemize}

%The first section of these notes deals with a classical construction of the It\^{o} stochastic integral, largely following the one dimensional case so well explicated in [\cite{karatzas1991brownian},\cite{Ox}]. We introduce the notion of a Cylindrical Brownian Motion which facilitates an infinite dimensional driving noise in our evolution equations, and the first section concludes by considering the integrals of Hilbert-Schmidt valued processes against a Cylindrical Brownian Motion.\\

%The second section begins to address this stochastic integration in the context of SPDEs, where we establish an abstract framework to define notions of solutions for equations in It\^{o} and Stratonovich form. We are keen to address Stratonovich equations for their physical significance, emphasised greatly in the seminal work [\cite{holm2015variational}] where the author deduces a new class of stochastic equations formulated with Stratonovich integration that serve as much improved fluid dynamics models. These models are highly technical, due to the nonlinearity persisting from the determinsitic form and the differential operator appearing in the diffusion term. Such models are therefore good motivation for us to establish a very general framework for SPDEs, which aren't covered by the standard monographs in the field. In this framework we make the conversion between Stratonovich and It\^{o} equations rigorous, and this framework is returned to for a very strong existence and uniqueness result in [\cite{goodair2023existence}]. 

\section{Stochastic Calculus in Infinite Dimensions} \label{section 1}

%In this section we present a `classical' construction of the It\^{o} stochastic integral, for processes evolving in a Hilbert Space. This occurs first for a one dimensional driving Brownian Motion, before generalisations to other one dimensional martingales and further to \textit{Cylindrical Brownian Motion}. Our construction is direct and designed to be familiar to a reader who has undertaken the finite dimensional study so well explicated in the likes of [\cite{karatzas1991brownian}, \cite{oksendal2013stochastic}]. In defining our infinite dimensional Brownian Motion, that is the Cylindrical Brownian Motion, we cover the fundamentals of martingale theory in Hilbert Spaces broadly by finite dimensional projections along with the real valued theory. The hope is again that this approach is accessible to a reader with background in the finite dimensional theory. The Cylindrical Brownian Motion is the only infinite dimensional driving process that we integrate with respect to; whilst we present a background on general $Q-$Cylindrical Processes which could be viable integrators, limiting ourselves to Cylindrical Brownian Motion enables the integral to be established as a straightforwards limit of the integrals against finite dimensional Brownian Motions. In particular we avoid the operator theoretic technicalities necessary in the general case, noting that our construction is sufficient for the framework and applications that follow.

\subsection{Elementary Notation}

Throughout these notes we work with a fixed filtered probability space $(\Omega, \mathcal{F}, (\mathcal{F}_t), \mathbbm{P})$, which is complete with respect to $\mathcal{F}_0$. We always consider Banach Spaces as measure spaces equipped with the Borel $\sigma$-algebra, and shall use $\lambda$ to represent the Lebesgue Measure. All of our Hilbert Spaces are assumed to be separable.

\begin{definition} \label{definition of spaces}
Let $(\mathcal{X},\mu)$ denote a general measure space, $(\mathcal{Y},\norm{\cdot}_{\mathcal{Y}})$ and $(\mathcal{Z},\norm{\cdot}_{\mathcal{Z}})$ be Banach Spaces, and $(\mathcal{U},\inner{\cdot}{\cdot}_{\mathcal{U}})$, $(\mathcal{H},\inner{\cdot}{\cdot}_{\mathcal{H}})$ be general Hilbert spaces. 
\begin{itemize}
    \item $L^p(\mathcal{X};\mathcal{Y})$ is the usual class of measurable $p-$integrable functions from $\mathcal{X}$ into $\mathcal{Y}$, $1 \leq p < \infty$, which is a Banach space with norm $$\norm{\phi}_{L^p(\mathcal{X};\mathcal{Y})}^p := \int_{\mathcal{X}}\norm{\phi(x)}^p_{\mathcal{Y}}\mu(dx).$$ The space $L^2(\mathcal{X}; \mathcal{Y})$ is a Hilbert Space when $\mathcal{Y}$ itself is Hilbert, with the standard inner product $$\inner{\phi}{\psi}_{L^2(\mathcal{X}; \mathcal{Y})} = \int_{\mathcal{X}}\inner{\phi(x)}{\psi(x)}_\mathcal{Y} \mu(dx).$$ %In the case $\mathcal{X} = \mathcal{O}$ and $\mathcal{Y} = \R^N$ note that $$\norm{\phi}_{L^2(\mathcal{O};\R^N)}^2 = \sum_{l=1}^N\norm{\phi^l}^2_{L^2(\mathcal{O};\R)}$$ for the component mappings $\phi^l: \mathcal{O} \rightarrow \R$. %We denote $\norm{\cdot}_{L^p(\mathcal{O};\R^N)}$ by $\norm{\cdot}_{L^p}$ and $\norm{\cdot}_{L^2(\mathcal{O};\R^N)}$ by $\norm{\cdot}$.
    
\item $L^{\infty}(\mathcal{X};\mathcal{Y})$ is the usual class of measurable functions from $\mathcal{X}$ into $\mathcal{Y}$ which are essentially bounded, which is a Banach Space when equipped with the norm $$ \norm{\phi}_{L^{\infty}(\mathcal{X};\mathcal{Y})} := \inf\{C \geq 0: \norm{\phi(x)}_Y \leq C \textnormal{ for $\mu$-$a.e.$ }  x \in \mathcal{X}\}.$$
    
    %\item $L^{\infty}(\mathcal{O};\R^N)$ is the usual class of measurable functions from $\mathcal{O}$ into $\R^N$ such that $\phi^l \in L^{\infty}(\mathcal{O};\R)$ for $l=1,\dots,N$, which is a Banach Space when equipped with the norm $$ \norm{\phi}_{L^{\infty}}:= \sup_{l \leq N}\norm{\phi^l}_{L^{\infty}(\mathcal{O};\R)}.$$
    
     \item $C(\mathcal{X};\mathcal{Y})$ is the space of continuous functions from $\mathcal{X}$ into $\mathcal{Y}$.

     \item $C_w(\mathcal{X};\mathcal{Y})$ is the space of `weakly continuous' functions from $\mathcal{X}$ into $\mathcal{Y}$, by which we mean continuous with respect to the given topology on $\mathcal{X}$ and the weak topology on $\mathcal{Y}$.

    \item $\mathscr{L}(\mathcal{Y};\mathcal{Z})$ is the space of bounded linear operators from $\mathcal{Y}$ to $\mathcal{Z}$. This is a Banach Space when equipped with the norm $$\norm{F}_{\mathscr{L}(\mathcal{Y};\mathcal{Z})} = \sup_{\norm{y}_{\mathcal{Y}}=1}\norm{Fy}_{\mathcal{Z}}.$$ It is the dual space $\mathcal{Y}^*$ when $\mathcal{Z}=\R$, with operator norm $\norm{\cdot}_{\mathcal{Y}^*}.$
    
     \item $\mathscr{L}^2(\mathcal{U};\mathcal{H})$ is the space of Hilbert-Schmidt operators from $\mathcal{U}$ to $\mathcal{H}$, defined as the elements $F \in \mathscr{L}(\mathcal{U};\mathcal{H})$ such that for some basis $(e_i)$ of $\mathcal{U}$, $$\sum_{i=1}^\infty \norm{Fe_i}_{\mathcal{H}}^2 < \infty.$$ This is a Hilbert space with inner product $$\inner{F}{G}_{\mathscr{L}^2(\mathcal{U};\mathcal{H})} = \sum_{i=1}^\infty \inner{Fe_i}{Ge_i}_{\mathcal{H}}$$ which is independent of the choice of basis (see e.g. [\cite{prevot2007concise}] Remark B.0.4). 

      \item For any $T>0$, $\mathscr{S}_T$ is the subspace of $C\left([0,T];[0,T]\right)$ of strictly increasing functions.

     \item For any $T>0$, $\mathcal{D}\left([0,T];\mathcal{Y}\right)$ is the space of c\'{a}dl\'{a}g functions from $[0,T]$ into $\mathcal{Y}$. It is a complete separable metric space when equipped with the metric $$d(\phi,\psi) := \inf_{\eta \in \mathscr{S}_T}\left[\sup_{t \in [0,T]}\left\vert \eta(t)- t\right\vert \vee \sup_{t \in [0,T]}\left\Vert \phi(t)-\psi(\eta(t)) \right\Vert_{\mathcal{Y}} \right]$$ which induces the so called Skorohod Topology (see [\cite{billingsley2013convergence}] pp124 for details).

    \item The total variation of a function $\phi:[0,T] \rightarrow \mathcal{Y}$, $V_{\mathcal{Y}}^T(\phi)$, is defined as
    $$V_{\mathcal{Y}}^T(\phi) := \sup_{I}\sum_{i=0}^{k-1}\norm{\phi(t_{i+1}) - \phi(t_i)}_{\mathcal{Y}} $$
    for the supremum taken over all partitions $I= \left\{0=t_0 < t_1 < \dots < t_{k}=T\right\}$. A function $\psi:[0, \infty) \rightarrow \mathcal{Y}$ is said to be of bounded-variation if $V_{\mathcal{Y}}^T(\psi) < \infty$ for all $T \geq 0$.

\end{itemize}

\end{definition}
	
	For the reader's convenience we also collect notation that is introduced in the paper, referencing where it is defined as relevant.
	
	\begin{itemize}
	\item $\lambda$ denotes the Lebesgue measure; 
	    \item $\mathbbm{1}_A$ denotes the indicator function of the set $A$;
	    \item $\mathcal{I}^{\mathcal{H}}_T$ is defined in \ref{viableintegrands};
	    \item $\mathcal{I}^{\mathcal{H}}$ is defined in \ref{I};
	    \item $\mathcal{M}^2, \mathcal{M}^2_c$ are defined in \ref{m2c};
	    \item $\mathcal{M}^2_c(\mathcal{H})$ is defined in \ref{def of m2ch};
	    \item $\bar{\mathcal{M}}^2_c, \bar{\mathcal{M}}^2_c(\mathcal{H})$ are the corresponding semi-martingale spaces;
	    \item $\mathcal{I}^T_M, \mathcal{I}^\mathcal{H}_M$ are defined in \ref{martingaleviableintegrands};
	    \item $\mathcal{I}^T_{\widetilde{M}}, \mathcal{I}^\mathcal{H}_{\widetilde{M}}$ are defined in \ref{localmartingaleintegrator};
	    \item $[ \cdot ]$ is defined in \ref{def of quad};
	    \item $[\cdot, \cdot]$ is defined in \ref{definition for cross variation inf dim};
	    \item $\mathcal{I}^\mathcal{H}_T(\mathcal{W})$ is defined in \ref{IHtw};
	    \item $\mathcal{I}^\mathcal{H}(\mathcal{W})$ is defined in \ref{ihw};
	    \item $\overbar{\mathcal{I}}^{\mathcal{H}}_T(\mathcal{W})$ is defined in \ref{overbar1};
	    \item $\overbar{\mathcal{I}}^{\mathcal{H}}(\mathcal{W})$ is defined in \ref{stochasticallyintegrableprocesses}.
	  
	\end{itemize}

	\subsection{A Classical Construction for Hilbert Space Valued Processes} \label{classicalconstruction}
	
	As alluded to, the construction precisely mirrors the standard one dimensional It\^{o} integral; as such we start from simple processes. $\mathcal{H}$ will denote a general Hilbert space, with norm and inner product $\norm{\cdot}_{\mathcal{H}}$ and $\inner{\cdot}{\cdot}_{\mathcal{H}}$ respectively. $W$ represents a standard one dimensional Brownian Motion with respect to the fixed filtered probability space.
	
	\begin{definition} \label{simpleprocess} Let $0 = t_0 < \dots < t_i < t_{i+1} < \dots $ be a time index such that $(t_i)$ approaches infinity. A simple $\mathcal{H}$ valued process is one that for $a.e. \ \omega$ takes the form $$\sy_t(\omega) = a_0(\omega)\mathbbm{1}_{\{0\}}(t)+\sum_{i=0}^{\infty} a_i(\omega)\mathbbm{1}_{(t_i,t_{i+1}]}(t)$$ where each $a_i \in L^2\big(\Omega;\mathcal{H}\big)$ and is $\mathcal{F}_{t_i}-$measurable, with respect to the Borel sigma algebra on $\mathcal{H}$. The limit is taken in $\mathcal{H}$.
	
	\end{definition}
	
	\begin{definition} \label{simpleprocessintegral}
	The It\^{o} integral of a simple $\mathcal{H}$ valued process $\sy$, with respect to Brownian motion, is defined as $$\int_0^t\sy_sdW_s := \sum_{i=0}^{\infty} a_i\big(W_{t_{i+1}\wedge t}-W_{t_i \wedge t}\big).$$

	\end{definition}
	
	Note that in reality the above is a finite sum, so there is no danger in how we take the limit. Indeed it can alternatively be expressed as
	
	\begin{equation} \label{simpleintegralsecond}
	    \sum_{i=0}^{k-1}a_i\big(W_{t_{i+1}}-W_{t_i}\big) + a_k\big(W_t-W_{t_k}\big)
	\end{equation}
 where $k$ is such that $t \in (t_k,t_{k+1}].$ Unsurprisingly now we define the integral for a more general class of integrands, using approximations by simple processes. 
	
	\begin{definition} \label{viableintegrands}
We use $\mathcal{I}_T^{\mathcal{H}}$ to denote the class of $\mathcal{H}$ valued processes $\sy$ which are progressively measurable\footnote{Here and throughout, progressive measurability is defined with respect to the fixed filtration $(\mathcal{F}_t)$ in the filtered probability space.} and satisfy the square integrability condition \begin{equation} \label{integrable} \mathbbm{E}\left(\int_0^T\norm{\sy_s}_{\mathcal{H}}^2ds\right) < \infty. \end{equation}
	In other words, $\sy \in L^2\big(\Omega\times [0,T]; \mathcal{H}\big)$ where the domain space $\Omega\times [0,T]$ is a measure space equipped with the product measure $\mathbbm{P} \times \lambda$.\footnote{The progressive measurability of $\sy$ ensures that it is measurable over this product space, and Tonelli's Theorem justifies exchanging the order of integration.}
	
	\end{definition}

	We have made the definition for \textit{progressively measurable} processes, not \textit{previsible} processes as will commonly be seen in the literature. Progressive measurability is a weaker condition than previsibility, but thankfully most reasonably behaved processes (adapted and left continuous for example) are both progressively measurable and previsible. We make the definition here for the more general class of integrands \textit{in the cases where the integrator is continuous}. Other authors may opt for previsible processes as these become necessary in retaining nice properties (e.g. martingality) when defining the stochastic integral with respect to discontinuous integrators, or even in making the definition itself.

	\begin{definition} \label{I}
	The class of processes $\sy$ such that $\sy \in \mathcal{I}_T^{\mathcal{H}}$ for all $T>0$ is denoted by $\mathcal{I}^{\mathcal{H}}$. 
	\end{definition}
	
	$\mathcal{I}^{\mathcal{H}}$ represents our class of integrands for all times, though there will be nothing wrong with defining the integral for times $t < T$ in the class $\mathcal{I}_T^{\mathcal{H}}$. The construction comes as a limit of simple integrals, for which we need the following proposition which holds no differently to [\cite{oksendal2013stochastic}] Lemma 3.1.5 for example.
	
	\begin{proposition} \label{simpleapproximation}
	For any $\sy \in \mathcal{I}_T^{\mathcal{H}}$, there exists a sequence of simple processes $(\sy^n)$ which converge to $\sy$ in $L^2\big(\Omega\times [0,T]; \mathcal{H}\big)$. 
	\end{proposition}

	\begin{definition} \label{stochintdefined}
	We define the It\^{o} stochastic integral for processes $\sy \in \mathcal{I}^{\mathcal{H}}$ by \begin{equation}\label{definestochint} \int_0^t \sy_s dW_s := \lim_{n \rightarrow \infty} \int_0^t\sy_s^ndW_s,\end{equation} where $(\sy^n)$ is the sequence of simple processes postulated in Proposition \ref{simpleapproximation} which approach $\sy$ in $L^2\big(\Omega\times [0,t]; \mathcal{H}\big),$ and the limit is taken in $L^2\big(\Omega;\mathcal{H}\big).$

	\end{definition}
	
	The fact that this is the natural topology in which to take the limit of simple stochastic integrals falls from the It\^{o} Isometry for simple processes, which further justifies that the construction is independent of the choice of simple approximation.
	
	\begin{proposition} \label{simpleisometryprop}
	    For a simple process $\sy^n$ and any time $t>0$, \begin{equation} \label{simpleisometry} \mathbbm{E}\Bigg(\Big\vert\Big\vert\int_0^t\sy^n_sdW_s\Big\vert\Big\vert_{\mathcal{H}}^2\Bigg) = \mathbbm{E}\Big(\int_0^t\norm{\sy_s^n}_{\mathcal{H}}^2ds\Big).\end{equation}
	\end{proposition}
	
	\begin{proof}
	 Let's suppose that $\sy^n$ takes the form \begin{equation} \label{standardsimple} \sy^n_t(\omega) = a^n_0(\omega)\mathbbm{1}_{\{0\}}(t) + \sum_{i=0}^{\infty} a_i^n(\omega)\mathbbm{1}_{(t^n_i,t^n_{i+1}]}(t)\end{equation} as outlined in Definition \ref{simpleprocess}. Then applying Definition \ref{simpleprocessintegral}, we deconstruct the LHS of (\ref{simpleisometry}):
	
	    \begin{align*}
	        \mathbbm{E}\Bigg(\Big\vert\Big\vert\int_0^t\sy^n_sdW_s\Big\vert\Big\vert_{\mathcal{H}}^2\Bigg) &= \mathbbm{E}\Bigg(\Big\vert\Big\vert\sum_{i=0}^{\infty} a^n_i\big(W_{t^n_{i+1}\wedge t}-W_{t^n_i \wedge t}\big)\Big\vert\Big\vert_{\mathcal{H}}^2\Bigg)\\
	        &= \mathbbm{E}\Bigg(\Big\langle\sum_{i=0}^{\infty} a^n_i\big(W_{t^n_{i+1}\wedge t}-W_{t^n_i \wedge t}\big), \sum_{j=0}^{\infty} a^n_j\big(W_{t^n_{j+1}\wedge t}-W_{t^n_j \wedge t}\big) \Big\rangle_{\mathcal{H}}\Bigg)\\
	        &= \sum_{i=0}^\infty \sum_{j=0}^\infty \mathbbm{E}\big(\inner{a^n_i}{a^n_j}_{\mathcal{H}}(W_{t^n_{i+1}\wedge t}-W_{t^n_i \wedge t})(W_{t^n_{j+1}\wedge t}-W_{t^n_j \wedge t})\big)
	    \end{align*}
recalling once more that the infinite sum is actually a finite sum (\ref{simpleintegralsecond}) so there is no difficulty in extracting it from the inner product and expectation. For $i \neq j$, and without loss of generality $i < j$, the random inner product is $\mathcal{F}_{t_j}-$measurable as the continuity of the inner product preserves measurability, and therefore $\inner{a^n_i}{a^n_j}_{\mathcal{H}}(W_{t^n_{i+1}\wedge t}-W_{t^n_i \wedge t})$ and $(W_{t^n_{j+1}\wedge t}-W_{t^n_j \wedge t})$ are independent from the independent increments of Brownian Motion. The terms thus vanish and we are left with $$\sum_{i=0}^\infty \mathbbm{E}\bigg(\norm{a^n_i}^2_{\mathcal{H}}(W_{t^n_{i+1}\wedge t}-W_{t^n_i \wedge t})^2\bigg)$$ 
	    to which we note independence again and assert that this is just $$\sum_{i=0}^\infty \mathbbm{E}\big(\norm{a^n_i}^2_{\mathcal{H}}\big)\big(t^n_{i+1}\wedge t-t^n_i \wedge t\big)$$
	    which is precisely the integral $$\int_0^t\sum_{i=0}^\infty \mathbbm{E}\big(\norm{a^n_i}_{\mathcal{H}}^2\big)\mathbbm{1}_{(t^n_i,t^n_{i+1}]}(s)ds.$$ We can write $$ \left\Vert a^n_i\right\Vert^2_{\mathcal{H}}\mathbbm{1}_{(t^n_i,t^n_{i+1}]}(s) = \left\Vert a^n_i\mathbbm{1}_{(t^n_i,t^n_{i+1}]}(s)\right\Vert^2_{\mathcal{H}},$$ the infinite sum of which is a single non-zero term, equal to $$ \left\Vert a^n_0(\omega)\mathbbm{1}_{\{0\}}(s)+\sum_{i=0}^{\infty} a^n_i(\omega)\mathbbm{1}_{(t^n_i,t^n_{i+1}]}(s)\right\Vert_{\mathcal{H}}^2$$ at every $s$ except for zero which is a set of Lebesgue measure zero in $[0,t]$. Again the infinite sum being only a single non-zero term justifies its exchange with expectation, and the above is of course $\norm{\sy^n_s}_{\mathcal{H}}^2$, justifying the result.

	    %or equivalently, in the form of (\ref{simpleintegralsecond}) as $$\sum_{i=0}^k \mathbbm{E}\bigg(\norm{a^n_i}^2_{\mathcal{H}}(W_{t^n_{i+1}}-W_{t^n_i})^2 + \norm{a^n_k}_{\mathcal{H}}^2(W_t-W_{t_k})^2\bigg)$$
	    
	\end{proof}
	
	So, why is this useful in terms of the limit in (\ref{definestochint})? First and foremost it ensures that the limit is uniquely defined; given that $L^2\big(\Omega;\mathcal{H}\big)$ is complete we need only show that the sequence of stochastic integrals is Cauchy in this space. The It\^{o} Isometry tells us that $\big(\int_0^t\sy^n_sdW_s\big)$ is Cauchy in $L^2\big(\Omega;\mathcal{H}\big)$ if and only if $(\sy^n)$ is Cauchy in $L^2\big(\Omega\times [0,t]; \mathcal{H}\big)$, which is of course true as by definition the $(\sy^n)$ are convergent (to $\sy$) in this space. Furthermore the Isometry extends to the general integral defined in Definition \ref{stochintdefined}, as a trivial corollary of the discussion here.
	
	\begin{corollary} \label{realItoIsom}
	The It\^{o} Isometry (\ref{simpleisometry}) holds for all processes $\sy \in \mathcal{I}^{\mathcal{H}}.$
	\end{corollary}

	%\textcolor{blue}{Further clarification about the class of processes which we make the definition for: we define the set $\mathcal{H}_t$ of $\mathcal{F}_s$ previsible processes (on $[0,t] \times \Xi$) such that \begin{equation} \label{integrable} \mathbbm{E}\Big(\int_0^t\norm{\sy_s(\theta)}_{L^2(\R^N;\R^N)}^2ds\Big) < \infty. \end{equation} Referring to my earlier question, we write this as $\sy \in L^2\big((\Xi,\mathbbm{P}) \times [0,t]; L^2(\R^N;\R^N)\big)$ where the domain space is a measure space with the product measure $\mathbbm{P} \times \lambda$, $\lambda$ is Lebesgue measure. Then for any $\sy \in \mathcal{H}_T$ we can find a sequence of simple processes converging to $\sy$ in $L^2\big((\Xi,\mathbbm{P}) \times [0,t]; L^2(\R^N;\R^N)\big)$.}\\
	
	 Without direct appeal to the formal construction, we may also understand the integral (\ref{stochint}) as a random element of the dual space $\mathcal{H}^*$ and identify the functional with its counterpart in $\mathcal{H}$ in the usual sense. 
	 
	 \begin{theorem} \label{dualityrep}
	 The It\^{o} stochastic integral defined in Definition \ref{stochintdefined} is the unique element of $\mathcal{H}$ satisfying the duality relation \begin{equation} \label{stochintduality} \Big\langle \int_0^t \sy_s dW_s, \phi\Big\rangle_{\mathcal{H}} =  \int_0^t\inner{\sy_s}{\phi}_{\mathcal{H}}dW_s \end{equation} for all $\phi \in \mathcal{H}$. The above are random inner products, defined by $$\inner{\sy_s}{\phi}_{\mathcal{H}}(\omega):= \inner{\sy_s(\omega)}{\phi}_{\mathcal{H}}$$ and similarly for the LHS. Therefore by the Riesz-Representation Theorem, it is consistent to define the It\^{o} stochastic integral as an $\mathcal{H}^*$ valued random variable via the mapping $$\phi \mapsto \bigg(\int_0^t\inner{\sy_s}{\phi}_{\mathcal{H}}dW_s\bigg)(\omega).$$
	 \end{theorem}
	 
	 \begin{proof}
	 Given that we have defined (\ref{stochint}) as a limit of simple processes, it will come as no surprise that we must use this approach to prove the relation (\ref{stochintduality}). We will demonstrate that this holds for simple processes $\sy^n$, and later that it is preserved in the $L^2\big(\Omega\times [0,t];\mathcal{H}\big)$ limit. Firstly though we ought to verify that the RHS of (\ref{stochintduality}) makes sense, that is to say $\inner{\sy}{\phi}_{\mathcal{H}}$ is a valid (one dimensional) integrand. Thus we must show the standard progressive measurability and square integrability conditions: for the former, note that the progressive measurability of $\sy$ is preserved under composition with the continuous mapping $\inner{\cdot}{\phi}_{\mathcal{H}}.$ The latter is straightforwards, as \begin{equation} \label{straightforwards} \mathbbm{E}\Big(\int_0^T\inner{\sy_s}{\phi}_{\mathcal{H}}^2ds\Big) \leq \mathbbm{E}\Big(\int_0^T\norm{\sy_s}_{\mathcal{H}}^2\norm{\phi}_{\mathcal{H}}^2ds\Big) = \norm{\phi}_{\mathcal{H}}^2 \mathbbm{E}\Big(\int_0^T\norm{\sy_s}_{\mathcal{H}}^2ds\Big)\end{equation} which is finite by (\ref{integrable}). Let's suppose that $\sy^n$ takes the form (\ref{standardsimple}). Then applying Definition \ref{simpleprocessintegral}, we deconstruct the LHS of (\ref{stochintduality}):
	 \begin{align*}
	     \Big\langle \int_0^t \sy^n_s dW_s, \phi\Big\rangle_{\mathcal{H}} &= \Big\langle\sum_{i=0}^{\infty} a^n_i\big(W_{t^n_{i+1}\wedge t}-W_{t^n_i \wedge t}\big), \phi\Big\rangle_{\mathcal{H}}\\
	     &= \sum_{i=0}^{\infty}\inner{a^n_i\big(W_{t^n_{i+1}\wedge t}-W_{t^n_i \wedge t}\big)}{\phi}_{\mathcal{H}}\\
	     &= \sum_{i=0}^{\infty}\inner{a^n_i}{\phi}_{\mathcal{H}}\big(W_{t^n_{i+1}\wedge t}-W_{t^n_i \wedge t}\big)
	 \end{align*}
	 and proceed similarly for the RHS, observing that the integrand
	 \begin{align*}
	     \inner{\sy_s}{\phi}_{\mathcal{H}} &= \Big\langle a^n_0\big(\mathbbm{1}_{\{0\}}(s)\big) + \sum_{i=0}^{\infty} a^n_i\big(\mathbbm{1}_{(t^n_i,t^n_{i+1}]}(s)\big), \phi\Big\rangle_{\mathcal{H}}\\
	     &= \inner{a^n_0\big(\mathbbm{1}_{\{0\}}(s)\big)}{\phi}_{\mathcal{H}} + \sum_{i=0}^{\infty}\inner{a^n_i\big(\mathbbm{1}_{(t^n_i,t^n_{i+1}]}(s)\big)}{\phi}_{\mathcal{H}}\\
	     &=  \inner{a^n_0}{\phi}_{\mathcal{H}}\big(\mathbbm{1}_{\{0\}}(s)\big) + \sum_{i=0}^{\infty}\inner{a^n_i}{\phi}_{\mathcal{H}}\big(\mathbbm{1}_{(t^n_i,t^n_{i+1}]}(s)\big)
	 \end{align*}
	 is again simple (this is completely analogous to showing that $\inner{\sy}{\phi}_{\mathcal{H}}$ was a valid integrand). Applying Definition \ref{simpleprocessintegral} to the above proves the result for simple processes, so all that remains to show is preservation in the limit. We have of course \begin{align*}
	     \Big\langle \int_0^t \sy_s dW_s, \phi\Big\rangle_{\mathcal{H}} &= \Big\langle \lim_{n \rightarrow \infty} \int_0^t \sy^n_s dW_s, \phi\Big\rangle_{\mathcal{H}}
	     \end{align*}
	     and a reminder that this limit is taken in $L^2\big(\Omega;\mathcal{H}\big).$ We would like to take the limit outside of the inner product, in some appropriate topology, and use the result for simple functions: the steps would be
	     \begin{align*}
	     \Big\langle \lim_{n \rightarrow \infty} \int_0^t \sy^n_s dW_s, \phi\Big\rangle_{\mathcal{H}} &= \lim_{n \rightarrow \infty}\Big\langle \int_0^t \sy^n_s dW_s, \phi\Big\rangle_{\mathcal{H}}\\
	     &= \lim_{n \rightarrow \infty}  \int_0^t\inner{\sy^n_s}{\phi}_{\mathcal{H}}dW_s
	 \end{align*}
so it should be clear that the topology we want to take this limit in is that of $L^2\big(\Omega;\R \big)$, as the last line would be precisely the RHS of (\ref{stochintduality}) by definition if we can show that the simple real valued process $\inner{\sy^n}{\phi}_{\mathcal{H}}$ converges to $\inner{\sy}{\phi}_{\mathcal{H}}$ in $L^2\big(\Omega\times [0,t]; \R\big).$ Thankfully it is straightforwards to justify taking this limit outside of the inner product: if $(f_n)$ converges to $f$ in $L^2\big(\Omega;\mathcal{H}\big)$ then $$\mathbbm{E}\Big(\inner{f_n}{\phi}_{\mathcal{H}}-\inner{f}{\phi}_{\mathcal{H}}\Big)^2 = \mathbbm{E}\Big(\inner{f_n-f}{\phi}_{\mathcal{H}}\Big)^2 \leq \mathbbm{E}\Big(\norm{f_n-f}_{\mathcal{H}}^2\norm{\phi}_{\mathcal{H}}^2\Big) = \norm{\phi}_{\mathcal{H}}^2 \mathbbm{E}\Big(\norm{f_n-f}_{\mathcal{H}}^2\Big) \longrightarrow 0$$
	 so $(\inner{f_n}{\phi}_{\mathcal{H}})$ converges to $\inner{f}{\phi}_{\mathcal{H}}$ in $L^2\big(\Omega;\R \big)$, as required to justify the interchange. To show the convergence of $\inner{\sy^n}{\phi}_{\mathcal{H}}$ to $\inner{\sy}{\phi}_{\mathcal{H}}$ in $L^2\big(\Omega\times [0,t]; \R\big)$ we apply the same trick:
	 \begin{align*}
	     \norm{\inner{\sy}{\phi}_{\mathcal{H}}-\inner{\sy^n}{\phi}_{\mathcal{H}}}_{L^2\big(\Omega\times [0,t]; \R\big)} &= \norm{\inner{\sy-\sy^n}{\phi}_{\mathcal{H}}}_{L^2\big(\Omega\times [0,t]; \R\big)}\\
	     &= \mathbbm{E}\Big(\int_0^t\inner{\sy_s-\sy^n_s}{\phi}_{\mathcal{H}}^2ds\Big)\\
	     &\leq \mathbbm{E}\Big(\int_0^t\norm{\sy_s-\sy^n_s}_{\mathcal{H}}^2\norm{\phi}_{\mathcal{H}}^2ds\Big)\\
	     &= \norm{\phi}_{\mathcal{H}}^2 \mathbbm{E}\Big(\int_0^t\norm{\sy_s-\sy^n_s}_{\mathcal{H}}^2ds\Big)\\
	     &=  \norm{\phi}_{\mathcal{H}}^2 \norm{\sy-\sy^n}_{L^2\big(\Omega\times [0,t]; \mathcal{H}\big)}\\
	     &\longrightarrow 0
	 \end{align*}
	 
	 where convergence to $0$ is by definition of the approximating sequence $\sy^n$.

	 \end{proof}
	 
	 We provide two applications of this result below, both of which will be fundamental to our SPDE framework.
	 
	 \begin{proposition} \label{multidimItoIsom}
	    The It\^{o} Isometry holds for a multi dimensional driving Brownian motion, in the sense that if $(\sy^i)_{i=1}^n$ are a collection of processes in $\mathcal{I}^\mathcal{H}$, and $(W^i)_{i=1}^n$ are independent Brownian Motions, then $$\mathbbm{E}\Bigg(\Big\vert\Big\vert \sum_{i=1}^n\int_0^t\sy^i_sdW^i
	    _s\Big\vert\Big\vert_{\mathcal{H}}^2\Bigg) = \sum_{i=1}^n\mathbbm{E}\Big(\int_0^t\norm{\sy_s^i}_{\mathcal{H}}^2ds\Big)$$ 
	\end{proposition}
	
	\begin{proof}
	    We look to simplify the left hand side of the required equality, swiftly applying Parseval's identity for a basis $(e_k)$ of $\mathcal{H}$:
	    
	    \begin{align*}
	        \mathbbm{E}\Bigg(\Big\vert\Big\vert \sum_{i=1}^n\int_0^t\sy^i_sdW^i
	    _s\Big\vert\Big\vert_{\mathcal{H}}^2\Bigg) &= \mathbbm{E}\sum_{k=1}^\infty \Big\langle \sum_{i=1}^n\int_0^t\sy^i_sdW^i_s,e_k\Big\rangle^2_\mathcal{H}\\ &= \mathbbm{E}\sum_{k=1}^\infty \Big( \sum_{i=1}^n \int_0^t\inner{\sy^i_s}{e_k}_{\mathcal{H}}dW^i_s\Big)^2
	    \end{align*}
having used linearity of the inner product to pull out the sum and Theorem \ref{dualityrep}. We can now regard the infinite sum as an integral with respect to the counting measure and apply Tonelli's Theorem, also expanding the square to obtain
	 
	 $$\sum_{k=1}^\infty \mathbbm{E} \Big( \sum_{i=1}^n \int_0^t\inner{\sy^i_s}{e_k}_{\mathcal{H}}dW^i_s\Big)^2 =  \sum_{k=1}^\infty \mathbbm{E} \sum_{i=1}^n\sum_{j=1}^n\Big( 
	 \int_0^t\inner{\sy^i_s}{e_k}_{\mathcal{H}}dW^i_s\Big)\Big( \int_0^t\inner{\sy^j_s}{e_k}_{\mathcal{H}}dW^j_s\Big).$$
  For the cross terms $i \neq j$ we make use of the independence of the Brownian Motions and hence the respective stochastic integrals, as well as the standard property that the It\^{o} integral has zero expectation to nullify these terms. Our expression reduces to $$\sum_{k=1}^\infty \sum_{i=1}^n \mathbbm{E} \Big( \int_0^t\inner{\sy^i_s}{e_k}_{\mathcal{H}}dW^i_s\Big)^2$$ to which we can apply the It\^{o} Isometry (Corollary \ref{realItoIsom} for the Hilbert Space $\R$, which is of course the standard Isometry) giving us $$\sum_{k=1}^\infty \sum_{i=1}^n \mathbbm{E} \int_0^t\inner{\sy^i_s}{e_k}^2_{\mathcal{H}}ds$$ from which we apply Tonelli twice more to take the infinite sum all the way through,
	 
	 $$\sum_{i=1}^n \mathbbm{E} \int_0^t\sum_{k=1}^\infty \inner{\sy^i_s}{e_k}^2_{\mathcal{H}}ds.$$ A final application of Parseval's identity gives the result. 
	 
	 \end{proof}
Whilst we chose to prove Proposition \ref{simpleisometryprop} and subsequently Corollary \ref{realItoIsom} from first principles in the Hilbert Space setting, the method of proof here touches upon a fundamental aspect of this theory: with a good understanding of the standard $\R$ valued setting, we can apply Theorem \ref{dualityrep} to straightforwardly deduce key properties here. Indeed if we accepted the It\^{o} Isometry in $\R$, we could have just proven Corollary \ref{realItoIsom} in the simple vein of Proposition \ref{multidimItoIsom}. We take this approach in extending the result of Theorem \ref{dualityrep}.
	 
	 \begin{theorem} \label{operator through integral 1D}
	     Suppose that $\mathcal{H}_1, \mathcal{H}_2$ are Hilbert spaces such that $\sy \in \mathcal{I}^{\mathcal{H}_1}$ and $T \in \mathscr{L}(\mathcal{H}_1;\mathcal{H}_2)$. Then the process $T\sy$ defined by $$T\sy_s(\omega) = T\big(\sy_s(\omega)\big)$$ belongs to $\mathcal{I}^{\mathcal{H}_2}$ and is such that \begin{equation} \label{Tthrough} T\Big(\int_0^t\sy_sdW_s\Big) = \int_0^tT\sy_sdW_s.\end{equation} 
	 \end{theorem}
	 
	 \begin{proof}
	     We shall prove first that $T\sy \in \mathcal{I}^{\mathcal{H}_2}$. The progressive measurability is preserved under the continuity of $T$, and for $C$ the (square of the) boundedness constant associated to $T$ we have that at any time $t$, $$\mathbbm{E}\Big(\int_0^t\norm{T\sy_s}_{\mathcal{H}_2}^2ds\Big) \leq C\mathbbm{E}\Big(\int_0^t\norm{\sy_s}_{\mathcal{H}_1}^2ds\Big) < \infty$$ as $\sy \in \mathcal{I}^{\mathcal{H}_1}$, showing that $T\sy \in \mathcal{I}^{\mathcal{H}_2}$. To commute $T$ with the integral we shall use the characterisation from Theorem \ref{dualityrep}, having now established that the right hand side of (\ref{Tthrough}) is well defined in $\mathcal{H}_2$. We introduce $T^* \in \mathscr{L}\left(\mathcal{H}_2;\mathcal{H}_1\right)$ as the adjoint of $T$, and observe that for any $\phi \in \mathcal{H}_2$,
	     \begin{align*}
	         \Big\langle T \Big(\int_0^t \sy_s dW_s\Big), \phi\Big\rangle_{\mathcal{H}_2} &= \Big\langle \int_0^t \sy_s dW_s, T^*\phi\Big\rangle_{\mathcal{H}_1}\\
	         &= \int_0^t \inner{\sy_s}{T^*\phi}_{\mathcal{H}_1}dW_s\\
	         &= \int_0^t \inner{T\sy_s}{\phi}_{\mathcal{H}_2}dW_s\\
	         &= \Big\langle \int_0^t T\sy_s dW_s, \phi\Big\rangle_{\mathcal{H}_2}
	     \end{align*}
applying Theorem \ref{dualityrep} twice. As this equality holds for arbitrary $\phi \in \mathcal{H}_2$ then we have proven (\ref{Tthrough}), which is of course an identity $\mathbbm{P}-a.e.$ in $\mathcal{H}_2$.  

	 \end{proof}

	 %Here we write $L^2(\R^N;\R^N)$ as simply $L^2$. For each time $t$, $\phi \in L^2$ and fixed $\theta \in \Xi$ the mapping $$\phi \mapsto \bigg(\int_0^t\inner{\phi}{\sy_s}_{L^2}dW_s\bigg)(\theta)$$ is bounded and linear on the Hilbert space $L^2(\R^N;\R^N)$. Note that $\inner{\phi}{\sy_s}_{L^2}$ is just the real valued random variable $\theta \mapsto \inner{\phi}{\sy_s(\theta)}_{L^2}.$ Linearity is clear, and boundedness can be seen from $$\int_0^t\inner{\phi}{\sy_s}_{L^2}dW_s \leq \int_0^t\norm{\phi}_{L^2}\norm{\sy_s}_{L^2}dW_s = \norm{\phi}_{L^2}\int_0^t\norm{\sy_s}_{L^2}dW_s$$ as $\phi$ is deterministic, and the stochastic integral is well-defined from the assumption $(\ref{integrable})$ so boundedness follows. By the Riesz representation theorem there is an element $\psi_{\theta} \in L^2$ such that for every $\phi \in L^2$, $$\inner{\phi}{\psi_{\theta}}_{L^2} = \bigg( \int_0^t\inner{\phi}{\sy_s}_{L^2}dW_s\bigg) (\theta)$$ so we define $(\ref{stochint})$, as a function of $\theta$, to be $\psi_{\theta}$.\\
	 
	 %\textcolor{blue}{We then want to discuss the infinite sum: the convergence is either $a.e.$ in $L^2(\R^N;\R^N)$ or just $L^2\big((\Xi,\mathbbm{P}); L^2(\R^N,\R^N)\big)$. What are our conditions to ensure the sums converge: look at Crisan/Lang the assumptions $(4)$ and then $(5)$. What do we mean by the last condition? What are the equivalent conditions in $N$ dimensions, and how do we prove this (maybe I can spend more time myself on this one).}

	 \subsection{Martingale and Local Martingale Integrators} \label{localchapter}
	 
	 As expected, we can extend the definition to integrators beyond Brownian motion, in the same manner as the standard It\^{o} integral. We begin the extension to continuous square integrable martingales, and then to continuous local martingales.  
	
	\begin{definition} \label{m2c}
	We shall denote the class of real valued martingales $M$ such that $M_t \in L^2\big(\Omega; \R\big)$ for every $t \geq 0$ by $\mathcal{M}^2$. The subclass of such martingales with $\mathbbm{P}-a.s.$ continuous paths will be represented by $\mathcal{M}^2_c$. 
	\end{definition}
	
	\begin{definition} \label{martingaleviableintegrands}
	For any $M \in \mathcal{M}^2_c$, define $\mathcal{I}_M^T$ to be the class of $\mathcal{H}$ valued processes $\sy$ which are progressively measurable on $[0,T] \times \Omega$ and satisfy the square integrability condition $$\mathbbm{E}\Big(\int_0^T\norm{\sy_s}_{\mathcal{H}}^2d[M]_s\Big) < \infty$$ where $\left[M\right]_{\cdot}$ is the quadratic variation of $M_{\cdot}$. We similarly define $\mathcal{I}^{\mathcal{H}}_M$ to be the class of processes $\sy$ such that $\sy\in \mathcal{I}_M^T$ for all $T>0$.
	\end{definition}

We do not put the space $\mathcal{H}$ explicitly into the time restricted notation for simplicity; once constructed, we will rarely need this notation, and when needed the space will be mentioned separately. Constructing the integral$$\int_0^t\sy_sdM_s$$ for $\sy\in \mathcal{I}^{\mathcal{H}}_M$ now falls from what we have already done for (\ref{stochint}). We use simple processes $\sy^n$ as in Definition \ref{simpleprocess} to approximate $\sy$, in the sense that $$\lim_{n \rightarrow 0}\mathbbm{E}\Big(\int_0^T\norm{\sy_s-\sy_s^n}_{\mathcal{H}}^2d[M]_s\Big)=0.$$ Simply replacing $W$ by $M$ in Definitions \ref{simpleprocessintegral} and \ref{stochintdefined} completes the construction, though we do not give the details here. Let's now move on to the more delicate matter of integration with respect to a local martingale. This begins again with notation for our set of integrands.

\begin{definition} \label{localmartingaleintegrator}

For a continuous local martingale $\widetilde{M}$, define $\mathcal{I}^T_{\widetilde{M}}$ to be the class of progressively measurable processes $\sy$ such that \begin{equation} \label{localmartingaleintegrability}\int_0^T\norm{\sy_s}_{\mathcal{H}}^2d[\widetilde{M}]_s < \infty \quad a.e.\end{equation} Also define $\mathcal{I}^{\mathcal{H}}_{\widetilde{M}}$ to be those $\sy$ in $\mathcal{I}^T_{\widetilde{M}}$ for every $T$.

\end{definition}

Suppose that $\widetilde{M}$ is localised by the stopping times $(T_n).$ Without loss of generality this sequence of stopping times can be chosen such that the stopped processes $\widetilde{M}^{T_n}$ defined by $$\widetilde{M}^{T_n}_t:=\widetilde{M}_{t \wedge T_n}$$ are bounded; if $(T_n')$ are localising stopping times, then we can simply set $$T_n = T_n' \wedge \inf\{0 \leq t < \infty : \abs{\widetilde{M}_t} \geq n\}$$ so that for each $n$, $\widetilde{M}^{T_n}$ is a bounded continuous martingale and hence in $\mathcal{M}_c^2$. Note of course that the new stopping times $(T_n)$ are still non-decreasing and approach infinity $\mathbbm{P}-a.s.$ by the pathwise continuity of $\widetilde{M}$. Continuing in this theme, for a process $\sy \in \mathcal{I}^{\mathcal{H}}_{\widetilde{M}}$ let's define some more non-decreasing random times $(R_n)$ by \begin{equation} \label{R_n}
    R_n:= n \wedge \inf\{0 \leq t < \infty: \int_0^t\norm{\sy_s}_{\mathcal{H}}^2d[\widetilde{M}]_s \geq n\}\end{equation} taking the convention that the infimum of the empty set is infinite. The $(R_n)$ are stopping times as they are simply first hitting times of the continuous and adapted random variable $$\int_0^t\norm{\sy_s}_{\mathcal{H}}^2d[\widetilde{M}]_s.$$ Again these times tend to infinity $\mathbbm{P}-a.s.$ by condition (\ref{localmartingaleintegrability}). Now define $\tau_n$ by $$\tau_n = R_n \wedge T_n$$ and the truncated processes $\sy^n$ as $$\sy^n_t:= \sy_t\mathbbm{1}_{t \leq \tau_n}.$$ We use the fact that for $m \leq n$, and $t \leq \tau_m$, we have $$\sy_t\mathbbm{1}_{t \leq \tau_n}=\sy_t\mathbbm{1}_{t \leq \tau_m}$$ and also that $$\widetilde{M}^{\tau_n}_t = \widetilde{M}^{\tau_m}_t$$ so we can make the consistent definition that \begin{equation} \label{localintegratordefinition}\Big(\int_0^t\sy_sd\widetilde{M}_s\Big)(\omega) := \Big(\int_0^t\sy^n_sd\widetilde{M}^{\tau_n}_s\Big)(\omega)\end{equation} at almost every $\omega$ for any $n$ such that $t \leq \tau_n(\omega)$, noting that such an $n$ exists (for almost every $\omega$). There is subtlety in this, as $n$ itself is a random variable (it is dependent on $\omega$) and the logic in which we are proceeding is vital. To be clear, we are \textit{not} considering the $n$ of (\ref{localintegratordefinition}) as a random variable; the random selection of $n$ occurs prior. That is to say we understand (\ref{localintegratordefinition}) as a definition of the left hand side pointwise $\mathbbm{P}-a.e.$ by fixing such an $\omega$, then fixing our $n$ as outlined, and then considering the right hand side as a process which can be evaluated at any $\omega' \in \Omega$, but is such that if $\omega' \neq \omega$ then it will not necessarily be true that $t \leq \tau_n(\omega')$. We simply evaluate this process at $\omega$ to make the definition. Of course to do this we require that at this choice of $n$, $\sy^n \in \mathcal{I}^{\mathcal{H}}_{\widetilde{M}^{\tau_n}}$: the process $(\mathbbm{1}_{t \leq \tau_n})$ is progressively measurable, as it is both left continuous and adapted (adaptedness becomes clear when for each fixed $t$, we write the random variable $\mathbbm{1}_{t \leq \tau_n}$ as $1 - \mathbbm{1}_{t > \tau_n}$). The square integrability in \ref{martingaleviableintegrands} comes from the fact that the random variable $$\int_0^t\norm{\sy^n_s}_{\mathcal{H}}^2d[\widetilde{M}^{\tau_n}]_s$$ is bounded by $n$ $\mathbbm{P}-a.s.$ (owing to (\ref{R_n})), hence the expectation satisfies the same bound. It is critical again here that the $n$ in (\ref{localintegratordefinition}) is not allowed to be random, as we would have instead a bound $$\bigg(\int_0^t\norm{\sy^n_s}_{\mathcal{H}}^2d[\widetilde{M}^{\tau_n}]_s\bigg)(\omega) \leq n(\omega)$$ so we cannot deduce a finite expectation as required because the bound is not uniform in $\omega$. Of course where $\widetilde{M}$ is itself a genuine martingale, this procedure defines the stochastic integral for processes with only the regularity (\ref{localmartingaleintegrability}). In this case we do not have to stop the integrator, just truncate the integrand.

\begin{definition} \label{localmartingalewrtBM}
In the special case where the continuous local martingale is given by the genuine martingale of Brownian Motion, we denote $\mathcal{I}^{\mathcal{H}}_W$ by simply $\overbar{\mathcal{I}}^{\mathcal{H}}$. This class of processes differs to $\mathcal{I}^{\mathcal{H}}$ because we only assume a bound almost everywhere, not in expectation.
\end{definition}

We extend properties of the stochastic integral to this class of processes. 

\begin{proposition} \label{bounded take it in }
    Let $\sy \in \bar{\mathcal{I}}^{\mathcal{H}}$ and $\phi \in L^\infty(\Omega;\mathcal{H})$ be $\mathcal{F}_0-$measurable. Then $\inner{\sy}{\phi} \in \bar{\mathcal{I}}^{\mathcal{H}}$ and for every $t>0$ we have that
    \begin{equation} \label{reffyyy}
        \left\langle \int_0^t\sy_r dW_r, \phi \right\rangle_{\mathcal{H}} = \int_0^t\left\langle \sy_r,\phi \right\rangle_{\mathcal{H}}dW_r
    \end{equation}
    $\mathbbm{P}-a.s.$. The above are random inner products defined by $$\inner{\sy_s}{\phi}_{\mathcal{H}}(\omega):=\inner{\sy_s(\omega)}{\phi(\omega)}_{\mathcal{H}}$$ and similarly for the left hand side. 
\end{proposition}

\begin{proof}
We should first justify that $\left\langle \sy_\cdot,\phi \right\rangle_{\mathcal{H}} \in \bar{\mathcal{I}}^{\R}$. The progressive measurability follows as for every $T>0$ the mapping $$\inner{\sy_{\cdot}}{\phi}_{\mathcal{H}}: t \times \omega \times \tilde{\omega} \rightarrow \inner{\sy_{t}(\omega)}{\phi(\tilde{\omega})}_{\mathcal{H}}$$ is $\mathcal{B}([0,T]) \times \mathcal{F}_T \times \mathcal{F}_0$ measurable, so in particular it is $\mathcal{B}([0,T]) \times \mathcal{F}_T \times \mathcal{F}_T$ measurable and as such $$\inner{\sy_{\cdot}}{\phi}_{\mathcal{H}}: t \times \omega  \rightarrow \inner{\sy_{t}(\omega)}{\phi(\omega)}_{\mathcal{H}}$$ is $\mathcal{B}([0,T]) \times \mathcal{F}_T$ measurable as required. Note that we have used the progressive measurability requirement on $\sy$. We also appreciate that for $\mathbbm{P}-a.e.$ $\omega$, $$\int_0^T\inner{\sy_r(\omega)}{\phi(\omega)}_{\mathcal{H}}^2dr \leq \norm{\phi(\omega)}_{\mathcal{H}}^2\int_0^T\norm{\sy_r(\omega)}_{\mathcal{H}}^2dr < \infty$$ again by assumption on $\sy \in \bar{\mathcal{I}}^{\mathcal{H}}$. Thus $\inner{\sy}{\phi}_{\mathcal{H}} \in \bar{\mathcal{I}}^{\R}$. To compute the integrals we introduce the stopping times $$\tau_j := j \wedge \inf\left\{0 \leq t < \infty: (1 + \norm{\phi}_{\mathcal{H}}^2)\int_0^t\norm{\sy_r}_{\mathcal{H}}^2dr \geq j \right\} $$ such that for every $j \in \N$, $\sy_{\cdot}\mathbbm{1}_{\cdot \leq \tau_j} \in \mathcal{I}^{\mathcal{H}}$,  $\inner{\sy_{\cdot}\mathbbm{1}_{\cdot \leq \tau_j}}{\phi}_{\mathcal{H}} \in \mathcal{I}^{\R}$. For $\omega$ fixed and $t$ fixed as in (\ref{reffyyy}) then we choose $n$ sufficiently large so that $\tau_j(\omega) \geq t$. The integrals are then defined at this $\omega$ by
\begin{equation} \label{reffyyy2}
        \left\langle \int_0^t\sy_r\mathbbm{1}_{r \leq \tau^j} dW_r, \phi \right\rangle_{\mathcal{H}} = \int_0^t\left\langle \sy_r\mathbbm{1}_{r \leq \tau^j},\phi \right\rangle_{\mathcal{H}}dW_r
    \end{equation}
so we in fact show that (\ref{reffyyy2}) holds $\mathbbm{P}-a.e.$ for every $j$. We now fix arbitrary $j \in \N$. Our plan is as follows: we consider a sequence of simple processes $(\py^n)$ which approximate $\sy_{\cdot}\mathbbm{1}_{\cdot \leq \tau_j}$ in $L^2(\Omega \times [0,t];\mathcal{H})$ as postulated in Proposition \ref{simpleapproximation}. We then claim that $(\inner{\py^n}{\phi}_{\mathcal{H}})$ is a sequence of $\R$ valued simple processes which converge to $\inner{\sy_{\cdot}\mathbbm{1}_{\cdot \leq \tau_j}}{\phi}_{\mathcal{H}}$ in $L^2(\Omega \times [0,t];\R)$. Following this we shall prove (\ref{reffyyy}) for this simple case and show the identity holds in the limit.\\

We first show that for each $n \in \N$, $\inner{\py^n}{\phi}_{\mathcal{H}}$ is a simple process. Let $\py^n$ have representation as in Definition \ref{simpleprocess}. Then
\begin{align*}
    \inner{\py^n}{\phi}_{\mathcal{H}} &= \left\langle a^n_0\mathbbm{1}_{\{0\}} + \sum_{i=0}^{\infty} a_i^n\mathbbm{1}_{(t^n_i,t^n_{i+1}]},\phi \right\rangle_{\mathcal{H}}\\
    &= \inner{a^n_0}{\phi}_{\mathcal{H}} \mathbbm{1}_{\{0\}} + \sum_{i=0}^{\infty} \left\langle a_i^n,\phi \right\rangle_{\mathcal{H}}\mathbbm{1}_{(t^n_i,t^n_{i+1}]}
\end{align*}
so this would satisfy the requirements of an $\R$ valued simple process if for each $i \in \N$, $\inner{a^n_i}{\phi}_{\mathcal{H}} \in L^2(\Omega;\R)$ and is $\mathcal{F}_{t_i}-$measurable. For the square integrability constraint, observe that $$\mathbbm{E}\left(\inner{a^n_i}{\phi}_{\mathcal{H}}^2 \right) \leq \mathbbm{E}\left(\norm{a^n_i}_{\mathcal{H}}^2\norm{\phi}_{\mathcal{H}}^2\right) \leq \norm{\phi}_{L^\infty(\Omega;\mathcal{H})}^2\mathbbm{E}\left(\norm{a^n_i}_{\mathcal{H}}^2\right) < \infty$$ by the assumptions of $a_i \in L^2(\Omega;\mathcal{H})$ and $\phi \in L^\infty(\Omega;\mathcal{H})$. The $\mathcal{F}_{t_i}$ measurability follows in the same way as the progressive measurability of $\inner{\sy}{\phi}_{\mathcal{H}}$. Indeed the required $L^2\left(\Omega\times [0,t]; \R\right)$ convergence follows similarly as \begin{align*}
\left\Vert \inner{\sy_{\cdot}\mathbbm{1}_{\cdot \leq \tau_j}}{\phi}_{\mathcal{H}} - \inner{\py^n_{\cdot}}{\phi}_{\mathcal{H}}\right\Vert_{L^2(\Omega \times [0,t];\R)} &= \norm{ \inner{\sy_{\cdot}\mathbbm{1}_{\cdot \leq \tau_j} - \py^n}{\phi}_{\mathcal{H}}}_{L^2(\Omega \times [0,t];\R)}\\ &\leq \left\Vert \norm{\sy_{\cdot}\mathbbm{1}_{\cdot \leq \tau_j} - \py^n}_{\mathcal{H}}\norm{\phi}_{\mathcal{H}}\right\Vert_{L^2(\Omega \times [0,t];\R)} \\&\leq \norm{\phi}_{L^\infty(\Omega;\mathcal{H})}  \left\Vert \norm{\sy_{\cdot}\mathbbm{1}_{\cdot \leq \tau_j} - \py^n}_{\mathcal{H}}\right\Vert_{L^2(\Omega \times [0,t];\R)}
\end{align*}
and by assumption $$\left\Vert \norm{\sy_{\cdot}\mathbbm{1}_{\cdot \leq \tau_j} - \py^n}_{\mathcal{H}}\right\Vert_{L^2(\Omega \times [0,t];\R)} \longrightarrow 0$$ as $n \rightarrow \infty$, so the convergence is proved. To show the identity (\ref{reffyyy}) in the case of the simple process $\py^n$, observe that
\begin{align*}
    \left\langle \int_0^t \py^n_r dW_r, \phi\right\rangle_{\mathcal{H}} &= \left\langle\sum_{i=0}^{\infty} a^n_i\big(W_{t^n_{i+1}\wedge t}-W_{t^n_i \wedge t}\big), \phi\right\rangle_{\mathcal{H}}\\
	     &= \sum_{i=0}^{\infty}\left\langle a^n_i\left(W_{t^n_{i+1}\wedge t}-W_{t^n_i \wedge t}\right),\phi\right\rangle_{\mathcal{H}}\\
	     &= \sum_{i=0}^{\infty}\inner{a^n_i}{\phi}_{\mathcal{H}}\left(W_{t^n_{i+1}\wedge t}-W_{t^n_i \wedge t}\right)\\
	     &= \int_0^t\inner{\py^n_r}{\phi}_{\mathcal{H}}dW_r
\end{align*}
as required. In order to conclude the argument, by definition of the integral we have that \begin{align*}\int_0^t\sy_r\mathbbm{1}_{r \leq \tau^j} dW_r &= \lim_{n \rightarrow \infty}\int_0^t\py^n_r dW_r\\
\int_0^t\inner{\sy_r\mathbbm{1}_{r \leq \tau^j}}{\phi}_{\mathcal{H}} dW_r &= \lim_{n \rightarrow \infty}\int_0^t\inner{\py^n_r}{\phi}_{\mathcal{H}} dW_r
\end{align*}
where the first limit is taken in $L^2(\Omega;\mathcal{H})$ and the second one in $L^2(\Omega;\R)$. For each we can thus extract a $\mathbbm{P}-a.s.$ convergent subsequence in the appropriate space, so by taking successive subsequences we can find one common subsequence indexed by $(n_k)$ such that the above limits hold $\mathbbm{P}-a.s.$. Thus
\begin{align*}
    \left \langle \int_0^t\sy_r\mathbbm{1}_{r \leq \tau^j} dW_r, \phi \right\rangle_{\mathcal{H}} &= \left \langle \lim_{n_k \rightarrow \infty}\int_0^t\py^{n_k}_r dW_r, \phi \right \rangle_{\mathcal{H}}\\
    &=\lim_{n_k \rightarrow \infty} \left \langle \int_0^t\py^{n_k}_r dW_r, \phi \right \rangle_{\mathcal{H}}\\
    &= \lim_{n_k \rightarrow \infty}\int_0^t\inner{\py^n_r}{\phi}_{\mathcal{H}}dW_r\\
    &= \int_0^t\inner{\sy_r\mathbbm{1}_{r \leq \tau^j}}{\phi}_{\mathcal{H}} dW_r
\end{align*}
so (\ref{reffyyy2}) is justified and the proof is complete.

\end{proof}

We note that the $\mathcal{F}_0-$measurability requirement on $\phi$ really comes into play in showing the $\mathcal{F}_{t_i}-$measurability of $\inner{a^n_i}{\phi}_\mathcal{H}$. If we were to consider the integral over some $[s,t]$ interval instead, then one could relax $\phi$ to only being $\mathcal{F}_s-$measurable. In fact the result can also be extended to unbounded $\phi$. To do this we shall prove a Stochastic Dominated Convergence Theorem. 

\begin{lemma} \label{stochstic dominated convergence}
    Let $(\sy^n)$ be a sequence in $\bar{\mathcal{I}}^{\mathcal{H}}$ such that there exists processes $\sy:\Omega \times [0,\infty) \rightarrow \mathcal{H}$ and $\py \in \bar{\mathcal{I}}^{\mathcal{H}}$ with the properties that for every $T>0$, $ \mathbbm{P} \times \lambda  - a.e.$ $(\omega,t) \in \Omega \times [0,T]$:
    \begin{enumerate}
        \item \label{enum1} $\norm{\sy^n_t(\omega)}_{\mathcal{H}} \leq \norm{\py_t(\omega)}_{\mathcal{H}}$ for all $n \in \N$;
        \item $(\sy^n_t(\omega))$ is convergent to $\sy_t(\omega)$ in $\mathcal{H}$.
    \end{enumerate}
    Then $\sy \in \bar{\mathcal{I}}^{\mathcal{H}}$ and for every $t>0$, there exists a subsequence indexed by $(n_k)$ such that \begin{equation}\label{endgame}
    \lim_{n_k \rightarrow \infty}\int_0^t\sy^{n_k}_rdW_r = \int_0^t\sy_r dW_r\end{equation} $\mathbbm{P}-a.s.$. 
\end{lemma}

\begin{proof}
Immediately we note that $\sy$ inherits the progressive measurablity from $\sy^n$ from the almost everywhere limit in the product space $\Omega \times [0,T]$ when equipped with product sigma algebra $\mathcal{F}_T \times \mathcal{B}([0,T])$. Similarly we must have that for $ \mathbbm{P} \times \lambda  - a.e.$ $(\omega,t)$, $\norm{\sy_t(\omega)}_{\mathcal{H}} \leq \norm{\py_t(\omega)}_{\mathcal{H}}$ so $\sy$ must too satisfy the integrability constraints and hence belongs to $\bar{\mathcal{I}}^{\mathcal{H}}$. We look to find a common sequence of localising times for the stochastic integrals, and then demonstrate (\ref{endgame}) by the showing the identity holds true when stopped at each localising time. To this end we introduce the stopping times
$$\tau_j := j \wedge \inf\left\{0 \leq t < \infty: \int_0^t\norm{\py_r}_{\mathcal{H}}^2dr \geq j \right\}$$
which from item \ref{enum1} serve as a sequence of localising times for every $\sy^n$, and too for $\sy$. Thus for any fixed $t >0$ and $j \in \N$ we wish to show that
\begin{equation} \label{newest show that}
    \lim_{n_k \rightarrow \infty}\int_0^t\sy^{n_k}_r\mathbbm{1}_{r \leq \tau_j}dW_r = \int_0^t\sy_r\mathbbm{1}_{r \leq \tau_j} dW_r
\end{equation}
for a subsequence $(n_k)$ $\mathbbm{P}-a.s.$, or equivalently that
$$ \lim_{n_k \rightarrow \infty}\int_0^t(\sy^{n_k}_r-\sy_r)\mathbbm{1}_{r \leq \tau_j}dW_r =0.$$
We first assess the convergence in $L^2(\Omega;\mathcal{H})$, applying Corollary \ref{realItoIsom} for each fixed $n$ to see that
$$\mathbbm{E}\left\Vert \int_0^t(\sy^{n}_r-\sy_r)\mathbbm{1}_{r \leq \tau_j}dW_r  \right\Vert_{\mathcal{H}}^2 = \mathbbm{E}\left( \int_0^t\norm{(\sy^{n}_r-\sy_r)\mathbbm{1}_{r \leq \tau_j}}_{\mathcal{H}}^2 dr  \right). $$
Observing that for $\mathbbm{P} \times \lambda-a.e.$ $(\omega,t)$, \begin{align*} \norm{(\sy^{n}_r(\omega)-\sy_r(\omega))\mathbbm{1}_{r \leq \tau_j}(\omega)}_{\mathcal{H}}^2 &\leq \left(\norm{\sy^{n}_r(\omega)\mathbbm{1}_{r \leq \tau_j(\omega)}}_{\mathcal{H}} + \norm{\sy_r(\omega)\mathbbm{1}_{r \leq \tau_j}(\omega)}_{\mathcal{H}}\right)^2\\
&\leq 4\norm{\py_r(\omega)\mathbbm{1}_{r \leq \tau_j}(\omega)}_{\mathcal{H}}^2
\end{align*}
Then with dominating function $4\norm{\py_\cdot\mathbbm{1}_{\cdot \leq \tau_j}}_{\mathcal{H}}^2$ we can apply the standard Dominated Convergence Theorem for the integral over the product space (we face no problems with the order and configuration of integration from Tonelli's Theorem given the progressive measurability) to deduce that $$ \lim_{n \rightarrow \infty}\mathbbm{E}\left( \int_0^t\norm{(\sy^{n}_r-\sy_r)\mathbbm{1}_{r \leq \tau_j}}_{\mathcal{H}}^2 dr  \right) = 0$$ and therefore $$\lim_{n \rightarrow \infty}\mathbbm{E}\left\Vert \int_0^t(\sy^{n}_r-\sy_r)\mathbbm{1}_{r \leq \tau_j}dW_r  \right\Vert_{\mathcal{H}}^2 = 0 .$$
Thus we have demonstrated the convergence (\ref{newest show that}) but for the whole sequence in $L^2(\Omega;\mathcal{H})$, from which we can deduce a $\mathbbm{P}-a.s.$ convergent subsequence and the result is proved.

\end{proof}

\begin{proposition} \label{unbounded take it in}
    Let $\sy \in \bar{\mathcal{I}}^{\mathcal{H}}$ and $\phi: \Omega \rightarrow \mathcal{H}$ be $\mathcal{F}_0-$measurable. Then $\inner{\sy}{\phi} \in \bar{\mathcal{I}}^{\mathcal{H}}$ and for every $t>0$ we have that
    \begin{equation} \label{reffyyy3}
        \left\langle \int_0^t\sy_r dW_r, \phi \right\rangle_{\mathcal{H}} = \int_0^t\left\langle \sy_r,\phi \right\rangle_{\mathcal{H}}dW_r
    \end{equation}
    $\mathbbm{P}-a.s.$. 
\end{proposition}

\begin{proof}
A justification that $\left\langle \sy_r,\phi \right\rangle _{\mathcal{H}} \in \bar{\mathcal{I}}^{\R}$ is precisely as in Proposition \ref{bounded take it in }. To apply this result, we rewrite $\phi$ in a trivial way as $$\phi:= \sum_{k=1}^\infty \phi \mathbbm{1}_{k \leq \norm{\phi}_{\mathcal{H}} < k+1}$$ where the limit is taken $\mathbbm{P}-a.s.$ in $\mathcal{H}$ (similarly to (\ref{simpleintegralsecond}) this is just a finite sum at each fixed $\omega$, or more precisely just a single element of the sum). Introducing the notation $$\phi^n:= \sum_{k=1}^n \phi \mathbbm{1}_{k \leq \norm{\phi}_{\mathcal{H}} < k+1} $$ then clearly $\phi^n \in L^\infty(\Omega;\mathcal{H})$ and is still $\mathcal{F}_0-$measurable, so we can apply Proposition \ref{bounded take it in } to see that $$\left\langle \int_0^t\sy_r dW_r, \phi^n \right\rangle_{\mathcal{H}} = \int_0^t\left\langle \sy_r,\phi^n \right\rangle_{\mathcal{H}}dW_r. $$
We can take the $\mathbbm{P}-a.s.$ limit in $\mathcal{H}$ outside of the inner product on the left hand side, so it is sufficient to show that 
\begin{equation} \label{new sufficient to show}\lim_{n \rightarrow \infty}\int_0^t\left\langle \sy_r,\phi^n \right\rangle_{\mathcal{H}}dW_r = \int_0^t\left\langle \sy_r,\phi \right\rangle_{\mathcal{H}}dW_r\end{equation}
or at least that this is true for a subsequence. This is an immediate application of Lemma \ref{stochstic dominated convergence}, with dominating function simply the limit $\inner{\sy}{\phi}_{\mathcal{H}}$.

%To make the computation we introduce the stopping times
%$$\tau_j := j \wedge \inf\left\{0 \leq t < \infty: 4\norm{\phi}_{\mathcal{H}}^2\int_0^t\norm{\sy_r}_{\mathcal{H}}^2dr \geq j \right\} $$
%with the understanding that $\norm{\phi^n}_{\mathcal{H}} \leq \norm{\phi}_{\mathcal{H}}$ $\mathbbm{P}-a.s.$.  We verify (\ref{new sufficient to show}) by showing that for every $j \in \N$,
%$$\lim_{n \rightarrow \infty}\int_0^t\left\langle \sy_r,\phi -\phi^n  \right\rangle_{\mathcal{H}}\mathbbm{1}_{r \leq \tau_j}dW_r=0 $$
%where the motivation behind the definition of $\tau_j$ was that \begin{align*}\int_0^t\left\langle \sy_r,\phi -\phi^n  \right\rangle_{\mathcal{H}}^2dr &\leq \int_0^t\norm{\sy_r}_{\mathcal{H}}^2\norm{\phi -\phi^n}_{\mathcal{H}}^2dr \\&\leq \int_0^t\norm{\sy_r}_{\mathcal{H}}^2\left(\norm{\phi}_{\mathcal{H}} +\norm{\phi^n}_{\mathcal{H}}\right)^2dr\\ &\leq 4\norm{\phi}_{\mathcal{H}}^2\int_0^t\norm{\sy_r}_{\mathcal{H}}^2 dr. \end{align*}
%We first assess the convergence in $L^2(\Omega;\R)$. 

\end{proof}

The same is true for multiplication by real valued random variables, where the proof is identical. We state the result here.

\begin{proposition} \label{real valued bounded take it in}
    Let $\sy \in \bar{\mathcal{I}}^{\mathcal{H}}$ and $\eta: \Omega \rightarrow \R$ be $\mathcal{F}_0-$measurable. Then $\eta \sy \in \bar{\mathcal{I}}^{\mathcal{H}}$ and for every $t > 0$ we have that $$\eta \int_0^t\sy_rdW_r = \int_0^t \eta \sy_r dW_r$$ $\mathbbm{P}-a.s.$. 
\end{proposition}

 \subsection{Cylindrical Processes} \label{sub cylindrical processes}
    
Having now addressed the question of how to integrate a Hilbert Space valued process with respect to a finite dimensional driving martingale, we look to extend this theory to the case of an infinite dimensional driving martingale. The aforementioned construction then arises from the one dimensional projections of the driving process. Our notion of infinite dimensional Brownian Motion is a Cylindrical Brownian Motion, which we explore in this subsection. Additional operator theory is used in this subsection to better understand the background of the process, however it is not necessary in the resulting integral constructed in Subsection \ref{subsection inf dim integral}. We will denote by $Y$ a real-valued zero-mean Gaussian process with correlation function $R(t,s)$, $\mathcal{H}$ a Hilbert space and $Q$ a bounded positive self-adjoint operator on $\mathcal{H}.$

    \begin{definition} \label{cylindricalprocess} A \textit{$Q$-Cylindrical $Y$ process} over $\mathcal{H}$ is a process $X^Q$ taking values in the space of functions from $\mathcal{H}$ to $\R$, that is $$X^Q: \R_+ \times \Omega \rightarrow \R^{\mathcal{H}}$$ such that for each $h \in \mathcal{H}$, $X^Q_h: \R_+ \times \Omega \rightarrow \R$ is a process of zero mean Gaussian random variables, and for every $g, h \in \mathcal{H}$ and all times $t,s$, $$\mathbbm{E}\big(X^Q_g(t)X^Q_h(s)\big) = \inner{Qg}{h}_{\mathcal{H}}R(t,s).$$

\end{definition}

\begin{definition} 
A $Q$-Cylindrical Brownian Motion is defined as above where the $Y$ process is a Brownian Motion. A Cylindrical Brownian Motion is a $Q$-Cylindrical Brownian Motion such that $Q$ is the identity operator. 
\end{definition}

To make this abstract definition workable, we look to how such processes can be represented. This motivates the notion of a regular process.

\begin{definition}
A $Q$-Cylindrical $Y$ process $X^Q$ is said to be regular if there exists a square integrable $\mathcal{H}$ valued process $\widetilde{X}^Q$ (that is, $\widetilde{X}^Q_t \in L^2\big(\Omega; \mathcal{H}\big)$ for all $t \geq 0$) such that for every $h \in \mathcal{H}$, $X^Q_h$ has the same law as the process $\inner{ \widetilde{X}^Q}{h}_{\mathcal{H}}$ defined by $$\inner{ \widetilde{X}^Q}{h}_{\mathcal{H}}(t,\omega):=\inner{ \widetilde{X}^Q(t, \omega)}{h}_{\mathcal{H}}. $$
\end{definition}

 We will not hesitate to identify the functional valued process $X^Q$ with the Hilbert space one $\widetilde{X}^Q.$ In the next theorem however, we keep the distinction for clarity.

\begin{theorem}
A $Q$-Cylindrical $Y$ process $X^Q$ is regular if and only if $Q$ is trace-class. In this case $X^Q$ admits the regular representation \begin{equation} \label{regularrep} \widetilde{X}^Q(t) = \sum_{i=1}^\infty \sqrt{\lambda_i}e_iY^i_t \end{equation} where the $(e_i)$ are an orthonormal basis of $\mathcal{H}$ consisting of eigenvectors of the self-adjoint trace-class (hence compact) operator $Q$, $(\lambda_i)$ are the corresponding eigenvalues, and $(Y^i)$ are independent copies of the process $Y$. The limit is taken in $L^2\big(\Omega;\mathcal{H}\big).$
\end{theorem}

\begin{proof} We consider the two directions. In the first one we show that if $Q$ is trace-class then (\ref{regularrep}) is a regular representation of $X^Q$. In the second we show that if $X^Q$ is regular then $Q$ is trace-class, so by the first implication, (\ref{regularrep}) is again a regular representation of $X^Q$.
\begin{enumerate}
    \item[$\impliedby$:]
    A sensible place to start would be to verify that for trace-class $Q$, (\ref{regularrep}) does indeed define an element of $L^2\big(\Omega;\mathcal{H}\big).$ We will rely on completeness of the space and show that the sequence of partial sums is Cauchy. Observe that $$\mathbbm{E}\big(\norm{\sum_{i=m}^n\sqrt{\lambda_i}e_iY^i_t}^2_{\mathcal{H}}\big) = \mathbbm{E}\big(\sum_{i=m}^n \abs{\lambda_i}\abs{Y^i_t}^2\big) = \mathbbm{E}\big(\abs{Y_t}^2\big) \sum_{i=m}^n \lambda_i.$$ All we then require to conclude the Cauchy property is that $\sum_{i=1}^{\infty} \lambda_i < \infty$ which it is given that $Q$ is trace class. To conclude that (\ref{regularrep}) is a regular representation of $X^Q$, it is sufficient to show that the one dimensional processes $\inner{ \widetilde{X}^Q}{h}_{\mathcal{H}}$ satisfy the conditions postulated of the $X^Q_h$ in Definition \ref{cylindricalprocess} as these characterise the distribution. %\textcolor{cyan}{In what sense will this qualify $\widetilde{X}^Q$ to be a regular representation of $X^Q$. What this will do is show that there is a $Q$ Cylindrical $Y$ process $Z^Q$ such that $\widetilde{X}^Q$ is a regular representation of $Z^Q$. In what sense are $Z^Q$ and $X^Q$ equivalent, and do we simply identify them? In particular, we can deduce that $X^Q_h(t)$ and $Z^Q_h(t)$ have the same law for every $h,t$, and we can deduce a little more from $$\mathbbm{E}\big(X^Q_g(t)X^Q_h(s)\big) = \mathbbm{E}\big(Z^Q_g(t)Z^Q_h(s)\big)$$}\\
    
    First of all for each $h \in \mathcal{H}$ and time $t$, we must verify that the random variable \begin{equation} \label{rv} \omega \mapsto \Big\langle\sum_{i=1}^\infty \sqrt{\lambda_i}e_iY^i_t(\omega),h\Big\rangle_{\mathcal{H}}\end{equation} is zero mean Gaussian. Note that $$\Big\langle\lim_{n \rightarrow \infty} \sum_{i=1}^n \sqrt{\lambda_i}e_iY^i_t,h\Big\rangle_{\mathcal{H}} = \lim_{n \rightarrow \infty} \Big\langle \sum_{i=1}^n \sqrt{\lambda_i}e_iY^i_t,h\Big\rangle_{\mathcal{H}}$$ where the second limit is in $L^2\big(\Omega;\R\big)$ just as we did in Theorem \ref{dualityrep}. The random variable (\ref{rv}) is thus an $L^2$ limit of zero mean Gaussian random variables, so is itself zero mean Gaussian (convergence in $L^2$ implies that in distribution so we have Gaussianity, and $L^2$ convergence implies $L^1$ from which we readily deduce the zero mean property). It remains to show that for each $g,h \in \mathcal{H}$ and $s,t >0$ that
     $$\mathbbm{E}\bigg(\Big\langle\sum_{i=1}^\infty \sqrt{\lambda_i}e_iY^i_t,g\Big\rangle_{\mathcal{H}}\Big\langle\sum_{j=1}^\infty \sqrt{\lambda_j}e_iY^j_s,h\Big\rangle_{\mathcal{H}}\bigg) = \inner{Qg}{h}_{\mathcal{H}}R(t,s).$$ We take the limit through the first inner product on the LHS as above, that is
    \begin{equation} \label{ddd}\mathbbm{E}\Bigg(\bigg(\lim_{n \rightarrow \infty}\Big\langle\sum_{i=1}^n \sqrt{\lambda_i}e_iY^i_t,g\Big\rangle_{\mathcal{H}}\bigg)\Big\langle\sum_{j=1}^\infty \sqrt{\lambda_j}e_iY^j_s,h\Big\rangle_{\mathcal{H}}\Bigg)\end{equation}
    and argue that for a sequence of functions $(f_n)$ convergent to $f$ in $L^2\big(\Omega;\R\big)$, and $a \in $$L^2\big(\Omega;\R\big)$, that $$(\lim_{L^2}f_n)(a) = \lim_{L^1}(f_na).$$ The right side is well defined as the limit of a Cauchy sequence: $$\mathbbm{E}\big(\abs{f_na-f_ma}\big) = \mathbbm{E}\big(\abs{(f_n - f_m)a}\big) \leq \mathbbm{E}\big(\abs{f_n-f_m}^2\big)\mathbbm{E}\big(\abs{a}^2\big)$$ and a similar calculation shows that this element of $L^1\big(\Omega;\R\big)$ is the left side. Applying to (\ref{ddd}) and pulling the $L^1$ limit through the expectation produces
    $$\lim_{n \rightarrow \infty}\mathbbm{E}\bigg(\Big\langle\sum_{i=1}^n \sqrt{\lambda_i}e_iY^i_t,g\Big\rangle_{\mathcal{H}}\Big\langle\sum_{j=1}^\infty \sqrt{\lambda_j}e_jY^j_s,h\Big\rangle_{\mathcal{H}}\bigg)$$ and playing the same game, this is
    $$\lim_{n \rightarrow \infty}\lim_{m \rightarrow \infty} \mathbbm{E}\bigg(\Big\langle\sum_{i=1}^n \sqrt{\lambda_i}e_iY^i_t,g\Big\rangle_{\mathcal{H}}\Big\langle\sum_{j=1}^m \sqrt{\lambda_j}e_jY^j_s,h\Big\rangle_{\mathcal{H}}\bigg)$$ and further
    $$\lim_{n \rightarrow \infty}\lim_{m \rightarrow \infty} \mathbbm{E}\bigg(\sum_{i=1}^n\sum_{j=1}^m\big\langle\sqrt{\lambda_i}e_i,g\big\rangle_{\mathcal{H}}\big\langle\sqrt{\lambda_j}e_j,h\big\rangle_{\mathcal{H}}Y^i_tY^j_s\bigg).$$ Independence of the copies $(Y^i)$ and definition of the correlation function gives that this is equal to
     $$\lim_{n \rightarrow \infty} \sum_{i=1}^n\big\langle\sqrt{\lambda_i}e_i,g\big\rangle_{\mathcal{H}}\big\langle\sqrt{\lambda_i}e_i,h\big\rangle_{\mathcal{H}}R(t,s) =\lim_{n \rightarrow \infty} \sum_{i=1}^n\big\langle \lambda_ie_i,g\big\rangle_{\mathcal{H}}\big\langle e_i,h\big\rangle_{\mathcal{H}}R(t,s)$$ which is just $$\Big\langle\sum_{i=1}^\infty\lambda_i\inner{e_i}{g}e_i,h\Big\rangle_{\mathcal{H}}R(t,s) = \inner{Qg}{h}_{\mathcal{H}}R(t,s)$$ as required.
     
     \item[$\implies$:] For the reverse direction, assume that $X^Q$ is regular with corresponding process $\widetilde{X}^Q$ (which is not necessarily of the form (\ref{regularrep})). We want an expression in terms of the trace of $Q$, so we exploit Definition \ref{cylindricalprocess}: $$\sum_{i=1}^\infty \mathbbm{E}\big(\abs{X^Q_{e_i}(t)}^2\big) = \sum_{i=1}^\infty \inner{Qe_i}{e_i}_{\mathcal{H}}R(t,t)$$ and use an alternative expression from the assumed regular representation: $$\sum_{i=1}^\infty \mathbbm{E}\big(\abs{X^Q_{e_i}(t)}^2\big) = \sum_{i=1}^\infty \mathbbm{E}\big(\inner{\widetilde{X}^Q(t)}{e_i}_{\mathcal{H}}^2\big) = \mathbbm{E}\Big( \sum_{i=1}^\infty \inner{\widetilde{X}^Q(t)}{e_i}_{\mathcal{H}}^2\Big) =  \mathbbm{E}\big(\norm{\widetilde{X}^Q(t)}_{\mathcal{H}}^2\big) < \infty$$ where the infinite sum is pulled inside the expectation from the monotone convergence theorem, and finiteness is by assumption on $\widetilde{X}^Q  \in L^2\big(\Omega; \mathcal{H}\big).$ Hence $Q$ is trace-class, and by the first implication, (\ref{regularrep}) is a regular representation of $X^Q$.
     
    \end{enumerate}
    
\end{proof}

A standard Cylindrical Brownian Motion, that is where $Q$ is the identity, is thus not regular. However it would be convenient to have such a representation for Cylindrical Brownian Motion (denote this $\mathcal{W}_t$); we would like this to be something along the lines of $$\mathcal{W}_t = \sum_{i=1}^{\infty}e_iW^i_t$$ where the $(e_i)$ form an orthonormal basis of $\mathcal{H}$, and $(W^i)$ are standard independent Brownian Motions. We can in fact explicitly construct a larger Hilbert space $\mathcal{H}'$ such that the inclusion mapping $J: \mathcal{H} \xhookrightarrow{}\mathcal{H}'$ is Hilbert-Schmidt. The composition $Q:=JJ^*$ is then trace-class on $\mathcal{H}'$, and indeed $\mathcal{W}_t$ is a $Q$-Cylindrical Brownian Motion on $\mathcal{H}$': we defer the details to e.g. [\cite{lototsky2017stochastic}] Problem 3.2.6. To be precise, that is $J(e_i) = \sqrt{\lambda_i}\eta_i$ for each $i$, where the $(\eta_i)$ form an orthonormal basis (of $\mathcal{H}'$) of eigenfunctions of $JJ^*$ with eigenvalues $\lambda_i$, and for any $h \in \mathcal{H}$, $$\mathcal{W}_h(t) = \inner{\sum_{i=1}^\infty J(e_i)W^i_t}{J(h)}_{\mathcal{H}'}.$$

\subsection{Cylindrical and Hilbert Space Valued Martingales}

As we look to make cylindrical processes the driving force in our stochastic integral, it will come as no surprise that we introduce martingality in this setting.

\begin{definition}
A $Q$-Cylindrical $Y$ process over $\mathcal{H}$, $X^Q$,  is said to be a martingale if for every $h \in \mathcal{H}$, the real valued process $X^Q_h$ is a martingale. 
\end{definition}

Whilst the choice is natural, there seems no way to extend the definition to include two key aspects of our martingale integrators for a one dimensional driving process, namely continuity and square integrability. As such we consider the analogy in Hilbert spaces, covering at least the cylindrical processes which have a regular representation, with an understanding that this would most likely be the starting point for developing the integration theory with respect to an infinite dimensional martingale.

\begin{definition} \label{def of m2ch}
A process $M$ taking values in a Hilbert space $\mathcal{H}$ is said to be a martingale if for every $h \in \mathcal{H}$, the process $\inner{M}{h}_{\mathcal{H}}$ is a real valued martingale. The martingale is said to be continuous if for $\mathbbm{P}-a.e.$ $\omega$ and every $T>0$, $M(\omega):[0,T] \rightarrow \mathcal{H}$ is continuous. The martingale is said to be square integrable if for every $t \geq  0$, $\mathbbm{E}\big(\norm{M_t}_{\mathcal{H}}^2\big) < \infty$. The class of continuous square integrable martingales will be denoted $\mathcal{M}^2_c(\mathcal{H}).$\footnote{Recall that $\mathcal{M}^2_c := \mathcal{M}^2_c(\R)$, see Definition \ref{m2c}.}
\end{definition}

We look to show that $\mathcal{M}^2_c(\mathcal{H})$ is closed in a relevant topology, for which we shall use a sufficient condition for conitnuity proven now.

\begin{lemma} \label{weak to strong continuity}
    Let $\sy \in C_w\left([0,T];\mathcal{H}\right)$ and $\norm{\sy}_{\mathcal{H}}^2 \in C\left([0,T];\R\right)$. Then $\sy \in C\left([0,T];\mathcal{H}\right)$.
\end{lemma}

\begin{proof}
    We fix some $t \in [0,T]$ and look to show that $\sy$ is continuous at $t$. To this end consider arbitrary $s \in [0,T]$. Then
    \begin{align*}
        \norm{\sy_t-\sy_s}_{\mathcal{H}}^2 &= \inner{\sy_t-\sy_s}{\sy_t-\sy_s}_{\mathcal{H}}\\ &= \inner{\sy_t-\sy_s}{\sy_t}_{\mathcal{H}} - \inner{\sy_t-\sy_s}{\sy_s}_{\mathcal{H}}\\ &= \inner{\sy_t-\sy_s}{\sy_t}_{\mathcal{H}} + \norm{\sy_s}_{\mathcal{H}}^2 - \inner{\sy_t}{\sy_s}_{\mathcal{H}}\\
        &= \inner{\sy_t-\sy_s}{\sy_t}_{\mathcal{H}} + \norm{\sy_s}_{\mathcal{H}}^2 + \inner{\sy_t}{\sy_t - \sy_s}_{\mathcal{H}} - \norm{\sy_t}_{\mathcal{H}}^2\\
        &= 2\inner{\sy_t-\sy_s}{\sy_t}_{\mathcal{H}} + \left(\norm{\sy_s}_{\mathcal{H}}^2 - \norm{\sy_t}_{\mathcal{H}}^2\right).
    \end{align*}
    For any given $\varepsilon > 0$, there exists a $\delta > 0$ such that for all $s \in [0 \vee (t- \delta), (t + \delta) \wedge T]$,
    \begin{align*}
        \inner{\sy_t-\sy_s}{\sy_t}_{\mathcal{H}} &< \frac{\varepsilon}{3},\\
        \norm{\sy_s}_{\mathcal{H}}^2 - \norm{\sy_t}_{\mathcal{H}}^2 &< \frac{\varepsilon}{3}
    \end{align*}
    from each of the two assumptions, which gives the result.
\end{proof}

\begin{proposition} \label{prop m2c closed}
    Let $(M^n)$ be a sequence in $\mathcal{M}^2_c(\mathcal{H})$ and $M$ a process with values in $\mathcal{H}$ such that at every time $t \geq 0$, $$\lim_{n \rightarrow \infty}\mathbbm{E}\left( \norm{M^n_t - M_t}_{\mathcal{H}}^2\right) = 0.$$
    Then $M \in \mathcal{M}^2_c(\mathcal{H})$.
\end{proposition}

\begin{proof}
    The square integrability is trivial, so we just consider the martingale property and continuity. For every $\phi \in \mathcal{H}$, at every $t \geq 0$, we have that \begin{align*} \mathbbm{E}\left( \left\vert \inner{M^n_t}{\phi}_{\mathcal{H}} - \inner{M_t}{\phi}_{\mathcal{H}} \right\vert \right) = \mathbbm{E}\left( \left\vert \inner{M^n_t- M_t}{\phi}_{\mathcal{H}} \right\vert \right) &\leq \norm{\phi}_{\mathcal{H}}\mathbbm{E}\left(\norm{M^n_t-M_t}_{\mathcal{H}}\right)\\ &\leq \norm{\phi}_{\mathcal{H}}\left(\mathbbm{E}\left(\norm{M^n_t-M_t}_{\mathcal{H}}^2\right)\right)^\frac{1}{2}\end{align*}
which converges to zero as $n \rightarrow \infty$. In particular, $(\inner{M^n_t}{\phi}_{\mathcal{H}})$ converges to $\inner{M_t}{\phi}_{\mathcal{H}}$ in $L^1\left(\Omega;\R\right)$, and as the former is a sequence of real-valued martingales by definition, then it is standard that $\inner{M}{\phi}_{\mathcal{H}}$ is again a martingale. As $\phi$ was arbitrary, we deduce the martingality of $M$. It remains to show pathwise continuity. For this we introduce the finite dimensional projections $(\mathcal{P}_k)$, defined by $$\mathcal{P}_k: \phi \mapsto \sum_{i=1}^k\inner{\phi}{e_i}_{\mathcal{H}}e_i$$ for $(e_i)$ a fixed orthonormal basis of $\mathcal{H}$. The goal is to first show that $\mathcal{P}_kM$ has continuous paths for each $k$; as $\mathcal{P}_k$ is an orthogonal projection in $\mathcal{H}$, it is clear that at every $t \geq 0$,
$$ \lim_{n \rightarrow \infty}\mathbbm{E}\left(\norm{\mathcal{P}_k\left(M^n_t - M_t\right)}_{\mathcal{H}}^2\right) = 0$$ and in particular $$\lim_{n \rightarrow \infty}\mathbbm{E}\left(\sum_{i=1}^k \inner{M^n_t - M_t}{e_i}_{\mathcal{H}}^2\right) = 0.$$ We look to deduce some convergence for a real valued martingale and use the developed theory in this area, for which we note that \begin{align*}
   \lim_{n \rightarrow \infty}\mathbbm{E}\left( \left\vert\sum_{i=1}^k \inner{M^n_t}{e_i} - \sum_{i=1}^k \inner{M_t}{e_i}_{\mathcal{H}}\right\vert^2\right) \leq k\lim_{n \rightarrow \infty}\mathbbm{E}\left(\sum_{i=1}^k \inner{M^n_t - M_t}{e_i}_{\mathcal{H}}^2\right) = 0
\end{align*}
so the sequence in $n$, $\left(\sum_{i=1}^k \inner{M^n}{e_i} \right)$, of real valued martingales is convergent in $L^2\left(\Omega;\R\right)$ at every $t \geq 0$. It is known, for example [\cite{karatzas1991brownian}] Proposition 1.5.23, that $\mathcal{M}^2_c$ is closed in this topology hence $\sum_{i=1}^k \inner{M}{e_i}$ is continuous. As this is true for every $k \in \N$, we may take differences to see that $\inner{M}{e_i}$ is continuous for every $i$, and therefore $\mathcal{P}_kM$ is continuous in $\mathcal{H}$ for each $k$ (of course, still pathwise $\mathbbm{P}-a.s.$). In fact we take the same approach as the referenced proposition, working directly with the sequence of submartingales $(\sum_{i=1}^k \inner{M}{e_i}_{\mathcal{H}}^2)$ indexed by $k$. For any $\varepsilon > 0$, by Doob's Inequality ([\cite{karatzas1991brownian}] Theorem 1.3.8), we have that
\begin{align} \label{may as well label}
    \mathbbm{P}\left(\left\{ \omega \in \Omega: \sup_{t \in [0,T]}\sum_{i=k+1}^{j}\inner{M_t(\omega)}{e_i}_{\mathcal{H}}^2 > \varepsilon \right\}\right) \leq \frac{1}{\varepsilon}\mathbbm{E}\left( \sum_{i=k+1}^{j}\inner{M_T}{e_i}_{\mathcal{H}}^2 \right)
\end{align}
but we have already established that $\mathbbm{E}\left(\norm{M_T}_{\mathcal{H}}^2 \right) < \infty$, or equivalently $\mathbbm{E}\left(\sum_{i=1}^{\infty}\inner{M_t(\omega)}{e_i}_{\mathcal{H}}^2 \right) < \infty$. Therefore in the limit as $k \rightarrow \infty$, (\ref{may as well label}) approaches zero uniformly in $j > k$. One may choose a subsequence such that for all $k_l < k_m$, $$\mathbbm{P}\left(\left\{ \omega \in \Omega: \sup_{t \in [0,T]}\sum_{i=k_l+1}^{k_m}\inner{M_t(\omega)}{e_i}_{\mathcal{H}}^2 > \frac{1}{l} \right\}\right) \leq \frac{1}{2^l}$$ so by the Borel Cantelli Lemma the subsequence $\sum_{i=1}^{k_l}\inner{M}{e_i}_{\mathcal{H}}^2$ is Cauchy in $C\left([0,T];\R\right)$ $\mathbbm{P}-a.s.$. It thus admits a limit $\mathbbm{P}-a.s.$ in $C\left([0,T];\R\right)$, which agrees with the limit at each $t$, given by $\sum_{i=1}^\infty \inner{M_t}{e_i}_{\mathcal{H}}^2$ which is of course $\norm{M_t}_{\mathcal{H}}^2$. Therefore, the process $\norm{M}_{\mathcal{H}}^2$ is pathwise continuous; from Lemma \ref{weak to strong continuity}, it is now sufficient to just show weak continuity. This is clear however, as for any given $\phi \in \mathcal{H}$ and $t \geq 0$, we have again that
$$\lim_{n \rightarrow \infty}\mathbbm{E}\left( \inner{M^n_t - M_t}{\phi}_{\mathcal{H}}^2\right) = 0$$ so $\inner{M}{\phi}_{\mathcal{H}}$ is shown to belong to $\mathcal{M}^2_c$ exactly as was done for $\sum_{i=1}^k\inner{M}{e_i}_{\mathcal{H}}$, thus concluding the proof.

\end{proof}

\begin{proposition} \label{1Dstochintismartingale}
    For $\sy \in \mathcal{I}^{\mathcal{H}}$, the It\^{o} stochastic integral $$\int_0^\cdot\sy_s dW_s$$ belongs to $\mathcal{M}^2_c(\mathcal{H})$.
\end{proposition}

\begin{proof}
Recalling Definition \ref{stochintdefined}, the integral is defined at each time $t$ as a limit in $L^2\left(\Omega;\mathcal{H}\right)$ of simple integrals. From Proposition \ref{prop m2c closed} it is sufficient to show that these approximating integrals all belong to $\mathcal{M}^2_c(\mathcal{H})$, which is clear referring to the representation (\ref{simpleintegralsecond}) and the definition of $\mathcal{M}^2_c(\mathcal{H})$.

%The martingality follows immediately from Theorem \ref{dualityrep} and the standard one dimensional result. Similarly the square integrability was proved in Corollary \ref{realItoIsom}. For the continuity we simply defer to the finite dimensional construction as this works in the same manner, see [\cite{karatzas1991brownian}] for details. We note that this is a \textit{version} of the process, as the integral is itself only defined as a limit in $L^2\left(\Omega \times [0,T]; \mathcal{H}\right)$ and thus in a $\mathbbm{P} \times \lambda - a.e.$ equal equivalence class. 
\end{proof}

This result extends to the case of a general martingale integrator completely synonymous with the finite dimensional setting. Local martingality is then defined as we would expect, and we have the following result.

\begin{proposition} \label{integralislocalmartingale}
    For a continuous local martingale $\widetilde{M}$ and $\sy \in \mathcal{I}^{\mathcal{H}}_{\widetilde{M}}$, the It\^{o} stochastic integral $$\int_0^t \sy_s d\widetilde{M}_s$$ is itself a continuous local martingale.
\end{proposition}

\begin{proof}
    We claim that the localising stopping times are given simply by the $(\tau_n)$ used in the definition of the integral. We have already seen that these tend to infinity almost surely, so it just remains to show that at any fixed $n \in \N$ the stopped process is a continuous martingale. We fix such $n$ and look at how the stopped process $$\int_0^{t\wedge \tau_n} \sy_s d\widetilde{M}_s$$ is actually defined. Well, at each fixed $\omega$ we need to choose an $m$ such that $t \wedge \tau_n(\omega) \leq \tau_m(\omega)$, so we can simply choose $m$ to be $n$. Indeed this can be done uniformly across all $\omega$, so the stopped process is genuinely the integral $$\int_0^t\sy^n_sd\widetilde{M}^{\tau_n}_s$$ for this fixed $n$, which is a continuous martingale simply as in Proposition \ref{1Dstochintismartingale}.
\end{proof}

The next ingredient would be a definition of quadratic variation, which we look to do via a Doob-Meyer decomposition for $M \in \mathcal{M}^2_c(\mathcal{H})$. That is, can we show that $\norm{M}^2_{\mathcal{H}}$ defines a sub-martingale? Well we have integrability by definition, and $$\norm{M_t}^2_{\mathcal{H}} = \sum_{i=1}^\infty\inner{M_t}{e_i}_{\mathcal{H}}^2$$ where the limit is defined $\mathbbm{P}-a.e.$. Again by definition the above projections are martingales and so the squares are sub-martingales. The process $\norm{M}^2_{\mathcal{H}}$ is adapted as each $\norm{M_t}^2_{\mathcal{H}}$ is the $a.e.$ limit of $\mathcal{F}_t-$measurable random variables (on the complete measure space), and it is a sub-martingale as we can apply the Monotone Convergence Theorem to take the limit through the expectation for the defining sub-martingale property. As such, we have the following.

\begin{definition} \label{def of quad}
    For $M \in \mathcal{M}^2_c(\mathcal{H})$, the quadratic variation $[ M ]$ of $M$ is defined to be the unique\footnote{Uniqueness here is `up to indistinguishability', which is to say there exists a set $A \in \mathcal{F}$ with $\mathbbm{P}(A) = 1$ such that at all $t \geq 0$ and all $\omega \in A$, the processes are equal.} continuous, adapted, non-decreasing process with $[M]_0 = 0$ ($\mathbbm{P}-a.s.$) specified in the Doob-Meyer decomposition such that $$\norm{M}^2_{\mathcal{H}} - [M]$$ is a real valued martingale.
\end{definition}

\begin{proposition} \label{quad of int}
    Suppose that $\sy \in \mathcal{I}^{\mathcal{H}}$, then \begin{equation}\label{crawrt}\left[\int_0^{\cdot}\sy_r dW_r \right]_t = \int_0^t\norm{\sy_r}_H^2dr.\end{equation}
\end{proposition}

\begin{proof}
The fact that the process (\ref{crawrt}) is continuous, adapted, non-decreasing and starting from zero is clear. It simply remains to show the required martingality. To this end observe that at each time $t$, 
\begin{align*}\left\Vert\int_0^{t}\sy_r dW_r\right\Vert_{\mathcal{H}}^2 - \int_0^t\norm{\sy_r}_H^2dr &= \sum_{i=1}^\infty \left\langle\int_0^{t}\sy_r dW_r, e_i \right\rangle^2_{\mathcal{H}} - \int_0^t\sum_{i=1}^\infty \inner{\sy_r}{e_i}_H^2dr\\
&= \sum_{i=1}^\infty\left(\int_0^t\inner{\sy_r}{e_i}_{\mathcal{H}}dW_r \right)^2 - \sum_{i=1}^\infty \int_0^t\inner{\sy_r}{e_i}_H^2dr\\
&= \sum_{i=1}^\infty\left(\left(\int_0^t\inner{\sy_r}{e_i}_{\mathcal{H}}dW_r \right)^2 -  \int_0^t\inner{\sy_r}{e_i}_H^2dr \right)
\end{align*}
having applied Theorem \ref{dualityrep} to the first term and the Monotone Convergence Theorem to the second term, where the infinite sum is a limit taken $\mathbbm{P}-a.s.$. From the standard one dimensional theory, for each $n$,
$$\sum_{i=1}^n\left(\left(\int_0^t\inner{\sy_r}{e_i}_{\mathcal{H}}dW_r \right)^2 -  \int_0^t\inner{\sy_r}{e_i}_H^2dr \right) $$
is a real valued martingale so to conclude the proof we only need to justify that the $\mathbbm{P}-a.e.$ limit also holds in $L^1(\Omega;\R)$ as convergence in this space preserves martingality. This is a straightforwards application of the Monotone Convergence Theorem applied to each integral separately, which concludes the proof.

%\begin{align*}
%     \mathbbm{E}\left(\left\Vert\int_0^{t}\sy_r dW_r\right\Vert_{\mathcal{H}}^2 - \int_0^t\norm{\sy_r}_H^2dr \vert \mathcal{F}_s \right) &= \mathbbm{E}\left(\left(\left\Vert\int_s^{t}\sy_r dW_r\right\Vert_{\mathcal{H}}^2 - \int_{s\wedge \tau}^{t \wedge \tau}\norm{\sy_r}_H^2dr\right) + \left(\left\Vert\int_0^{s}\sy_r dW_r\right\Vert_{\mathcal{H}}^2 - \int_0^s\norm{\sy_r}_H^2dr \right) \vert \mathcal{F}_s \right)\\
  %   &= \mathbbm{E}
%\end{align*}

\end{proof}

We also look to reconcile this definition with the one often stated in the real valued case, as a limit in probability over any time partition with mesh approaching zero. 

\begin{proposition} \label{analagous one}
    Let $\sy \in \mathcal{I}^{\mathcal{H}}_T$ and consider any sequence of partitions $$I_l:= \left\{0=t^l_0 < t^l_1 < \dots < t^l_{k_l}=T\right\}$$ with $\max_j\vert t^l_{j}-t^l_{j-1} \vert \rightarrow 0$ as $l \rightarrow \infty$. Then for all $t\in[0,T]$, for any $\varepsilon > 0$, 
 $$\lim_{l \rightarrow \infty}\mathbbm{P}\left(\left\{\left\vert \sum_{t^l_{j+1} \leq t}\left\Vert\int_{t^l_{j}}^{t^l_{j+1}}\sy_r dW_r\right\Vert_{\mathcal{H}}^2 - \int_0^t\norm{\sy_r}_H^2dr \right\vert > \varepsilon \right\}\right) = 0. $$

\end{proposition}

\begin{proof}
We once again look to prove this by considering the finite dimensional projections on which the result is known to be true, before showing that it is preserved in the limit. As in the real valued case, [\cite{karatzas1991brownian}] Theorem 1.5.8, we introduce a sequence of stopping times $(\tau^n)$ defined at every $n \in \N$ by
$$\tau^n := n \wedge \inf\left\{ t \in [0,T]: \left\Vert\int_{0}^t\sy_r dW_r\right\Vert_{\mathcal{H}}^2 + \int_0^t\norm{\sy_r}_H^2dr \geq n \right\}.$$
For every $n$ we define the process $$\sy^n_{\cdot} := \sy_{\cdot}\mathbbm{1}_{\cdot \leq \tau^n}$$
and now look to show that 
\begin{equation} \label{new look to show that} \lim_{l \rightarrow \infty}\mathbbm{E}\left(\left\vert \sum_{t^l_{j+1} \leq t}\left\Vert\int_{t^l_{j}}^{t^l_{j+1}}\sy^n_r dW_r\right\Vert_{\mathcal{H}}^2 - \int_0^t\norm{\sy^n_r}_H^2dr \right\vert \right) = 0.\end{equation}
As $(\tau^n)$ approach infinity $\mathbbm{P}-a.s.$ then (\ref{new look to show that}) is sufficient to deduce the result, as in [\cite{karatzas1991brownian}] Theorem 1.5.8. Identically to the proof of Proposition \ref{quad of int} we have that
\begin{align*}
    &\mathbbm{E}\left(\left\vert \sum_{t^l_{j+1} \leq t}\left\Vert\int_{t^l_{j}}^{t^l_{j+1}}\sy^n_r dW_r\right\Vert_{\mathcal{H}}^2 - \int_0^t\norm{\sy^n_r}_H^2dr \right\vert \right)\\ & \qquad \qquad \qquad \qquad \qquad \qquad  = \mathbbm{E}\left(\left\vert \sum_{t^l_{j+1} \leq t}\sum_{i=1}^\infty\left(\int_{t^l_{j}}^{t^l_{j+1}}\inner{\sy^n_r}{e_i}_{\mathcal{H}}dW_r \right)^2 -  \int_0^t\inner{\sy^n_r}{e_i}_H^2dr  \right\vert \right)\\
    & \qquad \qquad \qquad \qquad \qquad \qquad  \leq  \mathbbm{E}\left(\sum_{i=1}^\infty\left\vert \sum_{t^l_{j+1} \leq t}\left(\int_{t^l_{j}}^{t^l_{j+1}}\inner{\sy^n_r}{e_i}_{\mathcal{H}}dW_r \right)^2 -  \int_0^t\inner{\sy^n_r}{e_i}_H^2dr  \right\vert \right)\\
    & \qquad \qquad \qquad \qquad \qquad \qquad  = \sum_{i=1}^\infty \mathbbm{E}\left(\left\vert \sum_{t^l_{j+1} \leq t}\left(\int_{t^l_{j}}^{t^l_{j+1}}\inner{\sy^n_r}{e_i}_{\mathcal{H}}dW_r \right)^2 -  \int_0^t\inner{\sy^n_r}{e_i}_H^2dr  \right\vert \right)
\end{align*}
where on the last line we have applied the Monotone Convergence Theorem. From [\cite{karatzas1991brownian}] Theorem 1.5.8 we know that for each $i \in \N$, 
$$\lim_{l \rightarrow 0}\mathbbm{E}\left(\left\vert \sum_{t^l_{j+1} \leq t}\left[\left(\int_{t^l_{j}}^{t^l_{j+1}}\inner{\sy^n_r}{e_i}_{\mathcal{H}}dW_r \right)^2 -  \int_0^t\inner{\sy^n_r}{e_i}_H^2dr \right] \right\vert \right) = 0 $$
so we would be done if we can justify the interchange of infinite sum and limit in $l$. We look to apply the Dominated Convergence Theorem (for the infinite sum again considered as an integral with counting measure), noting that for each fixed $i,l \in \N$,
\begin{align}
    \nonumber &\mathbbm{E}\left(\left\vert \sum_{t^l_{j+1} \leq t}\left[\left(\int_{t^l_{j}}^{t^l_{j+1}}\inner{\sy^n_r}{e_i}_{\mathcal{H}}dW_r \right)^2 -  \int_0^t\inner{\sy^n_r}{e_i}_H^2dr \right] \right\vert \right)\\\nonumber & \qquad \qquad \qquad \qquad \qquad \qquad  \leq \sum_{t^l_{j+1} \leq t}\mathbbm{E}\left[\left(\int_{t^l_{j}}^{t^l_{j+1}}\inner{\sy^n_r}{e_i}_{\mathcal{H}}dW_r \right)^2\right] + \mathbbm{E}\left[\int_0^t\inner{\sy^n_r}{e_i}_H^2dr \right]\\
    \nonumber&\qquad \qquad \qquad \qquad \qquad \qquad  = \sum_{t^l_{j+1} \leq t}\mathbbm{E}\left[\int_{t^l_{j}}^{t^l_{j+1}}\inner{\sy^n_r}{e_i}_H^2dr\right] + \mathbbm{E}\left[\int_0^t\inner{\sy^n_r}{e_i}_H^2dr \right]\\
    &\qquad \qquad \qquad \qquad \qquad \qquad = 2\mathbbm{E}\left[\int_0^t\inner{\sy^n_r}{e_i}_H^2dr \right] \label{nearly finished}
\end{align}
which is a bound uniform in $l$ and summable in $i$. This is our dominating function which justifies the application of the Dominated Convergence Theorem, concluding the proof.

\end{proof}

\begin{lemma} \label{martingale quadratic variation limit}
    Suppose that $(M^n)$ is a sequence of martingales in $\mathcal{M}^2_c(\mathcal{H})$ which at every time $t \geq 0$, converges in $L^2\big(\Omega;\mathcal{H}\big)$ to some $M_t$. Suppose in addition that at any time $t \geq 0$, the sequence $\left([M^n]_t \right)$ converges to some $\psi_t$ in $L^1\big(\Omega;\R\big)$ where $\psi$ is continuous, adapted and non-decreasing ($\mathbbm{P}-a.s.$). Then $M \in\mathcal{M}^2_c(\mathcal{H})$ and $[M]$ is indistinguishable from $\psi$. 
\end{lemma}

\begin{proof}
    The fact that $M \in \mathcal{M}^2_c(\mathcal{H})$ is immediate from Proposition \ref{prop m2c closed}, so we move on to the quadratic variation. Observe that for each $n$, by definition $$\norm{M^n}^2_{\mathcal{H}}-[M^n]$$ is a real valued martingale, so the $L^1\big(\Omega;\R\big)$ limit at each time $t$, $$\lim_{n \rightarrow \infty}\big(\norm{M^n}^2_{\mathcal{H}}-[M^n]\big)$$ is again a real valued martingale if it exists.  But this is clear as it is just %Of course it needs to be justified that this limit is well defined. Firstly we note from the property that martingales have constant expectation then for each $t \geq 0$,
    %$$ \mathbbm{E}\left( \norm{M^n_t}^2_{\mathcal{H}}-[M^n]_t\right) = \mathbbm{E}\left( \norm{M^n_0}^2_{\mathcal{H}}-[M^n]_0\right) = \mathbbm{E}\left( \norm{M^n_0}^2_{\mathcal{H}}\right)$$
    %so in particular
    %$$\mathbbm{E}\left([M^n]_t \right) = \mathbbm{E}\left( \norm{M^n_t}^2_{\mathcal{H}}-\norm{M^n_0}^2_{\mathcal{H}}\right).$$

   % Beyond showing that the limit of $([M^n])$ exists in $L^1\left(\Omega;\R \right)$ at each $t \geq 0$, we will show that the limit exists in $L^1\left(\Omega;C\left([0,T];\R \right) \right)$ for each $T \geq 0$. We do this by showing the Cauchy property, for which we fix $n, k \in \N$ and  

   $$\lim_{n \rightarrow \infty}\norm{M^n}^2_{\mathcal{H}}-\lim_{n \rightarrow \infty}[M^n] = \norm{M}^2_\mathcal{H} - \psi$$ by definition of the limit. From Definition \ref{def of quad} then if $\psi_0 = 0$ $\mathbbm{P}-a.s.$ we have that $\psi$ satisfies the conditions of the quadratic variation of $M$, so is indistinguishable from it. This property is immediate as $\psi_0$ is the limit in $L^1\left(\Omega;\R\right)$ of $([M^n]_0)$ by assumption, whilst this is just a sequence of zeros $\mathbbm{P}-a.s.$, hence the result. 
\end{proof}

%\begin{theorem}
 %   For $M \in \mathcal{M}^2_c(\mathcal{H})$, $M_0 = 0$, there exists a unique $\mathcal{L}(\mathcal{H})$ valued process $Q^M = (Q^M_t)$ such that
  %  \begin{enumerate}
   %     \item For all times $t$ and $\omega \in \Omega$, $Q^M_t(\omega)$ is positive and trace-class;
    %    \item For every $g,h \in \mathcal{H}$, $$[\inner{M}{g}_{\mathcal{H}},\inner{M}{h}_{\mathcal{H}}]_t = \int_0^t \inner{Q^M_sg}{h}_{\mathcal{H}}d[M]_s$$ where $[\cdot,\cdot] = ([\cdot,\cdot]_t)$ is the cross variation process of real valued martingales.
     %    \end{enumerate}
      %   We call $Q^M$ the correlation process associated to $M$.
   
%\end{theorem}

%\textcolor{cyan}{Things to show in this subsection: Integral as constructed in 2.1 is a martingale. Cylindrical process on $\mathcal{H}$ is a martingale if and only if its corresponding regular representation on $\mathcal{H}'$, $\mathcal{H} \subset \mathcal{H}'$ is a martingale.  Cylindrical Brownian Motion is a Martingale, define explicitly its quadratic variation and correlation process.}\\

%\begin{definition}
  %  The cross-variation process between a $\mathcal{H}$ valued process $\sy$ and a real valued one $Y$ is defined, analogously to the standard case, as $$[\sy,Y]_t = \lim_{\abs{P} \rightarrow 0}\sum_{i=0}^{n-1} (\sy_{t_{i+1}}-\sy_{t_{i}})(Y_{t_{i+1}}-Y_{t_i})$$ where the limit is taken in probability of the $\mathcal{H}$ valued random variable, over nested partitions of $[0,t]$ of mesh approaching zero. 
%\end{definition}

In part due to this result, we provide an alternative characterisation of the quadratic variation under some additional boundedness assumptions. We reintroduce the projections $(\mathcal{P}_k)$ used in Proposition \ref{prop m2c closed}. 

\begin{proposition} \label{characterisation of quadratic}
    For any $M \in \mathcal{M}^2_c(\mathcal{H})$, $[M]$ has the representation $$[M] = \sum_{i=1}^\infty \left[\inner{M}{e_i}_{\mathcal{H}}\right]$$
    $\mathbbm{P}-a.s.$ for the limit taken in $C\left([0,T];\R\right)$ for any $T \geq 0$. 
\end{proposition}

%Before we begin the proof let's briefly comment on the assumption (\ref{assumption for comment}). We know, of course, that 

\begin{proof}
     Observe that $$\norm{\mathcal{P}_kM}_{\mathcal{H}}^2 = \sum_{i=1}^k \inner{M}{e_i}_{\mathcal{H}}^2$$ which is the sum of $k$ one dimensional submartingales, and in particular $$ \sum_{i=1}^k \left(\inner{M}{e_i}_{\mathcal{H}}^2 - \left[ \inner{M}{e_i}_{\mathcal{H}}\right]\right)$$
    is a martingale. From the definition of the quadratic variation it is thus clear that \begin{equation}\label{clear rep}\left[\mathcal{P}_kM\right] = \sum_{i=1}^k\left[ \inner{M}{e_i}_{\mathcal{H}}\right].\end{equation}
    Moreover we have that at each $t \geq 0$, $\mathcal{P}_kM_t$ is convergent to $M_t$ $\mathbbm{P}-a.s.$ in $\mathcal{H}$, and furthermore in $L^2\left(\Omega;\mathcal{H}\right)$ from an application of the Monotone Convergence Theorem to the difference process $\norm{M - \mathcal{P}_kM}_{\mathcal{H}}^2 = \sum_{i=k+1}^\infty \inner{M}{e_i}_{\mathcal{H}}^2$. Reminiscent of Lemma \ref{martingale quadratic variation limit} we look to show convergence of the sequence $\left(\left[\mathcal{P}_kM\right] \right)$ to some $\psi$ at each $t \geq 0$ in $L^1\left(\Omega;\R\right)$, and in fact we show the stronger convergence in $L^1\left(\Omega;C\left([0,T];\R\right)\right)$ for each $T \geq 0$. We proceed by showing the Cauchy property in this Banach Space. To do this we consider, for $j < k$, \begin{equation}\label{its a me identity}[\mathcal{P}_kM] - [\mathcal{P}_jM] = \sum_{i=j+1}^k\left[ \inner{M}{e_i}_{\mathcal{H}}\right].\end{equation}
    From this identity and the preceding work, it is clear that
    \begin{equation}\label{before therefore}[\mathcal{P}_kM] - [\mathcal{P}_jM] = [\mathcal{P}_kM -\mathcal{P}_jM]\end{equation}
    and therefore
    \begin{equation} \label{its a therefore}\sup_{t\in[0,T]}\left\vert [\mathcal{P}_kM]_t - [\mathcal{P}_jM]_t\right\vert = \sup_{t\in[0,T]}[\mathcal{P}_kM -\mathcal{P}_jM]_t = [\mathcal{P}_kM -\mathcal{P}_jM]_T. %= \sum_{i=j+1}^k\left[ \inner{M}{e_i}_{\mathcal{H}}\right]_T.
    \end{equation}
    Furthermore we are concerned with a control in expectation of this term. From the property that real valued martingales have constant expectation,
    \begin{align*}\mathbbm{E}\left( \norm{\mathcal{P}_kM_T -\mathcal{P}_jM_T }^2_{\mathcal{H}}-[\mathcal{P}_kM - \mathcal{P}_jM]_T\right) &= \mathbbm{E}\left( \norm{\mathcal{P}_kM_0 - \mathcal{P}_jM_0}^2_{\mathcal{H}}-[\mathcal{P}_kM - \mathcal{P}_jM]_0\right)\\ &= \mathbbm{E}\left( \norm{\mathcal{P}_kM_0 - \mathcal{P}_jM_0}^2_{\mathcal{H}}\right)
    \end{align*}
    so in particular
    \begin{align*}\mathbbm{E}\left([\mathcal{P}_kM - \mathcal{P}_jM]_T \right) &= \mathbbm{E}\left( \norm{\mathcal{P}_kM_T - \mathcal{P}_jM_T}^2_{\mathcal{H}}-\norm{\mathcal{P}_kM_0 - \mathcal{P}_jM_0}^2_{\mathcal{H}}\right)\\
    &= \mathbbm{E}\left(\sum_{i=j+1}^k\left(\inner{M_T}{e_i}_{\mathcal{H}}^2 - \inner{M_0}{e_i}_{\mathcal{H}}^2\right)\right)\\
    &\leq \mathbbm{E}\left(\sum_{i=j+1}^k\inner{M_T}{e_i}_{\mathcal{H}}^2 \right) + \mathbbm{E}\left(\sum_{i=j+1}^k \inner{M_0}{e_i}_{\mathcal{H}}^2\right).
    \end{align*}
By the square integrability assumption on $M$ we have that at every $t \geq 0$, $\mathbbm{E}\left(\sum_{i=1}^\infty\inner{M_t}{e_i}_{\mathcal{H}}^2 \right) < \infty$ which justifies that the right hand side of the above approaches zero as $j \rightarrow \infty$, uniformly in $k$. With (\ref{its a therefore}) we deduce that the sequence $([\mathcal{P}_k M] ) $ is Cauchy in $L^1\left(\Omega;C\left([0,T];\R\right)\right)$ for each $T \geq 0$ so admits a limit in this space which we call $\psi$. Through the $L^1\left(\Omega;C\left([0,T];\R\right)\right)$ convergence we can deduce the existence of a subsequence which is $\mathbbm{P}-a.s.$ convergent in $C\left([0,T];\R\right)$. In fact we can upgrade this to convergence over the whole sequence, as was note that $\mathbbm{P}-a.s.$ at each $t \geq 0$ the real valued sequence $([\mathcal{P}_kM]_t)$ not only has a convergent subsequence but is non-decreasing in $k$, which implies convergence of the whole sequence. Convergence of the whole sequence $\mathbbm{P}-a.s.$ in $C\left([0,T];\R\right)$ can then be deduced by the Cauchy property with (\ref{its a therefore}) and (\ref{before therefore}). In summary thus far we have that 
$$\psi = \sum_{i=1}^\infty \left[\inner{M}{e_i}_{\mathcal{H}}\right]$$
    $\mathbbm{P}-a.s.$ for the limit taken in $C\left([0,T];\R\right)$ for any $T \geq 0$. It only remains to show that $\psi$ is indistinguishable from $[M]$, which we do by verifying the conditions of Lemma \ref{martingale quadratic variation limit}. The convergence has already been established hence we need only the regularity on $\psi$. It is continuous by construction, and must be adapted as $\psi_t$ is the $\mathbbm{P}-a.s.$ limit of the $\mathcal{F}_t-$measurable $[\mathcal{P}_{k}M]_t$ on the complete measure space. This limit similarly preserves the non-decreasing property, so $\psi$ satisfies the required conditions and must be indistinguishable from $[M]$ due to Lemma \ref{martingale quadratic variation limit}.

\end{proof}

%With this representation in mind, we can show a specific case of the famed Burkholder-Davis-Gundy Inequality.

%\begin{theorem}
 %   There exists a constant $C$ such that for any $M \in \mathcal{M}^2_c(\mathcal{H})$ with $M_0 = 0$ $\mathbbm{P}-a.s.$, 
  %  $$\mathbbm{E}\left(\sup_{t \in [0,T]}\norm{M_t}_{\mathcal{H}} \right) \leq C\mathbbm{E}\left([M]_T^{\frac{1}{2}}\right)$$
   % for all $T \geq 0$. 
%\end{theorem}

%\begin{proof}
%Our approach is direct and uses the well established real valued version, see for example [\cite{burkholder1972integral}]. Firstly we have that $$ \sup_{t \in [0,T]}\norm{M_t}_{\mathcal{H}} \leq \sup_{t \in [0,T]} \sum_{i=1}^\infty \left\vert \inner{M_t}{e_i}_{\mathcal{H}}\right\vert \leq \sum_{i=1}^\infty\sup_{t \in [0,T]}  \left\vert \inner{M_t}{e_i}_{\mathcal{H}}\right\vert$$
%$\mathbbm{P}-a.s.$, which is not obviously finite at this stage. Referencing [\cite{burkholder1972integral}] again, we know that there exists a constant $C$ whereby 
%$$\mathbbm{E}\left( \sup_{t \in [0,T]}  \left\vert \inner{M_t}{e_i}_{\mathcal{H}}\right\vert\right) \leq C\mathbbm{E}\left$$

%\end{proof}

Following on from the quadratic variation we look to introduce the cross-variation between martingales. Guided by a motivation to use the Stratonovich integral, introduced in Subsection \ref{sub strato}, we will need to consider the cross-variation between elements of $\mathcal{M}^2_c(\mathcal{H})$ and $\mathcal{M}^2_c$. Defining this akin to a polarisation identity is out of the question as one cannot take sums of these martingales, so we look to use the characterisation in terms of their product.\\

Indeed from the classical theory, for example [\cite{karatzas1991brownian}] Theorem 1.5.13, for any given $\sy \in \mathcal{M}^2_c(\mathcal{H})$, $Y \in \mathcal{M}^2_c$ and $e_i$ a basis vector of $\mathcal{H}$, there exists a unique continuous, adapted, bounded-variation process $[\inner{\sy}{e_i}_{\mathcal{H}}, Y]$ with $[\inner{\sy}{e_i}_{\mathcal{H}}, Y]_0 = 0$ ($\mathbbm{P}-a.s.$) such that $$\inner{\sy}{e_i}_{\mathcal{H}} Y - [\inner{\sy}{e_i}_{\mathcal{H}}, Y]$$
is a real valued martingale. We would love to immediately have the existence and uniqueness of a corresponding $\mathcal{H}$ valued process which gives a martingale when subtracted from $\sy Y$, but that isn't clear in the same way that the quadratic variation was as in that case $\norm{\sy}_\mathcal{H}^2$ was a genuine real valued submartingale. Our approach, therefore, comes from the characterisation in Proposition \ref{characterisation of quadratic}. We make a first definition for the projected process.

\begin{definition} \label{fin dim cross def}
    For $\sy \in \mathcal{M}^2_c(\mathcal{H})$ and $Y \in \mathcal{M}^2_c$, for any $k \in \N$, we define the cross-variation process $[\mathcal{P}_k\sy, Y]$ by $$ [\mathcal{P}_k\sy, Y]:= \sum_{i=1}^k[\inner{\sy}{e_i}_{\mathcal{H}}, Y]e_i. $$
\end{definition}

Let's take a moment to process this definition, particularly in terms of how we would like to define the cross-variation process. It falls from the corresponding properties of the real valued cross-variations that $[\mathcal{P}_k\sy, Y]$ is continuous, adapted and of bounded-variation (one can apply the triangle inequality for the norm $\norm{\cdot}_{\mathcal{H}})$ satisfying $[\mathcal{P}_k\sy, Y]_0 = 0$ $\mathbbm{P}-a.s.$. In addition observe that $$ (\mathcal{P}_k\sy) Y - [\mathcal{P}_k\sy, Y] = \sum_{i=1}^k\left(\inner{\sy}{e_i}_{\mathcal{H}}Y -  [\inner{\sy}{e_i}_{\mathcal{H}}, Y]\right)e_i$$ 
which we hope to be an $\mathcal{H}$ valued martingale. To show this consider arbitrary $\phi \in \mathcal{H}$. Then
\begin{align*}
    \left\langle (\mathcal{P}_k\sy) Y - [\mathcal{P}_k\sy, Y], \phi \right\rangle_{\mathcal{H}} &= \left\langle \sum_{i=1}^k\left(\inner{\sy}{e_i}_{\mathcal{H}}Y -  [\inner{\sy}{e_i}_{\mathcal{H}}, Y]\right)e_i, \sum_{j=1}^\infty\inner{\phi}{e_j}_{\mathcal{H}}e_j \right\rangle_{\mathcal{H}}\\
    &= \sum_{i=1}^k\sum_{j=1}^\infty \left\langle \left(\inner{\sy}{e_i}_{\mathcal{H}}Y -  [\inner{\sy}{e_i}_{\mathcal{H}}, Y]\right)e_i, \inner{\phi}{e_j}_{\mathcal{H}}e_j \right\rangle_{\mathcal{H}}\\
    &= \sum_{i=1}^k \left(\inner{\sy}{e_i}_{\mathcal{H}}Y -  [\inner{\sy}{e_i}_{\mathcal{H}}, Y]\right)\inner{\phi}{e_i}_{\mathcal{H}}
\end{align*}
where we recall that each $\inner{\sy}{e_i}_{\mathcal{H}}Y -  [\inner{\sy}{e_i}_{\mathcal{H}}, Y]$ is a martingale, hence too is

$\left(\inner{\sy}{e_i}_{\mathcal{H}}Y -  [\inner{\sy}{e_i}_{\mathcal{H}}, Y]\right)\inner{\phi}{e_i}_{\mathcal{H}}$ and therefore the finite sum is a martingale. To comment on the uniqueness, we introduce a characterisation of indistinguishability in Hilbert Spaces.

\begin{lemma} \label{arming lemma}
    Let $\sy$ and $\py$ be $\mathcal{H}$ valued processes. Then $\sy$ is indistinguishable from $\py$ if and only if for every basis vector $e_i$, $\inner{\sy}{e_i}_{\mathcal{H}}$ is indistinguishable from $\inner{\py}{e_i}_{\mathcal{H}}$.
\end{lemma}

\begin{proof}
    The first implication is trivial so we consider only the reverse one. That is, assume that for every $e_i$ there exists a set $A_i \in \mathcal{F}$ with $\mathbbm{P}(A_i) = 1$ and for all $ t \geq 0$ and $\omega \in A_i$, $$\inner{\sy_t(\omega)}{e_i}_{\mathcal{H}} = \inner{\py_t(\omega)}{e_i}_{\mathcal{H}}.$$
    We now define $A:=\bigcap_{i}A_i$, which is again of full probability and in $\mathcal{F}$, and for any $\omega \in A$, $t \geq 0$,
\begin{align*}
    \norm{\sy_t(\omega) - \py_t(\omega)}_{\mathcal{H}}^2 = \sum_{i=1}^\infty \inner{\sy_t(\omega) - \py_t(\omega)}{e_i}_{\mathcal{H}}^2 = 0
\end{align*}
which completes the proof.
\end{proof}

Armed with this lemma we consider the uniqueness of $[\mathcal{P}_k\sy, Y]$. Suppose that $\Pi$ is an $\mathcal{H}$ valued process which continuous, adapted and of bounded-variation satisfying $\Pi_0 = 0$ $\mathbbm{P}-a.s.$. Moreover suppose that $$(\mathcal{P}_k \sy) Y - \Pi$$ is an $\mathcal{H}$ valued martingale. Take any basis vector $e_j$. Then \begin{equation}\nonumber\left\langle (\mathcal{P}_k \sy) Y - \Pi , e_j\right\rangle_{\mathcal{H}}\end{equation}
is a real valued martingale, but this is just \begin{equation}\label{the real valued martingale} \sum_{i=1}^k\inner{\sy}{e_i}_{\mathcal{H}}\inner{e_i}{e_j}_{\mathcal{H}}Y - \inner{\Pi}{e_j}_{\mathcal{H}}.\end{equation}
There are two cases here: $j \leq k$ and $j > k$. In the latter case then we have that $\inner{\Pi}{e_j}_{\mathcal{H}}$ is a martingale. In both cases from the regularity of $\Pi$ we have again that $\inner{\Pi}{e_j}_{\mathcal{H}}$ is continuous, adapted and of bounded-variation satisfying $\inner{\Pi_0}{e_j}_{\mathcal{H}} = 0$ $\mathbbm{P}-a.s.$. In the real valued case it is classical that a martingale of bounded-variation must be constant, hence $\mathbbm{P}-a.s.$ for all $t \geq 0$ we have that $\inner{\Pi_t}{e_j}_{\mathcal{H}} = 0.$ Therefore $\inner{\Pi}{e_j}_{\mathcal{H}}$ is indistinguishable from $\left\langle [\mathcal{P}_k \sy, Y], e_j\right\rangle $ which is similarly zero (recall Definition \ref{fin dim cross def}).\\

In the alternative case $j \leq k$, the real valued martingale (\ref{the real valued martingale}) is given by $$ \inner{\sy}{e_j}_{\mathcal{H}}Y - \inner{\Pi}{e_j}_{\mathcal{H}}$$
and from the much discussed uniqueness in this setting we have that $\inner{\Pi}{e_j}_{\mathcal{H}}$ is indistinguishable from $[\inner{\sy}{e_j}_{\mathcal{H}},Y]$ which is simply $\left\langle [\mathcal{P}_k \sy , Y] , e_k \right\rangle$ from its definition. Combining with Lemma \ref{arming lemma}, what we have proved is the following.

\begin{proposition} \label{prop for fin dim cross characterisation}
    For $\sy \in \mathcal{M}^2_c(\mathcal{H})$ and $Y \in \mathcal{M}^2_c$, for any $k \in \N$,  $[\mathcal{P}_k\sy, Y]$ is the unique continuous, adapted, bounded-variation $\mathcal{H}$ valued process satisfying $[\mathcal{P}_k\sy, Y]_0 = 0$ $\mathbbm{P}-a.s.$ such that $$(\mathcal{P}_k\sy) Y - [\mathcal{P}_k\sy, Y] $$
    is an $\mathcal{H}$ valued martingale.
\end{proposition}

It remains for us to define the cross-variation process $[\sy,Y]$. Referencing [\cite{karatzas1991brownian}] Problem 1.5.7 for example, we know that $\mathbbm{P}-a.s.$ for any $t \geq 0$, $$[\inner{\sy}{e_i}_{\mathcal{H}}, Y]_t^2 \leq [\inner{\sy}{e_i}_{\mathcal{H}}]_t[Y]_t$$
hence \begin{align*}
    \sup_{t \in [0,T]}[\inner{\sy}{e_i}_{\mathcal{H}}, Y]_t^2 &\leq \sup_{t\in [0,T]}\left([\inner{\sy}{e_i}_{\mathcal{H}}]_t[Y]_t \right)\\
    &\leq \left(\sup_{t\in [0,T]}[\inner{\sy}{e_i}_{\mathcal{H}}]_t\right)\left(\sup_{t\in[0,T]}[Y]_t \right)\\
    &= [\inner{\sy}{e_i}_{\mathcal{H}}]_T[Y]_T.   
\end{align*}
Moreover for $j < k$,
\begin{align*}
    \sup_{t \in [0,T]}\left\Vert [\mathcal{P}_k\sy, Y]_t - [\mathcal{P}_j\sy, Y]_t\right\Vert_{\mathcal{H}}^2 &= \sup_{t \in [0,T]}\left\Vert \sum_{i=j+1}^k\left[\inner{\sy}{e_i}_{\mathcal{H}},Y \right]_t e_i\right\Vert_{\mathcal{H}}^2\\
    &= \sup_{t \in [0,T]}\sum_{i=j+1}^k\left[\inner{\sy}{e_i}_{\mathcal{H}},Y \right]_t^2\\
&\leq \sum_{i=j+1}^k\sup_{t \in [0,T]}\left[\inner{\sy}{e_i}_{\mathcal{H}},Y \right]_t^2\\
    &\leq \sum_{i=j+1}^k[\inner{\sy}{e_i}_{\mathcal{H}}]_T[Y]_T\\
    &=  [\mathcal{P}_k\sy - \mathcal{P}_j\sy]_T[Y]_T
\end{align*}
referring to (\ref{its a me identity}) and (\ref{before therefore}) for the last line. We look to follow a similar approach to Proposition \ref{characterisation of quadratic}, taking the expectation. To do this we introduce the localising times $$\tau_n := n \wedge \inf\left\{0 \leq t < \infty : [Y]_t \geq n \right\}.$$
Then
\begin{align*}
    &\mathbbm{E}\left( \left( \sup_{t \in [0,T]}\left\Vert [\mathcal{P}_k\sy, Y]_t\mathbbm{1}_{t \leq \tau_n} - [\mathcal{P}_j\sy, Y]_t\mathbbm{1}_{t \leq \tau_n}\right\Vert_{\mathcal{H}} \right)^2 \right)\\ & \qquad \qquad \qquad \qquad \qquad \qquad \qquad \qquad \qquad = \mathbbm{E}\left( \sup_{t \in [0,T]}\left\Vert [\mathcal{P}_k\sy, Y]_t\mathbbm{1}_{t \leq \tau_n} - [\mathcal{P}_j\sy, Y]_t\mathbbm{1}_{t \leq \tau_n}\right\Vert_{\mathcal{H}}^2 \right)\\
    &\qquad \qquad \qquad \qquad \qquad \qquad \qquad \qquad \qquad\leq \mathbbm{E}\left( [\mathcal{P}_k\sy - \mathcal{P}_j\sy]_T[Y]_T\mathbbm{1}_{t \leq \tau_n} \right)\\
    & \qquad \qquad \qquad \qquad \qquad \qquad \qquad \qquad \qquad \leq n\mathbbm{E}\left( [\mathcal{P}_k\sy - \mathcal{P}_j\sy]_T\right)
\end{align*}
which was shown to approach zero as $j \rightarrow \infty$, uniformly in $k$, in Proposition \ref{characterisation of quadratic}. We have thus demonstrated that for every $n\in \N$ and $T \geq 0$, the sequence $\left([\mathcal{P}_k\sy, Y]_\cdot\mathbbm{1}_{\cdot \leq \tau_n} \right)$ is Cauchy in the Banach Space $L^2\left( \Omega; L^\infty\left([0,T];\mathcal{H}\right)\right)$. We can therefore extract a subsequence which converges $\mathbbm{P}-a.s.$ in $L^\infty\left([0,T];\mathcal{H}\right)$. To remove the truncation we introduce the sets $$ A_n := \left\{ \omega \in \Omega: \tau_n(\omega) \geq T \right\}.$$
Then on every $A_n$, there exists a subsequence of $\left([\mathcal{P}_k\sy, Y]\right)$ which is $\mathbbm{P}-a.s.$ (within $A_n$) convergent in $L^\infty\left([0,T];\mathcal{H}\right)$. It should be noted that the choice of subsequence may be dependent on $n$, hence the separation, and also that this convergence is now of continuous processes so can be taken in $C\left([0,T];\mathcal{H}\right)$. We look to upgrade this convergence to be of the whole sequence, by another Cauchy argument. Fix any $\varepsilon > 0$, and we look to show the existence of a $J \in \N$ such that for all $k \geq J$,
$$ \sup_{t \in [0,T]}\left\Vert [\mathcal{P}_k\sy, Y]_t - [\mathcal{P}_J\sy, Y]_t\right\Vert_{\mathcal{H}}^2 < \varepsilon$$
or equivalently, as already shown in the proof,
$$ \sup_{t \in [0,T]}\sum_{i=J+1}^k\left[\inner{\sy}{e_i}_{\mathcal{H}},Y \right]_t^2 < \varepsilon.$$
Let the subsequence be indexed by $k_m$. The subsequence is convergent hence Cauchy, so there exists a $J_m$ such that for all $k_m > J_m$, 
$$ \sup_{t \in [0,T]}\left\Vert [\mathcal{P}_{k_m}\sy, Y]_t - [\mathcal{P}_{J_m}\sy, Y]_t\right\Vert_{\mathcal{H}}^2 < \varepsilon$$
or equivalently 
$$ \sup_{t \in [0,T]}\sum_{i=J_m+1}^{k_m}\left[\inner{\sy}{e_i}_{\mathcal{H}},Y \right]_t^2 < \varepsilon.$$
We now set $J := J_m$ and argue that for every $k > J$ there exists a $k_m > k$ such that
$$\sup_{t \in [0,T]}\sum_{i=J+1}^k\left[\inner{\sy}{e_i}_{\mathcal{H}},Y \right]_t^2 \leq \sup_{t \in [0,T]}\sum_{i=J+1}^{k_m}\left[\inner{\sy}{e_i}_{\mathcal{H}},Y \right]_t^2 < \varepsilon$$
which proves the Cauchy property. Hence on each $A_n$, the sequence $([\mathcal{P}_k\sy, Y])$ is $\mathbbm{P}-a.s.$ convergent in $C\left([0,T];\mathcal{H}\right)$. From the fact that $\mathbbm{P}\left(\bigcup_n A_n \right) = 1$, we obtain the $\mathbbm{P}-a.s.$ convergence on $\Omega$.

\begin{definition} \label{definition for cross variation inf dim}
     For $\sy \in \mathcal{M}^2_c(\mathcal{H})$ and $Y \in \mathcal{M}^2_c$, we define the cross-variation process $[\sy, Y]$ by 
     $$ [\sy, Y] := \sum_{i=1}^\infty[\inner{\sy}{e_i}_{\mathcal{H}}, Y]e_i$$ $\mathbbm{P}-a.s.$  for the limit taken in $C\left([0,T];\mathcal{H} \right)$.
\end{definition}

Our next question is then very natural: do we have the corresponding characterisation as in Proposition \ref{prop for fin dim cross characterisation}? The fact that this cross-variation is continuous, adapted and starting from zero is proven as in Proposition \ref{characterisation of quadratic}. In fact the martingality of \begin{equation} \label{new marty} \sy Y -  [\sy, Y]\end{equation}  is again proven near identically, as we use that $$(\mathcal{P}_k \sy) Y - \sum_{i=1}^k[\inner{\sy}{e_i}_{\mathcal{H}}, Y]e_i $$
is a martingale, and the convergence of this at each $t \geq 0$ in $L^1\left(\Omega ; \mathcal{H}\right)$ to (\ref{new marty}). 
It is the bounded-variation which proves problematic. Perhaps the most logical approach is to show that the finite sums are of uniformly bounded total variation, through an argument like
$$V^T_{\mathcal{H}}\left([\mathcal{P}_k \sy , Y] \right) = V^T_{\mathcal{H}}\left( \sum_{i=1}^k[\inner{\sy}{e_i}_{\mathcal{H}}, Y]e_i  \right) \leq \sum_{i=1}^k V^T_{\mathcal{H}}\left([\inner{\sy}{e_i}_{\mathcal{H}}, Y]e_i  \right) = \sum_{i=1}^k V^T_{\R}\left([\inner{\sy}{e_i}_{\mathcal{H}}, Y] \right).$$
Then from the real valued theory, see [\cite{karatzas1991brownian}] Problem 1.5.7, we have that $$V^T_{\R}\left([\inner{\sy}{e_i}_{\mathcal{H}}, Y] \right) \leq \frac{1}{2}\left( [ \inner{\sy}{e_i}_{\mathcal{H}}]_T + [Y]_T \right) $$
but this means taking $\sum_{i=1}^k [Y]_T$ which will explode as $k \rightarrow \infty$. In lieu of this bounded-variation, we do of course have the weaker property that every projection is of bounded variation. In fact, this is enough for uniqueness.

Suppose that $\Pi$ is an $\mathcal{H}$ valued process which is continuous, adapted, satisfying $\Pi_0 = 0$ and such that for every basis vector $e_j$, $\inner{\Pi}{e_j}$ is of bounded-variation  $\mathbbm{P}-a.s.$. Moreover suppose that $$\sy Y - \Pi$$ is an $\mathcal{H}$ valued martingale. Take any basis vector $e_j$. Then \begin{equation}\nonumber\left\langle \sy Y - \Pi , e_j\right\rangle_{\mathcal{H}}\end{equation}
is a real valued martingale, but this is just \begin{equation}\nonumber\sum_{i=1}^\infty\inner{\sy}{e_i}_{\mathcal{H}}\inner{e_i}{e_j}_{\mathcal{H}}Y - \inner{\Pi}{e_j}_{\mathcal{H}}\end{equation}
or simply
$$ \inner{\sy}{e_j}_{\mathcal{H}}Y - \inner{\Pi}{e_j}_{\mathcal{H}}.$$
Thus $\inner{\Pi}{e_j}_{\mathcal{H}}$ is indistinguishable from  $[\inner{\sy}{e_j}_{\mathcal{H}},Y]$ which is simply $\left\langle [\sy , Y] , e_k \right\rangle$ from its definition. Combining with Lemma \ref{arming lemma} we deduce the uniqueness in this class, so have proven the following.

\begin{proposition} \label{prop for inf dim cross characterisation}
    For $\sy \in \mathcal{M}^2_c(\mathcal{H})$ and $Y \in \mathcal{M}^2_c$, $[\sy, Y]$ is the unique continuous, adapted $\mathcal{H}$ valued process satisfying $[\sy, Y]_0 = 0$ $\mathbbm{P}-a.s.$ such that for every basis vector $e_j$, $\inner{[\sy, Y]}{e_j}_{\mathcal{H}}$ is of bounded-variation $\mathbbm{P}-a.s.$ and  $$\sy Y - [\sy, Y] $$
    is an $\mathcal{H}$ valued martingale.
\end{proposition}

%Perhaps the most challenging property to show is that the limit is still of bounded-variation. We recall the classical functional analytic result that if a sequence of continuous functions are such that, on every interval $[0,T]$, they have uniformly bounded total variation (in $\mathcal{H}$) and converge to some function in $C\left([0,T];\mathcal{H} \right)$, then the limit function is again of bounded-variation. We consider this in the context of $\left([\mathcal{P}_k \sy , Y] \right)$, where $$V^T_{\mathcal{H}}\left([\mathcal{P}_k \sy , Y] \right) = V^T_{\mathcal{H}}\left( \sum_{i=1}^k[\inner{\sy}{e_i}_{\mathcal{H}}, Y]e_i  \right) \leq \sum_{i=1}^k V^T_{\mathcal{H}}\left([\inner{\sy}{e_i}_{\mathcal{H}}, Y]e_i  \right) = \sum_{i=1}^k V^T_{\R}\left([\inner{\sy}{e_i}_{\mathcal{H}}, Y] \right).$$
%From the real valued theory, see [\cite{karatzas1991brownian}] Problem 1.5.7, we have that $$V^T_{\R}\left([\inner{\sy}{e_i}_{\mathcal{H}}, Y] \right) \leq \frac{1}{2}\left( [ \inner{\sy}{e_i}_{\mathcal{H}}]_T + [Y]_T \right) $$

\begin{lemma} \label{cross variation convergence}
    Suppose that $(\sy^n)$ is a sequence of martingales in $\mathcal{M}^2_c(\mathcal{H})$ which at every time $t \geq 0$, converges in $L^2\big(\Omega;\mathcal{H}\big)$ to some $\sy_t$. Let $Y \in \mathcal{M}^2_c$. Suppose in addition that at any time $t \geq 0$, the sequence $\left([\sy^n,Y]_t \right)$ converges to some $y_t$ in $L^1\big(\Omega;\R\big)$ where $y$ is continuous, adapted and for every basis vector $e_j$, $\inner{y}{e_j}_{\mathcal{H}}$ is of bounded variation $\mathbbm{P}-a.s.$. Then $\sy \in\mathcal{M}^2_c(\mathcal{H})$ and $[\sy,Y]$ is indistinguishable from $y$. 
\end{lemma}

\begin{proof}
    The proof is, unsurprisingly, similar to that of Lemma \ref{martingale quadratic variation limit}. For each $n$ the martingale in question is the $\mathcal{H}$ valued one $$\sy^nY -[\sy^n,Y].$$ We use again that the $L^1\big(\Omega;\mathcal{H}\big)$ limit preserves martingality and that by H\"{o}lder's inequality the $L^1\big(\Omega;\mathcal{H}\big)$ limit of $\sy^n_tY_t$ is $\sy_tY_t$. With the same arguments as in Lemma \ref{martingale quadratic variation limit} we conclude the proof.
\end{proof}

%\begin{proposition} 
 %   Suppose that $(\sy^n) = ((\sy^n_t)_t)$ is a sequence of martingales in $\mathcal{M}^2_c(\mathcal{H})$ which at every time $t$ converges to a continuous $\sy=(\sy_t)_t$ in $L^2\big(\Omega;\mathcal{H}\big)$, and that $Y \in \mathcal{M}^2_c$. Then at any time $t$, $[\sy^n, Y]_t$ converges to $[\sy,Y]_t$ in $L^1\big(\Omega;\mathcal{H}\big)$.
%\end{proposition}

%\begin{proof}
 %   As you may have guessed the proof is largely similar to that of \ref{martingale quadratic variation limit}: on this occasion for each $n$ the martingale in question is the $\mathcal{H}$ valued one $$\sy^nY -[\sy^n,Y].$$ We use again that the $L^1\big(\Omega;\mathcal{H}\big)$ limit preserves martingality, that by H\"{o}lder's inequality the $L^1\big(\Omega;\mathcal{H}\big)$ limit of $\sy^n_tY_t$ is $\sy_tY_t$, that this limit also retains the finite variation and adaptedness properties (and too starting from zero, via the continuity), and the uniqueness of the cross-variation process to conclude the result. 
%\end{proof}

If the given processes were only (continuous) local martingales, then we can make a slightly modified version of the definition. Assuming without loss of generality that $\sy$ and $Y$ are locally square integrable (see the discussion after Definitionn \ref{localmartingaleintegrator}), localised by stopping times $(R_n)$ and $(T_n)$ respectively, then for a new sequence of stopping times defined by $$\tau_n = R_n \wedge T_n$$ the stopped processes $\sy^{\tau_n}$ and $Y^{\tau_n}$ are genuine square integrable martingales (in their respective spaces), so the cross variation $[\sy^{\tau_n},Y^{\tau_n}]$ can be defined. The canonical localisation procedure is evident once more, as the $(\tau_n)$ go to infinity almost surely, the consistency conditions that for $m \leq n$ and $t \leq \tau_m$ we have $$ \sy^{\tau_m}_t = \sy^{\tau_n}_t \qquad \textnormal{and} \qquad Y^{\tau_m}_t = Y^{\tau_n}_t$$ allow us once more to define the process at almost every $\omega$ and any $t$ by \begin{equation} \label{crossvlocalm} \Big([\sy,Y]_t\Big)(\omega) := \Big([\sy^{\tau_n},Y^{\tau_n}]_t\Big)(\omega)\end{equation} for any $n$ such that $t \leq \tau_n(\omega)$, independently of this choice of $n$. Then the process $$\sy Y - [\sy,Y]$$ is itself a local martingale, localised by the stopping times $(\tau_n)$. The argument justifying this is identical to Proposition \ref{integralislocalmartingale}, from which it is similarly clear that $$[\sy,Y]^{\tau_n} = [\sy^{\tau_n},Y^{\tau_n}].$$

In the traditional way, these notions can all be extended to semi-martingales (that is, a martingale plus a bounded-variation process). The quadratic and cross variation of such semi-martingales is then simply the quadratric/cross variation of the corresponding martingale parts. To this end we introduce the notation $\bar{\mathcal{M}}^2_c$ and $\bar{\mathcal{M}}^2_c(\mathcal{H})$ to be the corresponding spaces of square integrable continuous semi-martingales, and similarly $\bar{\mathcal{M}}_c, \bar{\mathcal{M}}_c(\mathcal{H})$ to be the spaces of continuous semi-martingales. 

\subsection{Integration Driven by Cylindrical Brownian Motion} \label{subsection inf dim integral}

For our analysis now we will need to make reference to two distinct Hilbert Spaces; one over which $\mathcal{W}$ is a Cylindrical Brownian Motion, and the other in which our integrand maps to. Henceforth we introduce $\mathfrak{U}$ as the Hilbert Space over which $\mathcal{W}$ is a Cylindrical Brownian Motion. We shall take $(e_i)$ as an orthonormal basis over $\mathfrak{U}$ and $(a_i)$ an orthonormal basis over $\mathcal{H}$.

\begin{definition}  \label{IHtw}
    Denote by $\mathcal{I}^{\mathcal{H}}_T(\mathcal{W})$ the class of progressively measurable operator valued processes $B$ belonging to the set $L^2\big(\Omega\times [0,T];\mathscr{L}^2(\mathfrak{U};\mathcal{H})\big).$ Measurability here is again defined with respect to the Borel Sigma algebra on $\mathscr{L}^2(\mathfrak{U};\mathcal{H}).$
\end{definition}

Note that we make no explicit reference to $\mathfrak{U}$, the space on which $\mathcal{W}$ is a cylindrical Brownian Motion. This is because, in practice, the space $\mathfrak{U}$ will be arbitrarily chosen; this shall be discussed later.

\begin{definition} \label{ihw}
    The class of processes $B$ such that $B \in \mathcal{I}^{\mathcal{H}}_T(\mathcal{W})$ for all $T$ will be denoted by $\mathcal{I}^{\mathcal{H}}(\mathcal{W}).$
\end{definition}

Recall from Subsection \ref{sub cylindrical processes} that if $\mathcal{W}$ is a Cylindrical Brownian motion over $\mathfrak{U}$, it can be formally represented by \begin{equation} \label{cylindricalBM} \mathcal{W}(t) = \sum_{i=1}^{\infty}e_iW^i_t\end{equation} where the $(W^i)$ are standard independent one-dimensional Brownian motions. 

\begin{definition} \label{squareintegrablewrtcylindrical}
    For $B \in \mathcal{I}^{\mathcal{H}}(\mathcal{W})$ we define the It\^{o} stochastic integral \begin{equation} \label{operatorintegral} \int_0^tB(s)d\mathcal{W}_s
    \end{equation}
    as the $\mathcal{H}$ valued random variable \begin{equation} \label{projectedintegral} \sum_{i=1}^{\infty}\int_0^tB_{e_i}(s)dW^i_s \end{equation}
    where each integral is defined as in Definition \ref{stochintdefined} and the infinite sum is taken in $L^2\big(\Omega;\mathcal{H}\big).$
    \end{definition}

The immediate response to this definition is to prove that $(\ref{projectedintegral})$ is well defined; that is the integrals are well defined, as is the limit. Firstly for each $i$, $B_{e_i}$ is trivially in $L^2\big(\Omega\times [0,T];\mathcal{H}\big)$ as this norm is bounded by the $L^2\big(\Omega\times [0,T];\mathscr{L}^2(\mathfrak{U};\mathcal{H})\big)$ norm of $B$. The progressive measurability is inherited from that of $B$.\\

In order to show that the limit of partial sums is well defined, we proceed similarly to the method applied for (\ref{regularrep}) and argue that the sequence of partial sums is Cauchy. Observe that
\begin{align*}
    \Big\vert\Big\vert\sum_{i=m}^n\int_0^tB_{e_i}(s)dW^i_s\Big\vert\Big\vert^2_{L^2\big(\Omega;\mathcal{H}\big)} &= \mathbbm{E}\Big\vert\Big\vert\sum_{i=m}^n\int_0^tB_{e_i}(s)dW^i_s\Big\vert\Big\vert^2_{\mathcal{H}}\\
    &= \sum_{i=m}^n\mathbbm{E}\int_0^t\norm{B_{e_i}(s)}^2_{\mathcal{H}}ds
\end{align*}
having applied the It\^{o} Isometry \ref{multidimItoIsom} to the above. But by assumption that $B \in L^2\big(\Omega\times [0,t];\mathcal{L}^2(\mathfrak{U};\mathcal{H})\big)$ we know $$\mathbbm{E}\int_0^t\sum_{i=1}^{\infty}\norm{B_{e_i}(s)}^2_{\mathcal{H}}ds < \infty$$ and thus, by Tonelli's theorem regarding the infinite sum as an integral with respect to the counting measure, $$\sum_{i=1}^\infty\mathbbm{E}\int_0^t\norm{B_{e_i}(s)}^2_{\mathcal{H}}ds < \infty$$ demonstrating the required Cauchy property. Of course the $L^2(\Omega;\mathcal{H})$ norm of the limit is the limit of the $L^2(\Omega;\mathcal{H})$ norms, so we have justified the following.
\begin{proposition} \label{ito isom for cylindrical}
    For $B \in \mathcal{I}^\mathcal{H}(\mathcal{W})$, we have 
    $$\mathbbm{E}\left(\left\Vert \int_0^t B(s) d\mathcal{W}_s\right\Vert_{\mathcal{H}}^2\right) = \mathbbm{E}\left(\int_0^t \norm{B(s)}_{\mathscr{L}^2(\mathfrak{U};\mathcal{H})}^2ds\right). $$
\end{proposition}
It is worth noting that whilst we impose the condition  $$\mathbbm{E}\int_0^t\sum_{i=1}^{\infty}\norm{B_{e_i}(s)}^2_{\mathcal{H}}ds < \infty$$ one may instead require the looser condition \begin{equation} \label{cylindricallocalcondition} \int_0^t\sum_{i=1}^{\infty}\norm{B_{e_i}(s)}^2_{\mathcal{H}}ds < \infty \qquad \mathbbm{P}-a.s.\end{equation} or equivalently that $B:\Omega\rightarrow L^2\big([0,t] ; \mathscr{L}^2(\mathfrak{U};\mathcal{H})\big)$ for $\mathbbm{P}-a.e.$ $\omega$. Our formulation follows the classical construction as laid out in Subsection \ref{classicalconstruction}, ensuring that the integral is a genuine square integrable martingale. We can just as straightforwardly follow the arguments from Definition \ref{localmartingaleintegrator}, which are laid out here.

\begin{definition} \label{overbar1}
    Denote by $\overbar{\mathcal{I}}^{\mathcal{H}}_T(\mathcal{W})$ the class of progressively measurable operator valued processes $B$ such that $B(\omega)$ belongs to the set $L^2\big( [0,T];\mathscr{L}^2(\mathfrak{U};\mathcal{H})\big)$ for $\mathbbm{P}-a.e.$ $\omega$.
    \end{definition}

\begin{definition} \label{stochasticallyintegrableprocesses}
    The class of processes $B$ such that $B \in \overbar{\mathcal{I}}^{\mathcal{H}}_T(\mathcal{W})$ for all $T$ will be denoted by $\overbar{\mathcal{I}}^{\mathcal{H}}(\mathcal{W}).$
\end{definition}

Using the template after Definition \ref{localmartingaleintegrator}, for a process $B \in \overbar{\mathcal{I}}^{\mathcal{H}}(\mathcal{W})$ let's introduce $$\tau_n:= n \wedge \inf\{0 \leq t < \infty: \int_0^t\norm{B(s)}_{\mathscr{L}^2(\mathfrak{U};\mathcal{H})}^2ds \geq n\}$$ taking the convention that the infimum of the empty set is infinite. The $(\tau_n)$ are stopping times as they are simply first hitting times of the continuous and adapted random variable $\int_0^t\norm{B(s)}_{\mathscr{L}^2(\mathfrak{U};\mathcal{H})}^2ds.$ These times tend to infinity $\mathbbm{P}-a.s.$ by condition (\ref{cylindricallocalcondition}). Now define the truncated processes $B^n$ as $$B^n(t):= B(t)\mathbbm{1}_{t \leq \tau_n}$$ and using the fact that for $m \leq n$, and $t \leq \tau_m$, we have $$B(t)\mathbbm{1}_{t \leq \tau_n}=B(t)\mathbbm{1}_{t \leq \tau_m}$$ we can make the following consistent definition.

\begin{definition} \label{cylindricalintlocal}

In the setting described, we define

\begin{equation} \label{localintegratordefinition2}
\Big(\int_0^tB(s)d\mathcal{W}_s\Big)(\omega) := \Big(\int_0^tB^n(s)d\mathcal{W}_s\Big)(\omega)\end{equation}
at $\mathbbm{P}-a.e.$ $\omega$ for any $n$ such that $t \leq \tau_n(\omega)$, noting that such an $n$ exists from the assumed condition (\ref{cylindricallocalcondition}).

\end{definition}

The justification of this definition is identical to that discussed in Subsection \ref{localchapter}, such that for this fixed $n$ we have $B^n \in \mathcal{I}^{\mathcal{H}}(\mathcal{W})$ and subsequently the complete definition
\begin{equation} \label{completedefinition}\Big(\int_0^tB(s)d\mathcal{W}_s\Big)(\omega) := \Big(\sum_{i=1}^\infty\int_0^tB_{e_i}(s)\mathbbm{1}_{s \leq \tau_n}dW^i_s\Big)(\omega)\end{equation} with the limit again in $L^2\big(\Omega;\mathcal{H}\big)$. To be clear this is a limit averaging over all $\Omega$ which is taken for the fixed $n$, which was chosen with respect to the specific $\omega$ in which we are evaluating the limit. We will have no quarrels in writing (\ref{localintegratordefinition2}) as a formal expression \begin{equation} \label{formalexpression} \sum_{i=1}^\infty\int_0^tB_{e_i}(s)dW^i_s\end{equation} motivated by the fact that as $B_{e_i} \in \overbar{\mathcal{I}}^{\mathcal{H}}$ (\ref{localmartingalewrtBM}) then by definition $$\int_0^tB_{e_i}(s)dW^i_s = \int_0^tB_{e_i}(s)\mathbbm{1}_{s \leq \tau_n}dW^i_s$$ at this same choice of $\omega$, for the same fixed $n$. We say the expression is only formal though, as the infinite sum in (\ref{formalexpression}) is \textit{not} the $L^2\big(\Omega;\mathcal{H}\big)$ limit of the partial sums of the local martingales as presented. We understand (\ref{formalexpression}) only by (\ref{completedefinition}), that is by choosing the $\omega$ at which we evaluate (\ref{formalexpression}), then fixing our $n$ associated to this $\omega$, before then taking the limit in $L^2\big(\Omega;\mathcal{H}\big)$ of the genuine square integrable martingales (given by stopping the local martingales at $\tau_n$) which is finally then evaluated at $\omega$.\\

We now look to show a series of properties of this integral which were shown for a one dimensional Brownian Motion across the earlier subsections. The first is the corresponding result of Theorem \ref{operator through integral 1D}.

\begin{theorem} \label{operatorthroughstochasticintegral}
    Suppose that $\mathcal{H}_1, \mathcal{H}_2$ are Hilbert spaces such that $B \in \overbar{\mathcal{I}}^{\mathcal{H}_1}(\mathcal{W})$ and $T \in \mathscr{L}(\mathcal{H}_1;\mathcal{H}_2)$. Then the process $TB$ defined by $$TB_{e_i}(s,\omega) = T\big(B_{e_i}(s,\omega)\big)$$ belongs to $\overbar{\mathcal{I}}^{\mathcal{H}_2}(\mathcal{W})$ and is such that $$ T\Big(\int_0^tB(s)d\mathcal{W}_s\Big) = \int_0^tTB(s)d\mathcal{W}_s.$$ In addition, the two integrals are defined pointwise $a.e.$ with respect to the same stopping times. 
\end{theorem}

\begin{proof}

We shall prove first that $TB \in \overbar{\mathcal{I}}^{\mathcal{H}_2}(\mathcal{W})$. The progressive measurability is preserved under the continuity of $T$, and for $C$ the (square of the) boundedness constant associated to $T$ we have $$\int_0^t\sum_{i=1}^{\infty}\norm{TB_{e_i}(s)}^2_{\mathcal{H}_2}ds \leq  C\int_0^t\sum_{i=1}^{\infty}\norm{B_{e_i}(s)}^2_{\mathcal{H}_1}ds < \infty$$ holding $\mathbbm{P}-a.s.$ as $B \in \overbar{\mathcal{I}}^{\mathcal{H}_1}(\mathcal{W})$. In addition for any stopping time $\tau_n$ as in Definition \ref{cylindricalintlocal}, \begin{align*}\mathbbm{E}\Big(\int_0^t\sum_{i=1}^{\infty}\norm{\big(TB_{e_i}(s)\big)\mathbbm{1}_{s \leq \tau_n}}^2_{\mathcal{H}_2}ds\Big) &= \mathbbm{E}\Big(\int_0^t\sum_{i=1}^{\infty}\norm{T\big(B_{e_i}(s)\mathbbm{1}_{s \leq \tau_n}\big)}^2_{\mathcal{H}_2}ds\Big)\\ &\leq  C\mathbbm{E}\Big(\int_0^t\sum_{i=1}^{\infty}\norm{B_{e_i}(s)\mathbbm{1}_{s \leq \tau_n}}^2_{\mathcal{H}_1}ds\Big)\\ &< \infty\end{align*} so the new stochastic integral $$\int_0^tTB(s)d\mathcal{W}_s$$ can be constructed using the same sequence of stopping times. We will freely use linearity of $T$ to commute it with the indicator function, not showing this explicitly with brackets. To carry $T$ through the integral however let's avoid danger and write out explicitly what we want to show, which is for almost any choice of $\omega$ with the associated $\tau_n$ and fixed $n$ as in Definition \ref{cylindricalintlocal}, then $$T\bigg(\Big(\int_0^tB(s)d\mathcal{W}_s\Big)(\omega)\bigg):= T\bigg(\Big(\sum_{i=1}^\infty \int_0^t B_{e_i}^n(s)dW^i_s\Big)(\omega)\bigg) = \Big(\sum_{i=1}^\infty \int_0^t TB_{e_i}^n(s)dW^i_s\Big)(\omega)$$ where the left hand side limit is taken in $L^2\big(\Omega;\mathcal{H}_1\big)$ and the right side in $L^2\big(\Omega;\mathcal{H}_2\big)$. From the $L^2\big(\Omega;\mathcal{H}_1\big)$ limit there exists a subsequence convergent almost everywhere in $\mathcal{H}_1$ (we assume w.l.o.g that $\omega$ belongs to this set of full probability). Working with this subsequence, we can pass the $\mathbbm{P}-a.e.$ limit through the continuous $T$ such that it is now the $\mathbbm{P}-a.e.$ limit in $\mathcal{H}_2$. Linearity of $T$ allows us to hit each term in the sum individually, so we have now that \begin{equation} \label{embeddingproofog} T\bigg(\Big(\int_0^tB(s)d\mathcal{W}_s\Big)(\omega)\bigg) = \bigg(\lim_{n_k \rightarrow \infty} \sum_{i=1}^{n_k} T \Big(\int_0^t B_{e_i}^n(s)dW^{i}_s\Big)\bigg)(\omega)\end{equation} for the limit $\mathbbm{P}-a.e.$ of the subsequence indexed by $(n_k)$. Applying Theorem \ref{operator through integral 1D}, we can commute $T$ with the integral on the right side of (\ref{embeddingproofog}). However we have justified already that the limit over the whole sequence in $L^2\big(\Omega;\mathcal{H}_2\big)$ exists, agreeing with the $L^2\big(\Omega;\mathcal{H}_2\big)$ limit of the subsequence, which in turn agrees with the $\mathbbm{P}-a.e.$ limit. Thus by definition of the stochastic integral, (\ref{embeddingproofog}) is simply the equality $$T\bigg(\Big(\int_0^tB(s)d\mathcal{W}_s\Big)(\omega)\bigg) = \Big(\int_0^t TB(s)d\mathcal{W}_s\Big)(\omega)$$ completing the proof.

\end{proof}

This is stated in the more general form for a process only in $\overbar{\mathcal{I}}^{\mathcal{H}_1}(\mathcal{W})$, but from the proof the following is clear. 

\begin{corollary} \label{operatorthroughstochasticintegral2}
    Suppose that $\mathcal{H}_1, \mathcal{H}_2$ are Hilbert spaces such that $B \in \mathcal{I}^{\mathcal{H}_1}(\mathcal{W})$ and $T \in \mathscr{L}(\mathcal{H}_1;\mathcal{H}_2)$. Then the process $TB$ defined by $$TB_{e_i}(s,\omega) = T\big(B_{e_i}(s,\omega)\big)$$ belongs to $\mathcal{I}^{\mathcal{H}_2}(\mathcal{W})$ and is such that $$ T\Big(\int_0^tB(s)d\mathcal{W}_s\Big) = \int_0^tTB(s)d\mathcal{W}_s.$$ 
\end{corollary}

\begin{proposition} \label{unbounded take it in 2}
    Let $B \in \bar{\mathcal{I}}^{\mathcal{H}}(\mathcal{W})$ and $\phi: \Omega \rightarrow \mathcal{H}$ be $\mathcal{F}_0-$measurable. Then for every $t>0$ we have that
    \begin{equation} \label{reffyyy4}
        \left\langle \int_0^tB_r d\mathcal{W}_r, \phi \right\rangle_{\mathcal{H}} = \int_0^t\left\langle B_r,\phi \right\rangle_{\mathcal{H}}d\mathcal{W}_r
    \end{equation}
    $\mathbbm{P}-a.s.$. 
\end{proposition}

\begin{proof}
Firstly we make clear that $\left\langle B,\phi \right\rangle_{\mathcal{H}}$ is understood as a process defined by the mapping $$e_i \times \omega \times t \rightarrow \left\langle B_t(e_i,\omega),\phi(\omega) \right\rangle_{\mathcal{H}}.$$ The fact that $\left\langle B,\phi \right\rangle_{\mathcal{H}} \in \bar{\mathcal{I}}^{\R}(\mathcal{W})$ is completely analogous to Proposition \ref{bounded take it in }, where the progressive measurability follows as the mapping 
$$\inner{B_{\cdot}}{\phi}_{\mathcal{H}}: t \times \omega \times \tilde{\omega} \rightarrow \inner{B_{t}(\omega)}{\phi(\tilde{\omega})}_{\mathcal{H}} $$ is $\mathcal{B}([0,T]) \times \mathcal{F}_T \times \mathcal{F}_0$ measurable as a mapping into $\mathscr{L}^2(\mathfrak{U};\R)$. Similarly we have that $$\left\Vert\inner{B_t(\omega)}{\phi(\omega)}_{\mathcal{H}}\right\Vert_{\mathscr{L}^2(\mathfrak{U};\R)} \leq \norm{\phi(\omega)}_{\mathcal{H}}\norm{B_t(\omega)}_{\mathscr{L}^2(\mathfrak{U};\mathcal{H})}$$ which is sufficient to justify that $\left\langle B,\phi \right\rangle_{\mathcal{H}} \in \bar{\mathcal{I}}^{\R}(\mathcal{W})$. The extension of Proposition \ref{unbounded take it in} to this result is then identical to the extension of Theorem \ref{operator through integral 1D} to \ref{operatorthroughstochasticintegral}, so we conclude the proof here. 

\end{proof}

\begin{proposition} \label{real valued bounded take it in cylindrical}
    Let $B \in \bar{\mathcal{I}}^{\mathcal{H}}(\mathcal{W})$ and $\eta: \Omega \rightarrow \R$ be $\mathcal{F}_0-$measurable. Then $\eta B \in \bar{\mathcal{I}}^{\mathcal{H}}(\mathcal{W})$ and for every $t > 0$ we have that $$\eta \int_0^tB_rd\mathcal{W}_r = \int_0^t \eta B_r d\mathcal{W}_r$$ $\mathbbm{P}-a.s.$. 
\end{proposition}

\begin{proof}
The proof follows identically to Proposition \ref{unbounded take it in 2} in analogy with Proposition \ref{real valued bounded take it in}. We make explicit that $\eta B$ is defined by the mapping $$e_i \times \omega \times t \rightarrow  \eta(\omega)B_t(e_i,\omega).$$
\end{proof}

\begin{remark}
Although we have not explicitly addressed the construction of the integral over a time interval $[s,t]$ where $s>0$, this can be done without any extra difficulty just as in the standard real valued case. If we were to just consider the integral over $[s,t]$ in Proposition \ref{unbounded take it in 2}, then the results extends to any $\mathcal{F}_s-$measurable $\phi, \eta$ in Propositions \ref{unbounded take it in 2}, \ref{real valued bounded take it in cylindrical}. To show this we of course revisit Proposition \ref{bounded take it in }, appreciating that the $\mathcal{F}_s-$measurability does not disturb the measurability requirements of the simple process.  
\end{remark}

We also extend the Stochastic Dominated Convergence Theorem to this setting.

\begin{lemma} \label{stochstic dominated convergence 2}
    Let $(B^n)$ be a sequence in $\bar{\mathcal{I}}^{\mathcal{H}}(\mathcal{W})$ such that there exists processes $B:\Omega \times [0,\infty) \rightarrow \mathscr{L}^2(\mathfrak{U};\mathcal{H})$ and $Q \in \bar{\mathcal{I}}^{\mathcal{H}}(\mathcal{W})$ with the properties that for every $T>0$, $ \mathbbm{P} \times \lambda  - a.e.$ $(\omega,t) \in \Omega \times [0,T]$:
    \begin{enumerate}
        \item \label{enum2} $\norm{B^n_t(\omega)}_{\mathscr{L}^2(\mathfrak{U};\mathcal{H})} \leq \norm{Q_t(\omega)}_{\mathscr{L}^2(\mathfrak{U};\mathcal{H})}$ for all $n \in \N$;
        \item $(B^n_t(\omega))$ is convergent to $B_t(\omega)$ in $\mathscr{L}^2(\mathfrak{U};\mathcal{H})$.
    \end{enumerate}
    Then $B \in \bar{\mathcal{I}}^{\mathcal{H}}(\mathcal{W})$ and for every $t>0$, there exists a subsequence indexed by $(n_k)$ such that \begin{equation}\label{endgame2}
    \lim_{n_k \rightarrow \infty}\int_0^tB^{n_k}_rd\mathcal{W}_r = \int_0^tB_r d\mathcal{W}_r\end{equation} $\mathbbm{P}-a.s.$. 
\end{lemma}

\begin{proof}
The proof is mechanically identical to that of Lemma \ref{stochstic dominated convergence}, simply replacing $\mathcal{H}$ by  $\mathscr{L}^2(\mathfrak{U};\mathcal{H})$ when dealing with the integrands and using the appropriate It\^{o} Isometry \ref{ito isom for cylindrical}.
\end{proof}

   We now shift attentions to results regarding the martingale properties of the integral.

\begin{proposition} \label{main martingale result}
    For $B \in \mathcal{I}^{\mathcal{H}}(\mathcal{W})$, the It\^{o} stochastic integral $$\int_0^tB(s)d\mathcal{W}_s$$ belongs to $\mathcal{M}^2_c(\mathcal{H}).$
\end{proposition}

\begin{proof}
This follows immediately from Propositions \ref{prop m2c closed} and \ref{1Dstochintismartingale}. 
\end{proof}

\begin{corollary} \label{main continuous local martingale result}
For $B \in \bar{\mathcal{I}}^{\mathcal{H}}(\mathcal{W})$, the It\^{o} stochastic integral $$\int_0^tB(s)d\mathcal{W}_s$$ is a continuous local martingale.
\end{corollary}

\begin{proof}
This now follows identically to Proposition \ref{integralislocalmartingale}.
\end{proof}

%Due to this martingale property, we have a Burkholder-Davis-Gundy type inequality for the stochastic integral. 

%\begin{theorem} \label{BDG}
 %   There exists a constant $c$ such that for any $B \in \bar{\mathcal{I}}^\mathcal{H}(\mathcal{W})$ and $t \geq 0$, we have that $$\mathbbm{E}\sup_{r \in [0,t]}\left\Vert \int_0^tB(s)d\mathcal{W}_s   \right\Vert_{\mathcal{H}} \leq c \mathbbm{E}\left(\int_0^t\norm{B(s)}^2_{\mathscr{L}^2(\mathfrak{U};\mathcal{H})}ds \right)^{\frac{1}{2}}.$$
%\end{theorem}

%\begin{remark}
%There is no a priori assumption on the finiteness of the expectation.
%\end{remark}

%For a proof and discussion of this result please see [\cite{da2014stochastic}] Theorem 4.36. 
The martingality of the integral also allows us to consider the quadratic variation as defined in Definition \ref{def of quad}.

\begin{proposition} \label{giving up}
    For $B \in \mathcal{I}^{\mathcal{H}}(\mathcal{W})$, we have that
    \begin{equation} \label{newestsestest}
        \left[\int_0^\cdot B_r d\mathcal{W}_r \right]_t = \int_0^t\norm{B_r}_{\mathscr{L}^2(\mathfrak{U};\mathcal{H})}^2dr.
    \end{equation}
\end{proposition}

\begin{proof}
At each time $t$, the integral $$\int_0^tB_r d\mathcal{W}_r $$ is defined to be the $L^2(\Omega;\mathcal{H})$ limit of the sequence
$$\sum_{i=1}^n\int_0^tB_r(e_i)dW^i_r.$$ We look to infer the quadratic variation of this sequence of processes using Proposition \ref{quad of int}, to then apply Lemma \ref{martingale quadratic variation limit}. We claim that \begin{equation}\nonumber \left[\sum_{i=1}^n\int_0^\cdot B_r(e_i)dW^i_r \right]_t = \int_0^t\sum_{i=1}^n\norm{B_r(e_i)}_{\mathcal{H}}^2dr\end{equation} 
which is to say
\begin{equation}\label{another new claim} \left\Vert\sum_{i=1}^n \int_0^{t}B_r(e_i) dW^i_r\right\Vert_{\mathcal{H}}^2 - \int_0^t \sum_{i=1}^n\norm{B_r(e_i)}_H^2dr\end{equation} 
is a martingale. For the orthonormal basis $(a_k)$ of $\mathcal{H}$,
\begin{align*}&\left\Vert\sum_{i=1}^n \int_0^{t}B_r(e_i) dW^i_r\right\Vert_{\mathcal{H}}^2 - \int_0^t \sum_{i=1}^n\norm{B_r(e_i)}_H^2dr\\ & \qquad = \sum_{k=1}^\infty \left\langle\sum_{i=1}^n \int_0^{t}B_r(e_i) dW^i_r, a_k \right\rangle^2_{\mathcal{H}} - \int_0^t\sum_{i=1}^n \norm{B_r(e_i)}_H^2dr\\
&\qquad = \sum_{k=1}^\infty\sum_{i=1}^n\sum_{j=1}^n \left\langle \int_0^{t}B_r(e_i) dW^i_r, a_k \right\rangle_{\mathcal{H}} \left\langle \int_0^{t}B_r(e_j) dW^j_r,a_k\right\rangle_{\mathcal{H}} - \int_0^t\sum_{i=1}^n \norm{B_r(e_i)}_H^2dr \\
&\qquad = \left(\sum_{k=1}^\infty\sum_{i=1}^n\left\langle \int_0^{t}B_r(e_i) dW^i_r, a_k \right\rangle_{\mathcal{H}}^2 - \int_0^t\sum_{i=1}^n \norm{B_r(e_i)}_H^2dr\right) \\& \qquad \qquad \qquad + \left(\sum_{k=1}^\infty\sum_{i\neq j} \left\langle \int_0^{t}B_r(e_i) dW^i_r, a_k \right\rangle_{\mathcal{H}} \left\langle \int_0^{t}B_r(e_j) dW^j_r,a_k\right\rangle_{\mathcal{H}}\right)\\
&\qquad = \sum_{i=1}^n\left(\left\Vert \int_0^{t}B_r(e_i) dW^i_r\right\Vert_{\mathcal{H}}^2 - \int_0^t \norm{B_r(e_i)}_H^2dr\right)\\& \qquad \qquad \qquad + \left(\sum_{k=1}^\infty\sum_{i\neq j} \left\langle \int_0^{t}B_r(e_i) dW^i_r, a_k \right\rangle_{\mathcal{H}} \left\langle \int_0^{t}B_r(e_j) dW^j_r,a_k\right\rangle_{\mathcal{H}}\right).
\end{align*}
Inspecting the last equality, by Proposition \ref{quad of int} we have that $$ \sum_{i=1}^n\left(\left\Vert \int_0^{t}B_r(e_i) dW^i_r\right\Vert_{\mathcal{H}}^2 - \int_0^t \norm{B_r(e_i)}_H^2dr\right)$$ is a finite sum of martingales, so we would prove that the process defined in (\ref{another new claim}) also belongs to this class if we show that the same is true of \begin{equation}\label{true of} \sum_{k=1}^\infty\sum_{i\neq j} \left\langle \int_0^{t}B_r(e_i) dW^i_r, a_k \right\rangle_{\mathcal{H}} \left\langle \int_0^{t}B_r(e_j) dW^j_r,a_k\right\rangle_{\mathcal{H}}.\end{equation} We consider the above for each fixed $k$, rewriting it as \begin{equation}\label{each fixed k}\sum_{i\neq j} \left(\int_0^{t}\left\langle B_r(e_j),a_k\right\rangle_{\mathcal{H}}  dW^j_r\right)\left(\int_0^{t}\left\langle B_r(e_i),a_k\right\rangle_{\mathcal{H}}  dW^i_r \right).\end{equation}
Adaptedness of this process is clear, and to show integrability observe that
\begin{align*}
    &\mathbbm{E}\left\vert\sum_{i\neq j} \left(\int_0^{t}\left\langle B_r(e_j),a_k\right\rangle_{\mathcal{H}}  dW^j_r\right)\left(\int_0^{t}\left\langle B_r(e_i),a_k\right\rangle_{\mathcal{H}}  dW^i_r \right)\right\vert
    \\ &\qquad \qquad \leq \sum_{i\neq j} \mathbbm{E}\left\vert \left(\int_0^{t}\left\langle B_r(e_j),a_k\right\rangle_{\mathcal{H}}  dW^j_r\right)\left(\int_0^{t}\left\langle B_r(e_i),a_k\right\rangle_{\mathcal{H}}  dW^i_r \right)\right\vert
    \\ &\qquad \qquad \leq \sum_{i\neq j} \left(\mathbbm{E}\left\vert\int_0^{t}\left\langle B_r(e_i),a_k\right\rangle_{\mathcal{H}}  dW^i_r \right\vert^2   \right)^{1/2}\left(\mathbbm{E}\left\vert\int_0^{t}\left\langle B_r(e_j),a_k\right\rangle_{\mathcal{H}}  dW^j_r \right\vert^2    \right)^{1/2}\\
    &\qquad \qquad = \sum_{i\neq j} \left(\mathbbm{E}\int_0^{t}\left\langle B_r(e_i),a_k\right\rangle_{\mathcal{H}}^2  dr  \right)^{1/2}\left(\mathbbm{E}\int_0^{t}\left\langle B_r(e_j),a_k\right\rangle_{\mathcal{H}}^2  dr    \right)^{1/2}
    \\&\qquad \qquad \leq n^2\mathbbm{E}\int_0^{t}\sum_{i=1}^n\left\langle B_r(e_i),a_k\right\rangle_{\mathcal{H}}^2  dr\\
    &\qquad \qquad \leq n^2\mathbbm{E}\int_0^{t}\sum_{i=1}^n\norm{ B_r(e_i)}_{\mathcal{H}}^2  dr\\
    & \qquad \qquad< \infty.
\end{align*}
As for the martingale property, for any times $s<t$,
\begin{align}
    \nonumber &\mathbbm{E}\left(\sum_{i\neq j} \left(\int_s^{t}\left\langle B_r(e_j),a_k\right\rangle_{\mathcal{H}}  dW^j_r\right)\left(\int_s^{t}\left\langle B_r(e_i),a_k\right\rangle_{\mathcal{H}}  dW^i_r \right) \bigg\vert \mathcal{F}_s\right)\\ 
    \nonumber &\qquad \qquad = \mathbbm{E}\left(\sum_{i\neq j} \left(\int_s^{t}\left\langle B_r(e_j),a_k\right\rangle_{\mathcal{H}}  dW^j_r\right)\left(\int_s^{t}\left\langle B_r(e_i),a_k\right\rangle_{\mathcal{H}}  dW^i_r \right) \right)\\
    &\qquad \qquad =0 \label{its equal to zero}
\end{align}
where passage from the first line to the second is through the independent increments property of the Brownian Motions, and the second to the third is from the independence of the Brownian Motions. 
So the process defined in (\ref{each fixed k}) is shown to be a martingale, where we wish to show that this property remains true in the limit of the infinite sum for (\ref{true of}). Convergence of the infinite sum is defined $\mathbbm{P}-a.s.$, and it is sufficient to show that the convergence also holds in $L^1(\Omega;\R)$. For this we show that the sequence is Cauchy in $L^1(\Omega;\R)$, taking the difference of the $l^{th}$ and $m^{th}$ terms to see that 
\begin{align*}
    &\mathbbm{E}\left\vert\sum_{k=m+1}^l\sum_{i\neq j} \left(\int_0^{t}\left\langle B_r(e_j),a_k\right\rangle_{\mathcal{H}}  dW^j_r\right)\left(\int_0^{t}\left\langle B_r(e_i),a_k\right\rangle_{\mathcal{H}}  dW^i_r \right)\right\vert
    \\ &\qquad \qquad \leq \sum_{k=m+1}^l\mathbbm{E}\left\vert\sum_{i\neq j} \left(\int_0^{t}\left\langle B_r(e_j),a_k\right\rangle_{\mathcal{H}}  dW^j_r\right)\left(\int_0^{t}\left\langle B_r(e_i),a_k\right\rangle_{\mathcal{H}}  dW^i_r \right)\right\vert
    \\& \qquad \qquad \leq n^2\sum_{k=m+1}^l\mathbbm{E}\int_0^{t}\sum_{i=1}^n\left\langle B_r(e_i),a_k\right\rangle_{\mathcal{H}}^2  dr
    \\& \qquad \qquad \leq n^2\sum_{k=m+1}^\infty \mathbbm{E}\int_0^{t}\sum_{i=1}^n\left\langle B_r(e_i),a_k\right\rangle_{\mathcal{H}}^2  dr
\end{align*}
which is a monotone decreasing sequence to zero in $m$, hence the Cauchy property is shown so there exists a limit in $L^1(\Omega;\R)$ which must agree with the $\mathbbm{P}-a.s.$ limit (we can take a $\mathbbm{P}-a.s.$ convergent subsequence from the $L^1(\Omega;\R)$ convergence) and the martingale property of the process defined in (\ref{true of}) and hence (\ref{another new claim}) is shown. We thus apply Lemma \ref{martingale quadratic variation limit} and deduce that $\left[\int_0^\cdot B_r d\mathcal{W}_r\right]_t$ is the $L^1(\Omega;\R)$ limit of the sequence $$\int_0^t \sum_{i=1}^n\norm{B_r(e_i)}_{\mathcal{H}}^2dr$$ in $n$. Similarly to the analysis just conducted we can show that this sequence is Cauchy in $L^1(\Omega;\R)$ and agrees with the $\mathbbm{P}-a.s.$ limit, which is of course $$\int_0^t\norm{B_r}_{\mathscr{L}^2(\mathfrak{U};\mathcal{H})}^2dr $$ taking the infinite sum through the integral with either Tonelli's Theorem (identifying the infinite sum as a integral with respect to the counting measure) or the Monotone Convergence Theorem. The proof is concluded.
\end{proof}

We also have the analogous result to Proposition \ref{analagous one}.

\begin{proposition} \label{analagous two}
    Let $B \in \mathcal{I}^{\mathcal{H}}_T(\mathcal{W})$ and consider any sequence of partitions $$I_l:= \left\{0=t^l_0 < t^l_1 < \dots < t^l_{k_l}=T\right\}$$ with $\max_j\vert t^l_{j}-t^l_{j-1} \vert \rightarrow 0$ as $l \rightarrow \infty$. Then for all $t\in[0,T]$, for any $\varepsilon > 0$,
    \begin{equation}\label{the problem}\lim_{l \rightarrow \infty}\mathbbm{P}\left(\left\{\left\vert \sum_{t^l_{j+1} \leq t}\left\Vert\int_{t^l_{j}}^{t^l_{j+1}}B_r d\mathcal{W}_r\right\Vert_{\mathcal{H}}^2 - \int_0^t\norm{B_r}_{\mathscr{L}^2(\mathfrak{U};\mathcal{H})}^2dr \right\vert > \varepsilon \right\}\right) = 0. \end{equation}
\end{proposition}

\begin{proof}
Following the method used in Proposition \ref{analagous one} we again would like to reduce this to a familiar case and extrapolate the result to the limit. We introduce a sequence of stopping times $(\tau^n)$ defined at every $n \in \N$ by
$$\tau^n := n \wedge \inf\left\{ t \in [0,T]: \left\Vert\int_{0}^tB_r d\mathcal{W}_r\right\Vert_{\mathcal{H}}^2 + \int_0^t\norm{B_r}_{\mathscr{L}^2(\mathfrak{U};\mathcal{H})}^2dr \geq n \right\}.$$
For every $n$ we define the process $$B^n_{\cdot} := B_{\cdot}\mathbbm{1}_{\cdot \leq \tau^n}$$
and now look to show that 
\begin{equation} \label{new look to show that 2} \lim_{l \rightarrow \infty}\mathbbm{E}\left(\left\vert \sum_{t^l_{j+1} \leq t}\left\Vert\int_{t^l_{j}}^{t^l_{j+1}}B^n_r d\mathcal{W}_r\right\Vert_{\mathcal{H}}^2 - \int_0^t\norm{B^n_r}_{\mathscr{L}^2(\mathfrak{U};\mathcal{H})}^2dr \right\vert \right) = 0.\end{equation}
This is precisely in line with the method of Proposition \ref{analagous one}. We have that 
\begin{align*}
    &\mathbbm{E}\left(\left\vert \sum_{t^l_{j+1} \leq t}\left\Vert\int_{t^l_{j}}^{t^l_{j+1}}B^n_r d\mathcal{W}_r\right\Vert_{\mathcal{H}}^2 - \int_0^t\norm{B^n_r}_{\mathscr{L}^2(\mathfrak{U};\mathcal{H})}^2dr \right\vert \right)
    \\ & \qquad \qquad \qquad \qquad = \mathbbm{E}\left(\left\vert \sum_{t^l_{j+1} \leq t}\left\Vert\int_{t^l_{j}}^{t^l_{j+1}}B^n_r d\mathcal{W}_r\right\Vert_{\mathcal{H}}^2 -  \int_0^t\sum_{i=1}^\infty\norm{B^n_r(e_i)}_{\mathcal{H}}^2dr  \right\vert \right)
     \\ & \qquad \qquad \qquad \qquad = \mathbbm{E}\left(\left\vert \sum_{t^l_{j+1} \leq t}\sum_{k=1}^\infty\left\langle\int_{t^l_{j}}^{t^l_{j+1}}B^n_rd\mathcal{W}_r,a_k\right\rangle_{\mathcal{H}} ^2 -  \int_0^t\sum_{i=1}^\infty\sum_{k=1}^\infty\inner{B^n_r(e_i)}{a_k}_{\mathcal{H}}^2dr  \right\vert \right)
    \\ & \qquad \qquad \qquad \qquad = \mathbbm{E}\left(\left\vert \sum_{t^l_{j+1} \leq t}\sum_{k=1}^\infty\left(\int_{t^l_{j}}^{t^l_{j+1}}\inner{B^n_r}{a_k}_{\mathcal{H}}d\mathcal{W}_r\right) ^2 -  \int_0^t\sum_{k=1}^\infty\sum_{i=1}^\infty\inner{B^n_r(e_i)}{a_k}_{\mathcal{H}}^2dr  \right\vert \right)
    \\ & \qquad \qquad \qquad \qquad = \mathbbm{E}\left(\left\vert \sum_{t^l_{j+1} \leq t}\sum_{k=1}^\infty\left(\int_{t^l_{j}}^{t^l_{j+1}}\inner{B^n_r}{a_k}_{\mathcal{H}}d\mathcal{W}_r\right) ^2 -  \int_0^t\sum_{k=1}^\infty\norm{\inner{B^n_r(e_i)}{a_k}_{\mathcal{H}}}_{\mathscr{L}^2(\mathfrak{U};\R)}^2dr  \right\vert \right)\\
    & \qquad \qquad \qquad \qquad \leq \sum_{k=1}^\infty\mathbbm{E}\left(\left\vert \sum_{t^l_{j+1} \leq t}\left(\int_{t^l_{j}}^{t^l_{j+1}}\inner{B^n_r}{a_k}_{\mathcal{H}}d\mathcal{W}_r\right) ^2 -  \int_0^t\norm{\inner{B^n_r(e_i)}{a_k}_{\mathcal{H}}}_{\mathscr{L}^2(\mathfrak{U};\R)}^2dr  \right\vert \right)
\end{align*}
having applied Theorem \ref{operatorthroughstochasticintegral} and the Dominated Convergence Theorem to take the infinite sum in $k$ through the time integral and expectation. From Theorem \ref{operatorthroughstochasticintegral} and Proposition \ref{main martingale result} then $\int_0^\cdot\inner{B^n_r}{a_k}_{\mathcal{H}}d\mathcal{W}_r $ belongs to $\mathcal{M}^2_c$, with quadratic variation $\int_0^t\norm{\inner{B_r(e_i)}{a_k}_{\mathcal{H}}}_{\mathscr{L}^2(\mathfrak{U};\R)}^2dr$ coming from Proposition \ref{giving up}. Just as we used in Proposition \ref{analagous one}, we have that for each fixed $k \in \N$, $$\lim_{l \rightarrow \infty}\mathbbm{E}\left(\left\vert \sum_{t^l_{j+1} \leq t}\left(\int_{t^l_{j}}^{t^l_{j+1}}\inner{B^n_r}{a_k}_{\mathcal{H}}d\mathcal{W}_r\right) ^2 -  \int_0^t\norm{\inner{B^n_r(e_i)}{a_k}_{\mathcal{H}}}_{\mathscr{L}^2(\mathfrak{U};\R)}^2dr  \right\vert \right) = 0$$ so it is sufficient to justify the interchange of limit in $l$ and summation in $k$. This follows identically to the justification in Proposition \ref{analagous one}, appealing this time to the It\^{o} Isometry \ref{ito isom for cylindrical}. 
\end{proof}

\newpage

\section{Stochastic Differential Equations in Infinite Dimensions} \label{section 2}

Throughout this section we shall again use $(e_i)$ as an orthonormal basis of $\mathfrak{U}$, the space over which $\mathcal{W}$ is a Cylindrical Brownian Motion (recall (\ref{cylindricalBM})), and $(a_k)$ an orthonormal basis of the Hilbert Space in which the equation takes place. 

%In this section we consider an abstract framework to define an SPDE, which will be general enough to include the highly involved SALT derived fluid dynamics models as discussed in the introduction.

\subsection{The Stratonovich Integral} \label{sub strato}

We look at first to define the Stratonovich Integral with respect to a one dimensional martingale, before then doing with respect to a Cylindrical Brownian Motion. %We make the definition here only for integrands which are martingales in the Hilbert Space, though we note that this can be extended to a wider class of processes by considering the cross-variation defined as a limit in probability over nested partitions in the traditional way.

\begin{definition} \label{one d strat}
For $M \in \bar{\mathcal{M}}^2_c$ and $\sy \in \mathcal{I}^{\mathcal{H}}_{M} \cap \bar{\mathcal{M}}^2_c(\mathcal{H})$, the Stratonovich stochastic integral is defined as $$\int_0^t\sy_s \circ dM_s := \int_0^t\sy_s dM_s + \frac{1}{2}[\sy,M]_t.$$
\end{definition}

\begin{definition} \label{regular cylindrical strat}
For $B \in \mathcal{I}^{\mathcal{H}}(\mathcal{W})$ such that $B_{e_i} \in \bar{\mathcal{M}}^2_c({\mathcal{H}})$ for every $e_i$ and the limit $$\sum_{i=1}^\infty[B_{e_i},W^i]_t $$ is well defined in $L^2\left(\Omega;\mathcal{H}\right)$, the Stratnovich stochastic integral is defined as $$\int_0^tB(s) \circ d\mathcal{W}_s := \sum_{i=1}^\infty \left(\int_0^tB_{e_i}(s) dW^i_s + \frac{1}{2}[B_{e_i},W^i]_t\right)$$ where the limit is taken in $L^2\left(\Omega;\mathcal{H}\right)$. The class of such processes will be denoted $\mathcal{I}^{\mathcal{H}}_{\circ}(\mathcal{W}).$
\end{definition}

It will be necessary to extend this definition to processes $B \in \bar{\mathcal{I}}^{\mathcal{H}}(\mathcal{W})$, but we encounter more technical issues in trying to construct a sequence of stopping times such that the stopped process belongs to $\mathcal{I}^{\mathcal{H}}_{\circ}(\mathcal{W})$. We find it simplest to give the definition below.

\begin{definition} \label{irregular cylindrical strat}
Suppose that there exists a sequence of stopping times $(\tau_n)$ which are $\mathbbm{P}-a.s.$ monotone increasing and convergent to infinity such that:
\begin{enumerate}
    \item For every $n$, the process $$B^n(\cdot):=B({\cdot})\mathbbm{1}_{\cdot \leq \tau^n}$$
    belongs to $\mathcal{I}^{\mathcal{H}}(\mathcal{W})$;
 \item For every $n$ and $i$, the process 
 $$ B^{\tau_n}_{e_i}(\cdot):= B_{e_i}(\cdot \wedge \tau^n)$$
 belongs to $\bar{\mathcal{M}}^2_c(\mathcal{H})$;
 \item The limit $$\sum_{i=1}^\infty [B^{\tau_n}_{e_i},W^i]_t$$ is well defined in $L^2\left(\Omega;\mathcal{H}\right)$.
\end{enumerate}
Then the Stratonovich stochastic integral is defined at a fixed $\omega$ for any $\tau^n(\omega) \geq t$ as $$\left(\int_0^tB(s) \circ d\mathcal{W}_s\right)(\omega) := \left(\sum_{i=1}^\infty \left(\int_0^tB^n_{e_i}(s) dW^i_s + \frac{1}{2}[B^{\tau_n}_{e_i},W^i]_t\right)\right)(\omega)$$ where the limit is taken in $L^2\left(\Omega;\mathcal{H}\right)$. The class of such processes will be denoted $\bar{\mathcal{I}}^{\mathcal{H}}_{\circ}(\mathcal{W}).$
\end{definition}

This definition is of course completely analogous to the localisation procedure used in the previous constructions, except we postulate in the first instance the existence of the localising sequence $(\tau_n)$.

\subsection{Strong Solutions in the Abstract Framework} \label{subs 2.2}

We work with a quartet of embedded Hilbert Spaces $$V \hookrightarrow H \hookrightarrow U \hookrightarrow X$$ where the embedding is meant as a continuous linear injection. We introduce at first the It\^{o} SPDE
\begin{equation} \label{thespde}
    \sy_t = \sy_0 + \int_0^t \mathcal{Q}\sy_sds + \int_0^t\mathcal{G}\sy_sd\mathcal{W}_s
\end{equation}
where $\mathcal{W}$ continues to be a Cylindrical Brownian Motion over $\mathfrak{U}$ relative to our fixed filtered probability space $(\Omega,\mathcal{F},(\mathcal{F}_t),\mathbbm{P})$ with representation (\ref{cylindricalBM}). We impose now some conditions on the operators relative to these spaces. To do so we define the general operator $\tilde{K}: H \rightarrow \R$ by $$\tilde{K}(\phi):= c\left(1 + \norm{\phi}_U^p + \norm{\phi}_H^q\right)$$ for any constants $c,p,q$ independent of $\phi$. 
 \begin{assumption} \label{Qassumpt}
$\mathcal{Q}: V \rightarrow U$ is measurable and for any $\phi \in V$, $$\norm{\mathcal{Q}\phi}_U \leq \tilde{K}(\phi)[1 + \norm{\phi}_V^2].$$
 \end{assumption}
 
\begin{assumption} \label{Gassumpt}
$\mathcal{G}$ is understood as a measurable operator \begin{align*}
    \mathcal{G}&: V \rightarrow \mathscr{L}^2(\mathfrak{U};H)\\
     \mathcal{G}&: H \rightarrow \mathscr{L}^2(\mathfrak{U};U)\\
       \mathcal{G}&: U \rightarrow \mathscr{L}^2(\mathfrak{U};X)
\end{align*}
defined over $\mathfrak{U}$ by its action on the basis vectors $$\mathcal{G}(\cdot, e_i):= \mathcal{G}_i(\cdot).$$ Each $\mathcal{G}_i$ is linear and there exists constants $c_i$ such that for all $\phi \in V$, $\psi \in H$, $\eta \in U$:
\begin{align*}
    \norm{\mathcal{G}_i\phi}_{H} &\leq c_i \norm{\phi}_V\\
    \norm{\mathcal{G}_i\psi}_{U} &\leq c_i \norm{\psi}_H\\
    \norm{\mathcal{G}_i\eta}_{X} &\leq c_i \norm{\eta}_U\\
    \sum_{i=1}^\infty c_i^2 &< \infty.
\end{align*}
\end{assumption}

It is worth clarifying how $\mathcal{G}$ is defined over $\mathfrak{U}$: fix a $\phi \in V$ and consider $\mathcal{G}(\phi,\cdot):\mathfrak{U} \rightarrow H$ (the arguments here apply for the larger spaces as well). Any $\alpha \in \mathfrak{U}$ has the representation $$\sum_{i=1}^\infty \inner{\alpha}{e_i}_{\mathfrak{U}}e_i$$ where $$\sum_{i=1}^\infty \inner{\alpha}{e_i}_{\mathfrak{U}}^2< \infty.$$ Then $$\mathcal{G}(\phi,\cdot):\alpha \mapsto \sum_{i=1}^\infty\inner{\alpha}{e_i}_{\mathfrak{U}}\mathcal{G}_i\phi$$ is well defined as an element of $H$. This is justified by showing that the sequence of partial sums is Cauchy in $H$: note that from Cauchy-Schwarz,
$$\left\Vert\sum_{i=m}^n\inner{\alpha}{e_i}_{\mathfrak{U}}\mathcal{G}_i\phi  \right\Vert_{H}^2 \leq \left( \sum_{i=m}^n\inner{\alpha}{e_i}_{\mathfrak{U}}^2\right)\left(\sum_{i=m}^n\norm{\mathcal{G}_i\phi}^2_{H} \right) \leq \left( \sum_{i=m}^\infty\inner{\alpha}{e_i}_{\mathfrak{U}}^2\right)\left(\sum_{i=m}^\infty c_i^2\norm{\phi}_V^2 \right)$$ which approaches zero as $m \rightarrow \infty$ as the sums are finite. We introduce now the first notion of our strong solutions. %and frame this definition for a specified initial condition, which is an $\mathcal{F}_0-$measurable mapping $\sy_0: \Omega \rightarrow H$. Measurability here is again defined with respect to the Borel Sigma Algebra on $H$.

\begin{definition} \label{definitionofregularsolution}
Let $\sy_0:\Omega \rightarrow H$ be $\mathcal{F}_0-$measurable. A pair $(\sy,\tau)$ where $\tau$ is a $\mathbbm{P}-a.s.$ positive stopping time and $\sy$ is a process such that for $\mathbbm{P}-a.e.$ $\omega$, $\sy_{\cdot}(\omega) \in C\left([0,T];H\right)$ and $\sy_{\cdot}(\omega)\mathbbm{1}_{\cdot \leq \tau(\omega)} \in L^2\left([0,T];V\right)$ for all $T>0$ with $\sy_{\cdot}\mathbbm{1}_{\cdot \leq \tau}$ progressively measurable in $V$, is said to be a local strong solution of the equation (\ref{thespde}) if the identity
\begin{equation} \label{identityindefinitionoflocalsolution}
    \sy_{t} = \sy_0 + \int_0^{t\wedge \tau} \mathcal{Q}\sy_sds + \int_0^{t \wedge \tau}\mathcal{G}\sy_s d\mathcal{W}_s
\end{equation}
holds $\mathbbm{P}-a.s.$ in $U$ for all $t \geq 0$.
\end{definition}

Let's take a few moments to process this definition and ensure that the integrals make sense with the given regularity of the solution and the operators $\mathcal{Q}$, $\mathcal{G}$. As an initial aside note that if $(\sy,\tau)$ is a $V$-valued local strong solution of the equation (\ref{thespde}), then $\sy_\cdot = \sy_{
\cdot \wedge \tau}$. Moreover the progressive measurability condition on $\sy_{\cdot}\mathbbm{1}_{\cdot \leq \tau}$ may look a little suspect as $\sy_0$ itself may only belong to $H$ and not $V$ making it impossible for $\sy_{\cdot}\mathbbm{1}_{\cdot \leq \tau}$ to be even adapted in $V$. We are mildly abusing notation here; what we really ask is that there exists a process $\py$ which is progressively measurable in $V$ and such that $\py_{\cdot} = \sy_{\cdot}\mathbbm{1}_{\cdot \leq \tau}$ almost surely over the product space $\Omega \times [0,T]$ for every $T \geq 0$ with product measure $\mathbbm{P}\times \lambda$.\\

The time integral in (\ref{identityindefinitionoflocalsolution}) is well defined in $U$ as a Bochner Integral: first of all $\sy_{\cdot}(\omega)\mathbbm{1}_{\cdot \leq \tau(\omega)} \in L^2\left([0,T];V\right)$ for $\mathbbm{P}-a.e.$ $\omega$ so $\sy_{\cdot}(\omega)\mathbbm{1}_{\cdot \leq \tau(\omega)}:[0,t]\rightarrow V$ is measurable hence $\mathcal{Q}(\sy_{\cdot}(\omega)\mathbbm{1}_{\cdot \leq \tau(\omega))}:[0,t] \rightarrow U$ is measurable from Assumption \ref{Qassumpt}. Moreover the mapping $\mathcal{Q}(\sy_{\cdot}(\omega)\mathbbm{1}_{\cdot \leq \tau(\omega)})\mathbbm{1}_{\cdot \leq \tau(\omega)}:[0,t] \rightarrow U$ is again measurable and we have that $$\int_0^{t\wedge \tau(\omega)} \mathcal{Q}\left(\sy_s(\omega)\right)ds = \int_0^{t} \mathcal{Q}\left(\sy_s(\omega)\mathbbm{1}_{s\leq \tau(\omega)}\right)\mathbbm{1}_{s \leq \tau(\omega)}ds$$ so the required measurability in order to define the integral is satisfied. In this vein we have that \begin{align*}\int_0^{t\wedge \tau(\omega)} \norm{\mathcal{Q}\left(\sy_s(\omega)\right)}_Uds &\leq \int_0^{t\wedge \tau(\omega)} \tilde{K}\left(\sy_s(\omega)\right)\left[1 + \norm{\sy_s(\omega)}_V^2 \right]ds\\ &\leq \sup_{s \in [0,t\wedge \tau(\omega)]}\left[  \tilde{K}\left(\sy_s(\omega)\right)\right]\int_0^{t\wedge \tau(\omega)}1+\norm{\sy_s(\omega)}_V^2 ds\\
&< \infty
\end{align*} employing Assumption \ref{Qassumpt} again and using the regularity specified by the local solution to deduce finiteness, which justifies that the integral is well defined. \\

As for the stochastic integral in (\ref{identityindefinitionoflocalsolution}), this is well defined in the sense of Definition \ref{cylindricalintlocal} in $H$ (which then embeds into $U$), though it is done so formally via the process $\py$ stipulated in the above remark regarding progressive measurability. We understand again that $$\int_0^{t \wedge \tau}\mathcal{G} \sy_s d\mathcal{W}_s = \int_0^{t}\mathcal{G} (\sy_s\mathbbm{1}_{s \leq \tau})\mathbbm{1}_{s\leq \tau} d\mathcal{W}_s$$ should the integral be well defined, and so we actually define the integral by $$\int_0^{t \wedge \tau}\mathcal{G} \sy_s d\mathcal{W}_s := \int_0^{t}\mathcal{G} \py_s\mathbbm{1}_{s\leq \tau} d\mathcal{W}_s.$$ The progressive measurability of $\mathcal{G}\py\mathbbm{1}_{\cdot \leq \tau}:[0,t] \times \Omega \rightarrow \mathscr{L}^2(\mathfrak{U};H)$ is immediate from the measurability assumption of \ref{Gassumpt}, the same arguments above for the time integral and the progressive measurability of $\mathbbm{1}_{\cdot \leq \tau}$ coming from the definition of the stopping time. Similarly we have that \begin{align*} \int_0^{t \wedge \tau}\sum_{i=1}^\infty\norm{\mathcal{G}_i\left(\py_s(\omega)\right)}_H^2ds = \int_0^{t \wedge \tau}\sum_{i=1}^\infty\norm{\mathcal{G}_i\left(\sy_s(\omega)\mathbbm{1}_{s \leq \tau}\right)}_H^2ds < \infty
\end{align*}
$\mathbbm{P}-a.s.$, using again Assumption \ref{Gassumpt}, validating that this term is well defined as a local martingale in $H$ and thus in $U$ from the continuous embedding. We make no claims that this is genuinely a square integrable martingale.\\

We treat the local solution here to deal with the additional technicalities which arise from accounting for the stopping time. Global solutions are similarly defined however.

\begin{definition} \label{definitionofregularsolutionglobal}
Let $\sy_0: \Omega \rightarrow H$ be $\mathcal{F}_0-$measurable. A process $\sy$ such that for $\mathbbm{P}-a.e.$ $\omega$, $\sy_{\cdot}(\omega) \in C\left([0,T];H\right) \cap L^2\left([0,T];V\right)$ for all $T>0$ with $\sy$ progressively measurable in $V$, is said to be a strong solution of the equation (\ref{thespde}) if the identity (\ref{thespde}) holds $\mathbbm{P}-a.s.$ in $U$ for all $t \geq 0$.
\end{definition}

\subsection{Stratonovich SPDEs in the Abstract Framework} \label{subs 2.3}

We work now with the same initial condition $\sy_0$ and operators $\mathcal{Q}$, $\mathcal{G}$ but instead pose the question of how to understand the Stratonovich SPDE
\begin{equation} \label{stratoSPDE}
    \sy_t = \sy_0 + \int_0^t \mathcal{Q}\sy_sds + \int_0^t\mathcal{G}\sy_s \circ d\mathcal{W}_s.
\end{equation}

We are slightly hesitant to define our strong solution for the equation (\ref{stratoSPDE}) as understanding the Stratonovich integral is delicate given the necessary semi-martingale structure. If we can show that $\sy$ satisfies an evolution equation of an It\^{o} stochastic integral plus a time integral then we can deduce the required semi-martingality and identify the martingale part, but we would need this assumption on semi-martingality \textit{a priori} to define the Stratonovich integral in the sense of \ref{irregular cylindrical strat} in order to show the desired representation. To this end, we identify a Stratonovich SPDE with an It\^{o} one in the sense given here.

%We tackle this by defining the global strong solution analogous to \ref{definitionofregularsolutionglobal}, as we don't wish to obfuscate the key ideas with the technicalities of the local time. It should be immediately noted however that the definition can be amended for the local strong solution in the corresponding way to \ref{definitionofregularsolution}. 

%\begin{definition}
%A process $\sy$ such that for $\mathbbm{P}-a.e.$ $\omega$, $\sy_{\cdot}(\omega) \in C\left([0,T];H\right) \cap L^2\left([0,T];V\right)$ for all $T>0$ with $\sy$ progressively measurable in $V$, is said to be a strong solution of the equation (\ref{stratoSPDE}) if the identity
%\begin{equation} \label{identityindefinitionofglobalsolutionstrato}
 %   \sy_{t} = \sy_0 + \int_0^{t} \mathcal{Q}\sy_sds + \int_0^{t}\mathcal{G}\sy_s \circ d\mathcal{W}_s
%\end{equation}
%holds $\mathbbm{P}-a.s.$ in $X$ for all $t \geq 0$.
%\end{definition}

\begin{theorem} \label{the conversion}
    If $\sy$ is a strong solution of the equation \begin{equation} \label{the identified one}
         \sy_t = \sy_0 + \int_0^t\left( \mathcal{Q}\sy_s + \frac{1}{2}\sum_{i=1}^\infty \mathcal{G}_i^2\sy_s\right)ds +  \int_0^t\mathcal{G}\sy_sd\mathcal{W}_s
    \end{equation}
    then $\sy$ satisfies the identity (\ref{stratoSPDE}) $\mathbbm{P}-$ $a.s.$ in $X$ for all $t \geq 0$. 
\end{theorem}

There is a little to unpack here before going on to the proof of this result. The first is how we understand the infinite sum of (\ref{the identified one}) and subsequently the SPDE. The operator $$\sum_{i=1}^\infty \mathcal{G}_i^2: V \rightarrow U$$ is defined as the pointwise limit of the partial sums, which is well defined as for any fixed $\phi \in V$, \begin{equation} \label{as seen in }\left\Vert \sum_{i=m}^n\mathcal{G}_i^2\phi \right\Vert_U \leq  \sum_{i=m}^n\left\Vert\mathcal{G}_i^2\phi \right\Vert_U \leq \sum_{i=m}^nc_i\left\Vert\mathcal{G}_i\phi \right\Vert_H \leq \sum_{i=m}^nc_i^2\left\Vert\phi \right\Vert_V\end{equation} which as seen before approaches zero as $m \rightarrow \infty$. To understand the strong solution as defined in Definition  \ref{definitionofregularsolutionglobal} we need to show that the new operator $\mathcal{Q} + \frac{1}{2}\sum_{i=1}^\infty \mathcal{G}_i^2$ satisfies the assumptions postulated in Assumption \ref{Qassumpt}, but this is clear as $\mathcal{G}_i: V \rightarrow H$ and $\mathcal{G}_i: H \rightarrow U$ are bounded linear hence continuous so measurable, thus too is $\mathcal{G}_i^2: V \rightarrow U$ and therefore the partial sums and the pointwise limit are as well. Moreover $$\left\Vert \sum_{i=1}^\infty \mathcal{G}_i^2\phi \right\Vert_U \leq \sum_{i=1}^\infty c_i^2\norm{\phi}_V$$ as seen in (\ref{as seen in }) so the boundedness is also satisfied so we can understand the SPDE (\ref{the identified one}) in the same manner as (\ref{thespde}). It also remains to be checked that the Stratonovich integral of (\ref{stratoSPDE}) is well defined for $\sy$ a strong solution of (\ref{the identified one}). We show this in the sense of Definition \ref{irregular cylindrical strat} for $\mathcal{H}=X$, the space in which the identity is satisfied. We show that the sequence of stopping times $$\tau_n:= n \wedge \inf\left\{s \geq 0: \int_0^s\norm{\sy_r}_V^2dr \geq n\right\}$$ fit the requirements of Definition \ref{irregular cylindrical strat}, noting immediately that the sequence is $\mathbbm{P}-a.s.$ monotone increasing and convergent to infinity as the process $$s \mapsto \int_0^s\norm{\sy_r}_V^2dr$$ is continuous. Moreover the process $\mathcal{G}\sy^n := \mathcal{G}(\sy_{\cdot}\mathbbm{1}_{\cdot \leq \tau_n}) = \mathcal{G}\sy_{\cdot}\mathbbm{1}_{\cdot \leq \tau_n}$ (using linearity of $\mathcal{G}$) is progressively measurable in $H$ (we remark again that this is really $\mathcal{G}\py^n$ as stipulated but we identify the two) and satisfies the bound $$\mathbbm{E}\int_0^t\sum_{i=1}^\infty\norm{\mathcal{G}_i\sy^n_r}_H^2dr \leq \mathbbm{E}\sum_{i=1}^\infty c_i^2\int_0^t\norm{\sy^n_r}_V^2 \leq \sum_{i=1}^\infty c_i^2 n < \infty$$ freely applying Tonelli's Theorem between the expectation, integral and sum. Thus $\mathcal{G}\sy^n \in \mathcal{I}^H(\mathcal{W})$ so the integral can be constructed in $H$ and then embedded into $X$, but we note that the embedding $J:H \xrightarrow{} X$ is a continuous linear operator and so from Corollary \ref{operatorthroughstochasticintegral2} then $J(\mathcal{G}\sy^n) \in \mathcal{I}^X(\mathcal{W})$ and in particular $$J\left(\int_0^t \mathcal{G}\sy^n_s d\mathcal{W}_s \right) = \int_0^t J(\mathcal{G}\sy^n_s) d\mathcal{W}_s$$ so there is no ambiguity in how we understand the integral as an element of $X$. Indeed we simply make the identification $\mathcal{G}\sy^n$ with $J(\mathcal{G}\sy^n)$ and will make no explicit reference to the embeddings in our analysis henceforth. As for showing that $\mathcal{G}_i\sy^{\tau_n} \in \bar{\mathcal{M}}^2_c(X)$, we look at the evolution equation satisfied by $\sy^{\tau_n}$ which is
$$\sy^{\tau_n}_t = \sy^{\tau_n}_0 + \int_0^t\left(\mathcal{Q} +  \frac{1}{2}\sum_{i=1}^\infty \mathcal{G}_i^2\right)(\sy^{n}_s)\mathbbm{1}_{s \leq \tau_n}ds +  \int_0^t\mathcal{G}\sy^{n}_sd\mathcal{W}_s$$ $\mathbbm{P}-a.s.$ in $U$, therefore from Corollary \ref{operatorthroughstochasticintegral2} we have that
\begin{equation}\label{integral representation}
\mathcal{G}_i\sy^{\tau_n}_t = \mathcal{G}_i\sy^{\tau_n}_0 + \int_0^t\mathcal{G}_i\left(\left(\mathcal{Q} +  \frac{1}{2}\sum_{i=1}^\infty \mathcal{G}_i^2\right)(\sy^{n}_s)\mathbbm{1}_{s \leq \tau_n}\right)ds +  \int_0^t\mathcal{G}_i\mathcal{G}\sy^{n}_sd\mathcal{W}_s\end{equation} $\mathbbm{P}-a.s.$ in $X$ ($\mathcal{G}_i:U \rightarrow X$ is bounded and linear). The time integral is of bounded-variation in $X$ and from Proposition \ref{main martingale result} we have the result. The last thing to prove here is that the infinite sum \begin{equation}\label{last to prove inf sum}\sum_{i=1}^\infty [\mathcal{G}_i\sy^{\tau_n}, W^i]_t\end{equation} converges in $L^2(\Omega;X)$. From the identity (\ref{integral representation}) and the definition of the cross-variation for the semi-martingale,
\begin{equation} \label{cross varry of semi m} [\mathcal{G}_i\sy^{\tau_n}, W^i]_t = \left[\int_0^\cdot\mathcal{G}_i\mathcal{G}\sy^n_sd\mathcal{W}_s,W^i \right]_t.\end{equation}
We can now use Lemma \ref{cross variation convergence} and the definition of the integral as an $L^2\left(\Omega; X\right)$ limit to see that $$ \left[\int_0^\cdot\mathcal{G}_i\mathcal{G}\sy^n_sd\mathcal{W}_s,W^i \right]_t = \left[\sum_{j=1}^\infty \int_0^\cdot \mathcal{G}_i\mathcal{G}_j\sy^n_sdW^j_s,W^i \right]_t = \lim_{m \rightarrow \infty}\left[\sum_{j=1}^m \int_0^\cdot \mathcal{G}_i\mathcal{G}_j\sy^n_sdW^j_s,W^i \right]_t$$ where the limit is taken in $L^1(\Omega,X)$, should this limit exist and satisfy the conditions of Lemma \ref{cross variation convergence}.
We consider $(a_k)$ as an orthonormal basis of $X$, so recalling Definition \ref{definition for cross variation inf dim}, $$\left[\sum_{j=1}^m \int_0^\cdot \mathcal{G}_i\mathcal{G}_j\sy^n_sdW^j_s,W^i \right]_t = \sum_{k=1}^\infty \left[\inner{\sum_{j=1}^m \int_0^\cdot \mathcal{G}_i\mathcal{G}_j\sy^n_sdW^j_s}{a_k}_{X},W^i \right]_ta_k.$$
Now we can first use Theorem \ref{dualityrep} to reduce this to
$$ \sum_{k=1}^\infty \left[\sum_{j=1}^m\int_0^\cdot \inner{ \mathcal{G}_i\mathcal{G}_j\sy^n_s}{a_k}_{X}dW^j_s,W^i \right]_ta_k$$
 from which the classical real valued theory informs us that for $m \geq i$, due to the independence of the Brownian Motions, this is simply $$\sum_{k=1}^\infty \left(\int_0^t\inner{\mathcal{G}_i^2\sy^n_s}{a_k}_Xds\right)a_k = \sum_{k=1}^\infty \inner{\int_0^t\mathcal{G}_i^2\sy^n_sds}{a_k}_Xa_k =\int_0^t\mathcal{G}_i^2\sy^n_sds.$$
Therefore the limit as $m \rightarrow \infty$ is well defined and we have the representation for (\ref{cross varry of semi m}). The convergence of the infinite sum in $L^2(\Omega,X)$, (\ref{last to prove inf sum}), will follow from our standard Cauchy argument, though we have to work a little harder here. The Cauchy argument requires showing that $$\mathbbm{E}\left\Vert\sum_{i=m}^k\int_0^t\mathcal{G}_i^2\sy^n_s ds \right\Vert_X^2 \longrightarrow 0$$ as $m,k \rightarrow \infty$ to which end we note that $$\mathbbm{E}\left\Vert\sum_{i=m}^k\int_0^t\mathcal{G}_i^2\sy^n_s ds \right\Vert_X^2 \leq \mathbbm{E}\left(\sum_{i=m}^k\int_0^t\left\Vert\mathcal{G}_i^2\sy^n_s \right\Vert_X ds\right)^2$$ and \begin{align*}
\mathbbm{E}\left(\sum_{i=m}^k\int_0^t\left\Vert\mathcal{G}_i^2\sy^n_s \right\Vert_X ds\right)^2 &\leq \mathbbm{E}\left( \sum_{i=m}^k c_i^2 \int_0^t\left\Vert\sy^n_s \right\Vert_H ds\right)^2\\ &\leq \mathbbm{E}\left( \sum_{i=m}^k c_i^2 \int_0^tc\left\Vert\sy^n_s \right\Vert_V ds\right)^2\\ &\leq \left(c\sum_{i=m}^kc_i^2\right)^2 \mathbbm{E}\left[ t\int_0^t \left\Vert\sy^n_s \right\Vert_V^2 ds\right]\\
&\leq  \left(c\sum_{i=m}^k c_i^2\right)^2 tn
\end{align*}
where $c$ is the constant from the embedding of $V \xhookrightarrow{} H$. As $\sum_{i=1}^\infty c_i^2 < \infty$ then the Cauchy property follows, hence the Stratonovich integral is indeed well defined.

%Thus defining $$Y_k:= \mathbbm{E}\left(\sum_{i=1}^k \int_0^t\left\Vert\mathcal{G}_i^2\sy^n_s \right\Vert_X ds\right)^2 $$ then $(Y_k)$ is a real valued sequence monotone increasing and bounded above, hence it is convergent so Cauchy. But $$\mathbbm{E}\left\Vert\sum_{i=m}^k\int_0^t\mathcal{G}_i^2\sy^n_s ds \right\Vert_X^2 \leq \mathbbm{E}\left(\sum_{i=m}^k\int_0^t\left\Vert\mathcal{G}_i^2\sy^n_s \right\Vert_X ds\right)^2 \leq Y_k - Y_{m-1}$$ where we've just said that this sequence $(Y_k)$ is Cauchy, proving the required convergence.

\begin{proof}[Proof of \ref{the conversion}:]
We must show that $$\int_0^t\mathcal{G}\sy_s \circ d\mathcal{W}_s = \int_0^t\mathcal{G}\sy_sd\mathcal{W}_s  + \frac{1}{2}\int_0^t\sum_{i=1}^\infty \mathcal{G}_i^2\sy_sds $$ $\mathbbm{P}-a.s.$ in $X$ for all $t \geq 0$. So working with fixed arbitrary $t$ and $\omega$, we choose any $n$ such that $\tau_n(\omega) \geq t$ and have that at this $\omega$, 
\begin{align*}
    \int_0^t\mathcal{G}\sy_s \circ d\mathcal{W}_s &= \sum_{i=1}^\infty  \int_0^t\mathcal{G}_i\sy^n_s dW^i_s + \frac{1}{2}\sum_{i=1}^\infty [\mathcal{G}_i\sy^n, W^i]_t\\
    &= \int_0^t\mathcal{G}\sy_s d\mathcal{W}_s + \frac{1}{2}\sum_{i=1}^\infty [\mathcal{G}_i\sy^n, W^i]_t
\end{align*}
where the limit is taken in $L^2(\Omega;X)$. Of course we have just shown that $$\sum_{i=1}^\infty [\mathcal{G}_i\sy^n, W^i]_t = \sum_{i=1}^\infty \int_0^t\mathcal{G}_i^2\sy^n_sds$$ is well defined in this topology, but we must show that it is equal to $$\int_0^t\sum_{i=1}^\infty \mathcal{G}_i^2\sy_sds$$ evaluated at this fixed $\omega$, for the pointwise limit in $U$ as it was defined. Firstly note that $$\int_0^t\sum_{i=1}^\infty \mathcal{G}_i^2\sy_s(\omega)ds =  \int_0^t\sum_{i=1}^\infty \mathcal{G}_i^2\sy^n_s(\omega)ds $$ and by an application of the Dominated Convergence Theorem with dominating function $$\sum_{i=1}^\infty c_i^2 \norm{\sy^n}_V$$ we can rewrite $$\int_0^t\sum_{i=1}^\infty \mathcal{G}_i^2\sy^n_s ds =  \sum_{i=1}^\infty \int_0^t\mathcal{G}_i^2\sy^n_s ds  $$ as a limit $\mathbbm{P}-a.s.$ in $U$, and thus $\mathbbm{P}-a.s.$ in $X$. However the convergence in $L^2(\Omega;X)$ implies that of a subsequence $\mathbbm{P}-a.s.$ in $X$, which agrees with the limit of the whole sequence $\mathbbm{P}-a.s.$ in $X$ thus giving the result. 
\end{proof}

This theorem has been stated for the strong solution (Definition \ref{definitionofregularsolutionglobal}), though we note that all arguments follow in the corresponding local case by incorporating the stopping time as done in the justification that the integrals in Definition \ref{definitionofregularsolution} are well defined. The result is stated below.

\begin{corollary}
    If $(\sy,\tau)$ is a local strong solution of the equation (\ref{the identified one})
    then $\sy$ satisfies the identity
    $$ \sy_t = \sy_0 + \int_0^{t\wedge\tau}\mathcal{Q}\sy_s ds +  \int_0^{t \wedge \tau}\mathcal{G}\sy_s \circ d\mathcal{W}_s $$
    $\mathbbm{P}-$ $a.s.$ in $X$ for all $t \geq 0$. 
\end{corollary}

We should emphasise why the identity (\ref{stratoSPDE}) holds only in $X$ and not in $U$, the space in which (\ref{the identified one}) is satisfied. The evolution equation (\ref{integral representation}) allowed us to identify the semi-martingale structure in $X$ as is this is where the integrals are constructed. We can, however, construct the stochastic integral in (\ref{integral representation}) in $U$, which allows us to conclude that the time integral is itself an element of $U$ (as $\mathcal{G}_i\sy^n_t, \mathcal{G}_i\sy^n_0$ are as well). Unfortunately we cannot say that this is actually an integral in $U$ (just an integral in $X$ which is in turn an element of $U$) so poignantly we cannot say that this is of bounded-variation in $U$, which would be necessary when considering the cross-variation $[\mathcal{G}_i\sy^n,W^i]_t$ in $U$.\\

We use Theorem \ref{the conversion} as a way of defining the Stratonovich SPDE, however as discussed, with a priori martingality assumptions then the Stratonovich integral is well defined and one can prove a converse of this theorem. %For this we shall also assume that $\mathcal{G}_i$ possesses a densely defined adjoint $\mathcal{G}_i^*: H \rightarrow U$ in $U$. 

\begin{theorem} \label{how was it unlabelled}
    Let $\sy$ be such that for $\mathbbm{P}-a.e.$ $\omega$, $\sy_{\cdot}(\omega) \in C\left([0,T];H\right) \cap L^2\left([0,T];V\right)$ for all $T>0$ with $\sy$ progressively measurable in $V$. Assume in addition that $\mathcal{G}\sy \in \bar{\mathcal{I}}^{U}_{\circ}(\mathcal{W})$ and $\sy$ satisfies the identity (\ref{stratoSPDE}) $\mathbbm{P}-a.s.$ in $U$ for all $t \geq 0$. Then $\sy$ satisfies the identity (\ref{the identified one}) $\mathbbm{P}-a.s.$ in $X$ for all $t \geq 0$.
\end{theorem}

\begin{proof}
    %The first step is to acknowledge how the Stratonovich Integral in (\ref{stratoSPDE}) is well defined; that is that for each $i$, $\mathcal{G}_i\sy$ is a semi-martingale in $U$. For this we consider the bounded-variation and martingale parts of $\sy$ in $H$. We use again that as $\mathcal{G}_i$ is bounded and linear from $H$ into $U$ then the bounded-variation property is preserved under $\mathcal{G}_i$. Let $\tilde{\sy}$ be the martingale part of $\sy$. To show that $\mathcal{G}_i\tilde{\sy}$ is a martingale in $U$ we must show that for all $\phi \in U$, $\inner{\mathcal{G}_i\tilde{\sy}}{\phi}_U$ is a martingale in $\R$. However $\mathcal{G}_i$ possesses an adjoint $\mathcal{G}_i^*:U \rightarrow H$ which is such that $$\inner{\mathcal{G}_i\tilde{\sy}}{\phi}_U = \inner{\tilde{\sy}}{\mathcal{G}_i^*\phi}_H$$ which is a martingale given the martingality of $\tilde{\sy}$ in $H$. Thus the Stratonovich Integral is well defined and we can write (\ref{stratoSPDE}) as
    By assumption the Stratonovich integral is well defined so we can write $$\sy_t = \sy_0 + \int_0^t\mathcal{Q}\sy_s ds + \sum_{i=1}^\infty\int_0^t \mathcal{G}_i\sy_s dW^i_s + \frac{1}{2}\sum_{i=1}^\infty\left[ \mathcal{G}_i \sy, W^i\right]_t$$ for the cross variation taken in $U$. We wish to write out this cross variation process explicitly, so as in the proof of Theorem \ref{the conversion} we consider the evolution equation satisfied by $\mathcal{G}_i\sy$. It is the cross-variation term which could be problematic, but we appreciate that
    $$\mathcal{G}_i\left( \sum_{j=1}^\infty\left[ \mathcal{G}_j \sy, W^j\right]_t\right)  =  \sum_{j=1}^\infty\mathcal{G}_i\left(\left[ \mathcal{G}_j \sy, W^j\right]_t\right)$$
    for the limit now in $L^2(\Omega;X)$, validated as $\mathcal{G}_i$ is bounded and linear from $U$ into $X$. It is not clear if this if of finite-variation, so to proceed similarly to (\ref{cross varry of semi m}), we wish to show that $$\left[\sum_{j=1}^\infty\mathcal{G}_i\left(\left[ \mathcal{G}_j \sy, W^j\right]\right) , W^i \right]_t = 0$$
    $\mathbbm{P}-a.s.$ for any $t \geq 0$. For this we again use Lemma \ref{cross variation convergence} to see that $$\left[\sum_{j=1}^\infty\mathcal{G}_i\left(\left[ \mathcal{G}_j \sy, W^j\right]\right) , W^i \right]_t = \lim_{m \rightarrow \infty}\left[\sum_{j=1}^m\mathcal{G}_i\left(\left[ \mathcal{G}_j \sy, W^j\right]\right) , W^i \right]_t$$
    for the limit in $L^1\left(\Omega;X\right)$. Each term in this sequence must be zero, though, as each $\left[ \mathcal{G}_j \sy, W^j\right]$ is of finite-variation in $U$ so $\mathcal{G}_i\left(\left[ \mathcal{G}_j \sy, W^j\right]\right)$ is of finite-variation in $X$.
 Thus through considering the identity (\ref{stratoSPDE}) in $X$, by the exact same process as Theorem \ref{the conversion}, we prove the result.

\end{proof}

\subsection{Weak Solutions in the Abstract Framework}

As in the study of PDEs, to expand the existence theory we shall also consider weaker notions of solution. We do this in two ways: \textit{analytically} weak solutions and \textit{probabilistically} weak solutions, the second of which we shall refer to as martingale solutions. We begin by giving a definition of analytically weak solutions in the established framework, though it is difficult to be too precise here given how the nature of weak solutions is dependent on the equation. As a brief example to illustrate that point we mention the Navier-Stokes Equation, where the weak form for the Laplacian involves passing a derivative over to the test function whereas the weak form for the nonlinear term uses no such adjoint but rather an understanding of the $L^2$ inner product as an $L^{6/5} \times L^6$ duality. To this end we shall just consider $\mathcal{Q}_1$, $\mathcal{Q}_2$ to be two functions on $H$ which are such that if $\phi \in V$ and $\psi \in H$ then \begin{equation}\label{property a la 1}\inner{\mathcal{Q}_1 \phi}{\mathcal{Q}_2 \psi}_U = \inner{\mathcal{Q}\phi}{ \psi}_U \end{equation} with enough regularity for $\inner{\mathcal{Q}_1 \phi}{\mathcal{Q}_2 \psi}_U$ to be well defined, along with the corresponding integrals to be considered shortly. As for the noise operator $\mathcal{G}$ we can be more precise, and assume that there exists an adjoint operator $\mathcal{G}_i^*$ with the same boundedness properties as $\mathcal{G}_i$ satisfying \begin{equation}\label{property a la 2}\inner{\mathcal{G}_i \phi}{\psi}_U = \inner{ \phi}{\mathcal{G}_i^*\psi}_U \end{equation} for all $\phi, \psi \in H$. To simplify notation we consider $\mathcal{G}^*$ as an operator on $\mathfrak{U}$ in the same way as we do for $\mathcal{G}$. We now also require that the inclusions $V \xhookrightarrow{} H \xhookrightarrow{} U$ are dense. With this in place we can define such a solution. 

\begin{definition} \label{analytically weak sol def}
    Let $\sy_0: \Omega \rightarrow U$ be $\mathcal{F}_0-$measurable. A process $\sy$ such that for $\mathbbm{P}-a.e.$ $\omega$, $\sy_{\cdot}(\omega) \in C_{w}\left([0,T];U\right) \cap L^2\left([0,T];H\right)$ for all $T > 0$ with $\sy_{\cdot}$ progressively measurable in $H$, is said to be a weak solution of the equation (\ref{thespde}) if the identity 
    \begin{equation} \label{identity for weak def}
        \inner{\sy_t}{\phi}_U = \inner{\sy_0}{\phi}_U + \int_0^t\inner{\mathcal{Q}_1\sy_s}{\mathcal{Q}_2 \phi}_Uds + 
        \int_0^t\inner{\sy_s}{\mathcal{G}^* \phi}_U d\mathcal{W}_s
    \end{equation}
    holds $\mathbbm{P}-a.s.$ in $\R$ for all $\phi \in H$ and $t \geq 0$. 
\end{definition}

Immediately we observe that the stochastic integral is well defined through exactly the same justification as the integral for strong solutions. We also justify our use of the terminology `weak'.

\begin{proposition}
    Let $\sy_0: \Omega \rightarrow H$ be $\mathcal{F}_0-$measurable, and suppose that $\sy$ is a process whereby for $\mathbbm{P}-a.e.$ $\omega$, $\sy_{\cdot}(\omega) \in C\left([0,T];H\right) \cap L^2\left([0,T];V\right)$ for all $T > 0$ with $\sy_{\cdot}$ progressively measurable in $V$. Then $\sy$ is a strong solution of the equation (\ref{thespde}) if and only if it is a weak solution.
\end{proposition}

\begin{proof}
    We consider the two implications in turn, noting that we only need to show the equivalence of the two identities (\ref{thespde}), (\ref{identity for weak def}):
    \begin{itemize}
        \item[$\implies$:] This direction is clear, as we simply take the inner product with $\phi$ in the identity (\ref{thespde}) and then use Proposition \ref{unbounded take it in 2}. The properties (\ref{property a la 1}), (\ref{property a la 2}) then conclude the implication.
        \item[$\impliedby$:] Via the identical process in reverse, we obtain that $$\inner{\sy_t}{\phi}_U = \inner{\sy_0}{\phi}_U + \inner{\int_0^t\mathcal{Q}\sy_sds}{\phi}_U + 
        \inner{\int_0^t\mathcal{G}\sy_s d\mathcal{W}_s}{\phi}_U$$
        for all $\phi \in H$. We then use the density of $H$ in $U$ to deduce this identity for all $\phi \in U$, from which the result follows. 
    \end{itemize}
\end{proof}

There is an important difference between the weak and strong solutions in terms of the Stratonovich Equation (\ref{stratoSPDE}). Recall from the previous subsection that we needed to pass to an additional Hilbert Space $X$ to make the conversion between It\^{o} and Stratonovich forms, which was owing to the `loss of a derivative' from $\mathcal{G}_i$. In the weak form this operator hits the test function instead, so it is there where some additional regularity is required. We have the following analogue of Theorem \ref{the conversion}.

\begin{theorem} \label{poiu}
    If $\sy$ is a weak solution of the equation (\ref{the identified one}) then $\sy$ satisfies the identity
     \begin{equation} \label{identity for weak def strato}
        \inner{\sy_t}{\phi}_U = \inner{\sy_0}{\phi}_U + \int_0^t\inner{\mathcal{Q}_1\sy_s}{\mathcal{Q}_2 \phi}_Uds + 
        \int_0^t\inner{\sy_s}{\mathcal{G}^* \phi}_U \circ d\mathcal{W}_s
    \end{equation}
    $\mathbbm{P}-a.s.$ in $\R$ for all $\phi \in V$ and $t \geq 0$ .
\end{theorem}

As we did for Theorem \ref{the conversion}, we first make precise the meaning of a weak solution of (\ref{the identified one}) in terms of the `It\^{o}-Stratonovich Corrector'. The infinite sum is again taken as a pointwise limit, where $\sy$ satisfies the identity $$\inner{\sy_t}{\phi}_U = \inner{\sy_0}{\phi}_U + \int_0^t\inner{\mathcal{Q}_1\sy_s}{\mathcal{Q}_2 \phi}_Uds + \frac{1}{2}\int_0^t\sum_{i=1}^\infty\inner{\mathcal{G}_i\sy_s}{\mathcal{G}_i^*\phi}_U ds + 
        \int_0^t\inner{\sy_s}{\mathcal{G}^* \phi}_U d\mathcal{W}_s.$$
        Many of the technicalities of this result were addressed in Theorem \ref{the conversion}, so in the proof here we only show directly that the It\^{o}-Stratonovich Corrector is of the right form, where now we consider $\tau^n$ as the first hitting time in $H$ and use the same notation $\sy^n$.
        
\begin{proof}[Proof of Theorem \ref{poiu}:]
    Through the arguments of Theorem \ref{the conversion}, it is sufficient to show that \begin{equation} \label{sufficient implies}\left[\inner{\sy^n}{\mathcal{G}_i^*\phi}_U, W^i \right]_t = \int_0^t \inner{\mathcal{G}_i\sy_s^n}{\mathcal{G}_i\phi}_U ds\end{equation}
for any given $\phi \in V$. For this we consider the evolution equation satisfied by $\inner{\sy^n}{\mathcal{G}_i^*\phi}_U$. Because of this then $\mathcal{G}_i^*\phi \in H$ and can be used as a test function in the weak formulation, such that $\sy$ satisfies the identity \begin{align*}\inner{\sy^n_t}{\mathcal{G}_i^*\phi}_U = \inner{\sy^n_0}{\mathcal{G}_i^*\phi}_U + \int_0^t\inner{\mathcal{Q}_1\sy^n_s}{\mathcal{Q}_2\mathcal{G}_i^* \phi}_Uds &+ \frac{1}{2}\int_0^t\sum_{i=1}^\infty\inner{\mathcal{G}_i\sy^n_s}{\mathcal{G}_i^*\mathcal{G}_i^*\phi}_U ds\\ &+ 
        \int_0^t\inner{\sy^n_s}{\mathcal{G}_i^*\mathcal{G}^* \phi}_U d\mathcal{W}_s.\end{align*}
Of course we can rewrite $\inner{\sy^n_s}{\mathcal{G}_i^*\mathcal{G}^* \phi}_U = \inner{\mathcal{G}_i\sy^n_s}{\mathcal{G}^* \phi}_U$ so through the same process as in Theorem \ref{the conversion}, the above identity implies (\ref{sufficient implies}) which concludes the proof.
\end{proof}

%Similarly we have the analogy of Theorem \ref{how was it unlabelled}.

%\begin{theorem}
 %   Let $\sy$ be such that for $\mathbbm{P}-a.e.$ $\omega$, $\sy_{\cdot}(\omega) \in C_{w}\left([0,T];U\right) \cap L^2\left([0,T];H\right)$ for all $T > 0$ with $\sy$ progressively measurable in $H$. Assume in addition that $\sy \in \bar{\mathcal{I}}^U_{\circ}(\mathcal{W})$ and satisfies the identity (\ref{identity for weak def strato}) $\mathbbm{P}-a.s.$ in $\R$ for all $\phi \in H$ and $t \geq 0$. Then $\sy$ satisfies the identity (\ref{identity for weak def}) $\mathbbm{P}-a.s.$ in $\R$ for all $\phi \in V$ and $t \geq 0$.
    
  %  \end{theorem}

%\begin{proof}
 %   The proof is, in essence, a translation of that of Theorem \ref{how was it unlabelled} into the weak formulation. Here $\sy$ satisfies the identity $$  \inner{\sy_t}{\phi}_U = \inner{\sy_0}{\phi}_U + \int_0^t\inner{\mathcal{Q}_1\sy_s}{\mathcal{Q}_2 \phi}_Uds + 
  %      \int_0^t\inner{\sy_s}{\mathcal{G}^* \phi}_U  d\mathcal{W}_s + \frac{1}{2}\sum_{i=1}^\infty \left[\inner{\sy_s}{\mathcal{G}_i^* \phi}_U, W^i \right].$$
   %By assuming that $\phi \in V$, we can use $\mathcal{G}_i^*\phi$ as a test function in this representation from which (\ref{sufficient implies}) can again be deduced, using that $\sum_{i=1}^\infty \left[\inner{\sy_s}{\mathcal{G}_i^* \mathcal{G}_i^*\phi}_U, W^i \right]$ is of finite variation.
%\end{proof}

To conclude this subsection we very briefly comment on the notion of a probabilistically weak solution, which we refer to as a martingale solution. 

\begin{definition}
    Let $\sy_0: \Omega \rightarrow H$ be $\mathcal{F}_0-$measurable. If there exists a filtered probability space $\left(\tilde{\Omega},\tilde{\mathcal{F}},(\tilde{\mathcal{F}}_t), \tilde{\mathbbm{P}}\right)$, a cylindrical Brownian Motion $\tilde{\mathcal{W}}$ over $\mathfrak{U}$ with respect to $\left(\tilde{\Omega},\tilde{\mathcal{F}},(\tilde{\mathcal{F}}_t), \tilde{\mathbbm{P}}\right)$, an $\tilde{\mathcal{F}_0}-$measurable $\tilde{\sy}_0: \tilde{\Omega} \rightarrow H$ with the same law as $\sy_0$, a process $\tilde{\sy}$ such that for $\tilde{\mathbbm{P}}-a.e.$ $\omega$, $\tilde{\sy}_{\cdot}(\omega) \in C\left([0,T];H\right) \cap L^2\left([0,T];V\right)$ for all $T>0$ with $\tilde{\sy}$ progressively measurable in $V$, is said to be a martingale strong solution of the equation (\ref{thespde}) if the identity $$\tilde{\sy}_t = \tilde{\sy}_0 + \int_0^t \mathcal{Q}\tilde{\sy}_sds + \int_0^t\mathcal{G}\tilde{\sy}_sd\mathcal{W}_s$$ holds $\tilde{\mathbbm{P}}-a.s.$ in $U$ for all $t \geq 0$.
\end{definition}

\subsection{Time-Dependent Operators}

We did not facilitate time dependence in the operators $\mathcal{Q},\mathcal{G}$ as solely for the fact that if $\mathcal{G}$ was time dependent then the conversion from Stratonovich to It\^{o} Form would be much more troublesome. There is no real additional difficulty in establishing a framework for the It\^{o} Form for time-dependent operators, so we briefly do so now. There is no longer a need for the space $X$ so we work with the triple $$V \xhookrightarrow{} H \xhookrightarrow{} U$$ and now the SPDE
\begin{equation} \label{thespde*}
    \sy_t = \sy_0 + \int_0^t \mathcal{Q}(s,\sy_s)ds + \int_0^t\mathcal{G}(s,\sy_s)d\mathcal{W}_s.
\end{equation}
We require the assumptions now as:

 \begin{assumption} \label{Qassumpt2}
For any $T>0$, the operators $\mathcal{Q}:[0,T] \times V \rightarrow U$ and $\mathcal{G}:[0,T] \times V \rightarrow \mathscr{L}^2(\mathfrak{U};H)$ are measurable.
\end{assumption}

\begin{assumption} \label{Gassumpt2}
There exists a $C_{\cdot}:[0,\infty) \rightarrow \R$ bounded on $[0,T]$ for every $T$, and constants $c_i$ such that for every $\phi \in V$ and $t \in [0,\infty)$, \begin{align*}\norm{\mathcal{Q}(t,\boldsymbol{\phi})}_{U} &\leq C_t\tilde{K}(\phi)\left[1 + \norm{\boldsymbol{\phi}}_{V}^2\right]\\
 \norm{\mathcal{G}_i(t,\boldsymbol{\phi})}^2_{H} &\leq C_tc_i(1 + \norm{\phi}_V^2)\\
 \sum_{i=1}^\infty c_i &< \infty
    \end{align*}
\end{assumption}
Definitions of solutions in this framework, and a justification that the integrals are well-defined, then follows largely in the same way as for (\ref{thespde}) so we omit the details here. What is slightly more delicate is the progressive measurability of the process $\mathcal{G}(\cdot,\py_{\cdot})\mathbbm{1}_{\cdot \leq \tau}$ in $\mathscr{L}^2(\mathfrak{U};H)$. From the measurability of $\mathcal{G}$ and the progressive measurability of $\py$ we have that for any fixed $t$, the mapping $$\mathcal{G}(\cdot,\py_{\cdot}): [0,t] \times [0,t] \times \Omega \rightarrow \mathscr{L}^2(\mathfrak{U};H)$$ defined by $$(s,r,\omega) \mapsto \mathcal{G}(s,\py_{r}(\omega))$$ is $\mathcal{B}([0,t]) \times \mathcal{B}([0,t]) \times \mathcal{F}_t$ measurable, and hence the mapping $$(s,\omega) \mapsto \mathcal{G}(s,\py_{s}(\omega))$$ is $\mathcal{B}([0,t]) \times \mathcal{F}_t$ measurable as is the indicator up to the stopping time, so the product retains the required measurability. For completeness we define the notion of a strong solution here.

\begin{definition} \label{definitionofregularsolutionglobaltime}
A process $\sy$ such that for $\mathbbm{P}-a.e.$ $\omega$, $\sy_{\cdot}(\omega) \in C\left([0,T];H\right) \cap L^2\left([0,T];V\right)$ for all $T>0$ with $\sy$ progressively measurable in $V$, is said to be a strong solution of the equation (\ref{thespde*}) if the identity (\ref{thespde*}) holds $\mathbbm{P}-a.s.$ in $U$ for all $t \geq 0$.
\end{definition}

\subsection{An Energy Equality}
Having established this framework we introduce techniques to facilitate our analysis in it. The It\^{o} Formula is well regarded as one of the most useful tools in stochastic analysis, and we formulate an infinite dimensional version here in the case of an energy equality. We shall introduce a new setting in which our established solution framework falls, with the understanding that we would like to apply this to solutions whilst also using the results to deduce the existence of solutions when they are not a priori known. The ideas of this subsection are just a mild extension of [\cite{prevot2007concise}] Theorem 4.2.5.\\

To prove this energy equality we shall rely on looking at partitions in time over which some nice properties are satisfied, before taking the limit as the increments go to zero. Towards this we recall the following lemma from [\cite{prevot2007concise}], Lemma 4.2.6.

\begin{lemma} \label{4.2.6}
    Let $\mathcal{X}_1 \xhookrightarrow{} \mathcal{X}_2$ be two Banach Spaces with continuous embedding and suppose that for some $T>0$ and stopping time $\tau$, $\py: \Omega \times [0,T] \rightarrow \mathcal{X}_2$ is such that for $\mathbbm{P}-a.e.$ $\omega$, $\py_{\cdot}(\omega) \in C\left([0,T];\mathcal{X}_2\right)$ and  
        $\py_{\cdot}\mathbbm{1}_{\cdot \leq \tau} \in L^2\left(\Omega \times [0,T]; \mathcal{X}_1\right)$. Then for any $A \subset [0,T]$ with $\lambda(A) = 0$ there exists a sequence of partitions $(I_l)$ such that
    \begin{enumerate}
        \item $I_l \cap A = \emptyset$;
        \item $I_l \subset I_{l+1}$;
        \item $I_l:= \left\{0=t^l_0 < t^l_1 < \dots < t^l_{k_l}=T\right\}$, $\max_j\vert t^l_{j}-t^l_{j-1} \vert \rightarrow 0$ as $l \rightarrow \infty$;
        \item For $\mathbbm{P}-a.e.$ $\omega$ and every $t^l_j$ with $1 \leq j \leq k_{l-1}$, $\py_{t^l_j}(\omega)\mathbbm{1}_{t^l_j \leq \tau(\omega)} \in \mathcal{X}_1$;
        \item \label{does it work} The processes $\hat{\py}^l, \tilde{\py}^l$ defined at each $t \in [0,T]$ and $\omega \in \Omega$ by $$\hat{\py}^l_t(\omega) := \sum_{j=2}^{k_l}\mathbbm{1}_{[t^l_{j-1},t^l_j]}(t)\py_{t^l_{j-1}}(\omega)\mathbbm{1}_{t^l_{j-1} \leq \tau(\omega)}, \qquad  \tilde{\py}^l := \sum_{j=1}^{k_{l-1}}\mathbbm{1}_{[t^l_{j-1},t^l_j]}(t)\py_{t^l_{j}}(\omega)\mathbbm{1}_{t^l_j \leq \tau(\omega)}$$ belong to $L^2\left(\Omega \times [0,T]; \mathcal{X}_1\right)$ and both converge to $\py_{\cdot}\mathbbm{1}_{\cdot \leq \tau}$ in this space.
    \end{enumerate}

\end{lemma}

    Before moving on we take a moment to dissect this result. In the statement of [\cite{prevot2007concise}] there is no continuity assumption on $\py$ and indeed this is superfluous to requirement, however we want to make explicit that is genuinely the process $\py$ taken in item \ref{does it work} and not some other representative of an equivalence class. The assumptions are of course reminiscent of Definition \ref{definitionofregularsolution} and just as was stressed there that the progressively measurable process in $V$ was not necessarily the continuous process itself but just a $\mathbbm{P} \times \lambda-a.s.$ equivalent representation, we are reminded again that elements of $L^2\left(\Omega \times [0,T]; \mathcal{X}_1\right)$ are only an equivalence class of $\mathbbm{P} \times \lambda-a.s.$ equal functions so the fact that we can fix the representation $\py$ in $\hat{\py}^l,\tilde{\py}^l$ is significant.\\
    
    We now fix a framework in which we conduct the analysis of this subsection. We work with a triple of embedded Hilbert Spaces $$V \xhookrightarrow{} H \xhookrightarrow{} U$$ where the embeddings are continuous, $V$ is assumed dense in $H$, and there exists a continuous bilinear form $\inner{\cdot}{\cdot}_{U \times V}: U \times V \rightarrow \R$ such that for every $\phi \in H, \psi \in V$, $$\inner{\phi}{\psi}_{U \times V} = \inner{\phi}{\psi}_H.$$ We suppose that for some $T>0$ and stopping time $\tau$:
    \begin{enumerate}
        \item \label{11111} $\sy_0 \in L^2(\Omega;H)$ is $\mathcal{F}_0-$measurable;
        \item \label{22222} $\eta \in L^2\left(\Omega;L^2([0,T];U)\right)$;
        \item \label{33333} $B \in \mathcal{I}^H(\mathcal{W})$;
        \item  \label{4} $\sy_{\cdot}\mathbbm{1}_{\cdot \leq \tau} \in L^2\left(\Omega;L^2([0,T];V)\right)$ and is progressively measurable in $V$;
        \item \label{item 5 again} The identity
        \begin{equation} \label{newest identity}
            \sy_t = \sy_0 + \int_0^{t \wedge \tau}\eta_sds + \int_0^{t \wedge \tau}B_s d\mathcal{W}_s
        \end{equation}
        holds $\mathbbm{P}-a.s.$ in $U$ for all $t \in [0,T]$.
    \end{enumerate}
    
    \begin{remark}
    Once more, it is clear from assumption \ref{item 5 again} that for $\mathbbm{P}-a.e.$ $\omega$, $\sy(\omega) \in C\left([0,T];U\right)$ but the progressively measurable process assumed in item \ref{4} is only equivalent to this $\sy_{\cdot}\mathbbm{1}_{\cdot \leq \tau}$ $\mathbbm{P} \times \lambda-a.s.$. It is a slight abuse of notation commonplace in the literature and we shall follow suit without excessive clarification at each use. 
    \end{remark}

With this structure in place, we first look to deduce some improved regularity on $\sy$. For this we fix an application of Lemma \ref{4.2.6} relative to the assumptions laid out above. We take $T$ and $\tau$ as in the assumptions, $\mathcal{X}_1=V$, $\mathcal{X}_2=U$, $\py=\sy$. From item \ref{4} and the continuous embedding of $V$ into $H$ we have that for $\lambda-a.e.$ $t \in [0,T]$, $\mathbbm{E}\left(\norm{\sy_t}_H^2\right) < \infty$ and we take $A$ to be the $\lambda-$zero set on which this does not hold. Then $\hat{\sy}^l,\tilde{\sy}^l,I^l$ are defined as in Lemma \ref{4.2.6} and we define $$I:= \bigcup_{l \in \N}I^l.$$

\begin{lemma}
     We have that $$\mathbbm{E}\left(\sup_{t \in [0,T]}\norm{\sy_t}_H^2\right) < \infty. $$
\end{lemma}

\begin{proof}
For $\mathbbm{P}-a.e.$ $\omega$ and every $s<t$ with $t \in I \cap [0, \tau(\omega)]$ and $s \in I \cup \{0\}$, observe that
\begin{align*}
    &\left\Vert \int_{s}^{t }B_rd\mathcal{W}_r  \right\Vert_{H}^2 - \left\Vert \sy_t-\sy_s - \int_{s}^{t }B_rd\mathcal{W}_r\right\Vert_H^2 + 2\left\langle \sy_s, \int_{s}^{t }B_rd\mathcal{W}_r  \right\rangle_{\mathcal{H}}\\ & \qquad= \left\Vert \int_{s}^{t }B_rd\mathcal{W}_r  \right\Vert_{H}^2 - \left\Vert \sy_t-\sy_s \right\Vert_H^2 - \left\Vert \int_{s}^{t }B_rd\mathcal{W}_r  \right\Vert_{H}^2\\ & \qquad \qquad \qquad + 2\left\langle \sy_t - \sy_s, \int_{s}^{t }B_rd\mathcal{W}_r  \right\rangle_{\mathcal{H}} + 2\left\langle \sy_s, \int_{s}^{t }B_rd\mathcal{W}_r  \right\rangle_{\mathcal{H}}\\
    & \qquad = 2\left\langle \sy_t, \int_{s}^{t }B_rd\mathcal{W}_r  \right\rangle_{\mathcal{H}} - \left\Vert \sy_t-\sy_s \right\Vert_H^2\\
    & \qquad= 2\left\langle \sy_t, \sy_t - \sy_s - \int_{s}^{t }\eta_rdr  \right\rangle_{\mathcal{H}} -\left\Vert \sy_t\right\Vert_H^2 - \left\Vert \sy_s \right\Vert_H^2 + 2\left\langle \sy_t, \sy_s \right\rangle_H\\
    & \qquad= \left\Vert \sy_t\right\Vert_H^2 - \left\Vert \sy_s \right\Vert_H^2 -2\left\langle \sy_t, \int_{s}^{t }\eta_rdr  \right\rangle_{\mathcal{H}}\\
    & \qquad= \left\Vert \sy_t\right\Vert_H^2 - \left\Vert \sy_s \right\Vert_H^2 -2 \int_{s}^{t }\left\langle \eta_r, \sy_t\right\rangle_{U \times V}dr
\end{align*}
which we rewrite as the equality
\begin{align*}
     \left\Vert \sy_t\right\Vert_H^2 - \left\Vert \sy_s \right\Vert_H^2&= 2 \int_{s}^{t }\left\langle \eta_r, \sy_t\right\rangle_{U \times V}dr + 2\left\langle \sy_s, \int_{s}^{t }B_rd\mathcal{W}_r  \right\rangle_{\mathcal{H}}  \\&+ \left\Vert \int_{s}^{t }B_rd\mathcal{W}_r  \right\Vert_{H}^2 - \left\Vert \sy_t-\sy_s - \int_{s}^{t }B_rd\mathcal{W}_r\right\Vert_H^2. 
\end{align*}
Using this equality, for any $l \in \N$ and $t=t^l_i \in I_l \cap (0,\tau(\omega)] /\{T\}$,
\begin{align*}
    \norm{\sy_t}_{H}^2 &- \norm{\sy_0}_H^2 = \sum_{j=0}^{i-1}\left( \norm{\sy_{t^l_{j+1}}}_{H}^2 - \norm{\sy_{t^l_j}}_H^2 \right)\\
    &= \sum_{j=0}^{i-1}\Bigg(2 \int_{t^l_{j}}^{t^l_{j+1} }\left\langle \eta_r, \sy_{t^l_{j+1}}\right\rangle_{U \times V}dr + 2\left\langle \sy_{t^l_{j}}, \int_{t^l_{j}}^{t^l_{j+1} }B_rd\mathcal{W}_r  \right\rangle_{\mathcal{H}}  \\& \qquad \qquad + \left\Vert \int_{t^l_{j}}^{t^l_{j+1} }B_rd\mathcal{W}_r  \right\Vert_{H}^2 - \left\Vert \sy_{t^l_{j+1}}-\sy_{t^l_{j}} - \int_{t^l_{j}}^{t^l_{j+1} }B_rd\mathcal{W}_r\right\Vert_H^2     \Bigg)\\
    &= 2\sum_{j=0}^{i-1}\left( \int_{t^l_{j}}^{t^l_{j+1} }\left\langle \eta_r, \sy_{t^l_{j+1}}\right\rangle_{U \times V}dr +  \int_{t^l_{j}}^{t^l_{j+1}}\left\langle B_r,\sy_{t^l_{j}}\right\rangle_{\mathcal{H}}d\mathcal{W}_r  \right)  \\& \qquad \qquad +\sum_{j=0}^{i-1}\left( \left\Vert \int_{t^l_{j}}^{t^l_{j+1} }B_rd\mathcal{W}_r  \right\Vert_{H}^2 - \left\Vert \sy_{t^l_{j+1}}-\sy_{t^l_{j}} - \int_{t^l_{j}}^{t^l_{j+1} }B_rd\mathcal{W}_r\right\Vert_H^2     \right)\\
    &= 2\sum_{j=0}^{i-1}\left( \int_{t^l_{j}}^{t^l_{j+1} }\left\langle \eta_r, \tilde{\sy}^l_{r}\right\rangle_{U \times V}dr\right) +  2\sum_{j=1}^{i-1}\left(\int_{t^l_{j}}^{t^l_{j+1}}\left\langle B_r, \hat{\sy}^l_{r}\right\rangle_{\mathcal{H}}d\mathcal{W}_r  \right) + 2\int_0^{t^l_{1}}\left\langle B_r,\sy_{0}\right\rangle_{\mathcal{H}}d\mathcal{W}_r   \\& \qquad \qquad +\sum_{j=0}^{i-1}\left( \left\Vert \int_{t^l_{j}}^{t^l_{j+1} }B_rd\mathcal{W}_r  \right\Vert_{H}^2 - \left\Vert \sy_{t^l_{j+1}}-\sy_{t^l_{j}} - \int_{t^l_{j}}^{t^l_{j+1} }B_rd\mathcal{W}_r\right\Vert_H^2     \right)\\
    &= 2 \int_{0}^{t}\left\langle \eta_r, \tilde{\sy}^l_{r}\right\rangle_{U \times V}dr +  2\int_{0}^{t}\left\langle B_r,\hat{\sy}^l_{r} \right\rangle_{H}d\mathcal{W}_r + 2\int_0^{t^l_{1}}\left\langle B_r,\sy_{0}\right\rangle_{H}d\mathcal{W}_r   \\& \qquad \qquad +\sum_{j=0}^{i-1}\left( \left\Vert \int_{t^l_{j}}^{t^l_{j+1} }B_rd\mathcal{W}_r  \right\Vert_{H}^2 - \left\Vert \sy_{t^l_{j+1}}-\sy_{t^l_{j}} - \int_{t^l_{j}}^{t^l_{j+1} }B_rd\mathcal{W}_r\right\Vert_H^2     \right)
\end{align*}
where we have applied Proposition \ref{unbounded take it in 2} and the associated remark thereafter. In particular we have that
\begin{align*}
    \norm{\sy_t}_{H}^2 &\leq \norm{\sy_0}_H^2 + 2 \int_{0}^{t}\left\langle \eta_r, \tilde{\sy}^l_{r}\right\rangle_{U \times V}dr +  2\int_{0}^{t}\left\langle B_r,\hat{\sy}^l_{r} \right\rangle_{H}d\mathcal{W}_r\\ &+ 2\int_0^{t^l_{1}}\left\langle B_r,\sy_{0}\right\rangle_{H}d\mathcal{W}_r +\sum_{j=0}^{i-1}\left( \left\Vert \int_{t^l_{j}}^{t^l_{j+1} }B_rd\mathcal{W}_r  \right\Vert_{H}^2\right).
\end{align*}
Our goal is to show that
$$\mathbbm{E}\left(\sup_{t \in I_l \cap (0,\tau] /\{T\}}\norm{\sy_t}_{H}^2\right) \leq c $$
for some constant $c$ independent of $l$. To this end, observe that 
\begin{align*}
   \mathbbm{E}\left(\sup_{t \in I_l \cap (0,\tau] /\{T\}}\norm{\sy_t}_{H}^2\right) &\leq \mathbbm{E}\left(\norm{\sy_0}_H^2\right)\\ &+ 2\mathbbm{E}\left( \int_{0}^{T \wedge \tau}\left\vert\left\langle \eta_r, \tilde{\sy}^l_{r}\right\rangle_{U \times V}\right\vert dr\right) +  2\mathbbm{E}\left(\sup_{t \in [0,T]}\left\vert\int_{0}^{t \wedge \tau}\left\langle B_r,\hat{\sy}^l_{r} \right\rangle_{H}d\mathcal{W}_r\right\vert\right)\\ &+ 2\mathbbm{E}\left(\sup_{t \in [0,T]}\left\vert\int_0^{t\wedge \tau}\left\langle B_r,\sy_{0}\right\rangle_{H}d\mathcal{W}_r\right\vert\right)\\ &+\mathbbm{E}\left[\sup_{t \in I_l \cap (0,\tau] /\{T\}}\sum_{j=0}^{i-1}\left( \left\Vert \int_{t^l_{j}}^{t^l_{j+1} }B_rd\mathcal{W}_r  \right\Vert_{H}^2\right)\right].
\end{align*}
We shall treat each term individually. Firstly we have that
\begin{align*}
    2\mathbbm{E}\left( \int_{0}^{T \wedge \tau}\left\vert\left\langle \eta_r, \tilde{\sy}^l_{r}\right\rangle_{U \times V}\right\vert dr\right) &\leq \mathbbm{E}\left( \int_{0}^{T \wedge \tau}\norm{\eta_r}_U^2 + \norm{\tilde{\sy}^l_{r}}_V^2dr\right)\\
    &\leq \mathbbm{E}\left( \int_{0}^{T \wedge \tau}\norm{\eta_r}_U^2 dr\right) + \max_{m \leq L}\mathbbm{E}\left( \int_{0}^{T \wedge \tau} \norm{\tilde{\sy}^m_{r}}_V^2dr\right)  + 1
\end{align*}
where $L$ is taken sufficiently large so that for all $m \geq L$, $$\mathbbm{E}\left( \int_{0}^{T \wedge \tau} \left\vert\norm{\tilde{\sy}^m_{r}}_V^2 - \norm{\sy_{r}}_V^2\right \vert dr\right) \leq \frac{1}{2}.$$
For the first stochastic integral we apply the classical Burkholder-Davis-Gundy Inequality [\cite{burkholder1972integral}], seeing that
\begin{align*}
    2\mathbbm{E}\left(\sup_{t \in [0,T]}\left\vert\int_{0}^{t \wedge \tau}\left\langle B_r,\hat{\sy}^l_{r} \right\rangle_{H}d\mathcal{W}_r\right\vert\right) &\leq c\mathbbm{E}\left(\int_{0}^{T \wedge \tau}\left\Vert\left\langle B_r,\hat{\sy}^l_{r} \right\rangle_{H}\right\Vert^2_{\mathscr{L}^2(\mathfrak{U};\R)}dr\right)^\frac{1}{2}\\
    &\leq c\mathbbm{E}\left(\int_{0}^{T \wedge \tau}\left\Vert B_r\right\Vert^2_{\mathscr{L}^2(\mathfrak{U};H)}\left\Vert\hat{\sy}^l_{r} \right\Vert_{H}^2dr\right)^\frac{1}{2}\\
    &\leq c\mathbbm{E}\left(\sup_{r \in [0,T \wedge \tau]}\left\Vert\hat{\sy}^l_{r} \right\Vert_{H}^2\int_{0}^{T \wedge \tau}\left\Vert B_r\right\Vert^2_{\mathscr{L}^2(\mathfrak{U};H)}dr\right)^\frac{1}{2}\\
    &= c\mathbbm{E}\left(\sup_{r \in [0,T \wedge \tau]}\left\Vert\hat{\sy}^l_{r} \right\Vert_{H}^2\right)^{\frac{1}{2}}\left(\int_{0}^{T \wedge \tau}\left\Vert B_r\right\Vert^2_{\mathscr{L}^2(\mathfrak{U};H)}dr\right)^\frac{1}{2}\\
    &= \frac{1}{2}\mathbbm{E}\left(\sup_{r \in [0,T \wedge \tau]}\left\Vert\hat{\sy}^l_{r} \right\Vert_{H}^2\right) + c\mathbbm{E}\left(\int_{0}^{T \wedge \tau}\left\Vert B_r\right\Vert^2_{\mathscr{L}^2(\mathfrak{U};H)}dr\right)\\
    &= \frac{1}{2}\mathbbm{E}\left(\sup_{t \in I_l \cap (0,\tau] /\{T\}}\norm{\sy_t}_{H}^2\right) + c\mathbbm{E}\left(\int_{0}^{T \wedge \tau}\left\Vert B_r\right\Vert^2_{\mathscr{L}^2(\mathfrak{U};H)}dr\right)
\end{align*}
where $c$ here is a generic constant changing from line to line, independent of $l$. Putting this together we now see that 
\begin{align*}
   \frac{1}{2}\mathbbm{E}\left(\sup_{t \in I_l \cap (0,\tau] /\{T\}}\norm{\sy_t}_{H}^2\right) &\leq \mathbbm{E}\left(\norm{\sy_0}_H^2\right) + \mathbbm{E}\left( \int_{0}^{T \wedge \tau}\norm{\eta_r}_U^2 dr\right)\\ &+ \max_{m \leq L}\mathbbm{E}\left( \int_{0}^{T \wedge \tau} \norm{\tilde{\sy}^m_{r}}_V^2dr\right)  + 1 +  c\mathbbm{E}\left(\int_{0}^{T \wedge \tau}\left\Vert B_r\right\Vert^2_{\mathscr{L}^2(\mathfrak{U};H)}dr\right) \\ &+ 2\mathbbm{E}\left(\sup_{t \in [0,T]}\left\vert\int_0^{t}\left\langle B_r,\sy_{0}\right\rangle_{H}d\mathcal{W}_r\right\vert\right) \\ &+\mathbbm{E}\left[\sup_{t \in I_l \cap (0,\tau] /\{T\}}\sum_{j=0}^{i-1}\left( \left\Vert \int_{t^l_{j}}^{t^l_{j+1} }B_rd\mathcal{W}_r  \right\Vert_{H}^2\right)\right].
\end{align*}
For the stochastic integral involving the initial condition we can treat this identically to generate the bound
$$2\mathbbm{E}\left(\sup_{t \in [0,T]}\left\vert\int_0^{t \wedge \tau}\left\langle B_r,\sy_{0}\right\rangle_{H}d\mathcal{W}_r\right\vert\right) \leq  \mathbbm{E}\left(\norm{\sy_0}_{H}^2\right) + c\mathbbm{E}\left(\int_{0}^{T \wedge \tau}\left\Vert B_r\right\Vert^2_{\mathscr{L}^2(\mathfrak{U};H)}dr\right).$$
As for the final term, it is clear that the supremum over all partitions can be bounded by taking the partition for $t = t^l_{k_l} = T$. Thus
\begin{align*}
  \mathbbm{E}\left[\sup_{t \in I_l \cap (0,\tau] /\{T\}}\sum_{j=0}^{i-1}\left( \left\Vert \int_{t^l_{j}}^{t^l_{j+1} }B_rd\mathcal{W}_r  \right\Vert_{H}^2\right)\right] &\leq \mathbbm{E}\left[\sum_{j=0}^{k_l-1}\left( \left\Vert \int_{t^l_{j}}^{t^l_{j+1} }B_rd\mathcal{W}_r  \right\Vert_{H}^2\right)\right]\\
  &= \sum_{j=0}^{k_l-1}\mathbbm{E}\left( \left\Vert \int_{t^l_{j}}^{t^l_{j+1} }B_rd\mathcal{W}_r  \right\Vert_{H}^2\right)\\
  &= \sum_{j=0}^{k_l-1}\mathbbm{E}\left(  \int_{t^l_{j}}^{t^l_{j+1} }\left\Vert B_r\right\Vert_{\mathscr{L}^2(\mathfrak{U};H)}^2dr  \right)\\
  &=\mathbbm{E}\left(  \int_{0}^{T}\left\Vert B_r\right\Vert_{\mathscr{L}^2(\mathfrak{U};H)}^2dr  \right)
\end{align*}
having applied Proposition \ref{ito isom for cylindrical}. In total then we have that 
\begin{align*}
   \mathbbm{E}\left(\sup_{t \in I_l \cap (0,\tau] /\{T\}}\norm{\sy_t}_{H}^2\right) &\leq 4\mathbbm{E}\left(\norm{\sy_0}_H^2\right) + 2\mathbbm{E}\left( \int_{0}^{T \wedge \tau}\norm{\eta_r}_U^2 dr\right)\\ &+ 2\max_{m \leq L}\mathbbm{E}\left( \int_{0}^{T \wedge \tau} \norm{\tilde{\sy}^m_{r}}_V^2dr\right)  + 2 +  c\mathbbm{E}\left(\int_{0}^{T}\left\Vert B_r\right\Vert^2_{\mathscr{L}^2(\mathfrak{U};H)}dr\right)
\end{align*}
which is a finite bound independent of $l$. As $I_l \subset I_{l+1}$ then the sequence $$\sup_{t \in I_l \cap (0,\tau] /\{T\}}\norm{\sy_t}_{H}^2$$ is $\mathbbm{P}-a.s.$ monotone increasing in $l$, so we can apply the Monotone Convergence Theorem to see that \begin{equation}\label{abbabba}\mathbbm{E}\left(\sup_{t \in I \cap (0,\tau] /\{T\}}\norm{\sy_t}_{H}^2\right) = \lim_{l \rightarrow \infty}\mathbbm{E}\left(\sup_{t \in I_l \cap (0,\tau] /\{T\}}\norm{\sy_t}_{H}^2\right) < \infty.\end{equation} Thus for $\mathbbm{P}-a.e.$ $\omega$, we have that
$$\sup_{t \in I \cap (0,\tau(\omega)] /\{T\}}\norm{\sy_t(\omega)}_{H}^2 = \tilde{c} < \infty$$ and $\sy_{\cdot}(\omega) \in C([0,T];U)$. We fix such an $\omega$ and any $t \in [0,\tau(\omega)]$. As the mesh of the partitions go to zero then there there is a sequence of times $(t_n)$ in $I \cap (0,\tau(\omega)] /\{T\}$ such that $t_n \longrightarrow t$.  The sequence $\left(\sy_{t_n}(\omega)\right)$ is uniformly bounded in $H$ so admits a weakly convergent subsequence in this space, to a limit which we call $\psi$. From the continuous embedding of $H$ into $U$ this weak convergence also holds in $U$, but from the continuity of $\sy_{\cdot}(\omega)$ in $U$ we have that $\left(\sy_{t_n}(\omega)\right)$ converges strongly and therefore weakly to $\sy_t(\omega)$ in $U$. By the uniqueness of limits in the weak topology, we conclude that $\sy_t(\omega) = \psi$ and thus belongs to $H$. Moreover the weak limit preserves the boundedness in $H$, so $\norm{\sy_t(\omega)}_H \leq \tilde{c}.$ Therefore
$$\sup_{t \in I \cap (0,\tau(\omega)] /\{T\}}\norm{\sy_t(\omega)}_{H}^2 = \sup_{t \in [0,\tau(\omega)] }\norm{\sy_t(\omega)}_{H}^2$$ for our fixed $\omega$ in a full measure set, and thus $\mathbbm{P}-a.s.$. We also know that $\sy_{\cdot} = \sy_{\cdot \wedge \tau}$ from the identity (\ref{newest identity}), so this equality extends to $$\sup_{t \in I \cap (0,\tau(\omega)] /\{T\}}\norm{\sy_t(\omega)}_{H}^2 = \sup_{t \in [0,T] }\norm{\sy_t(\omega)}_{H}^2.$$ Combining this with (\ref{abbabba}) concludes the proof. 
\end{proof}

Having now justified that for $\mathbbm{P}-a.e.$ $\omega$ $\sy_t(\omega)\in H$, we move on to prove weak continuity in this space.

\begin{lemma} \label{weak continuity for ito}
    For $\mathbbm{P}-a.e.$ $\omega$, $\sy_{\cdot}(\omega)$ is weakly continuous in $H$. 
\end{lemma}

\begin{proof}
We fix $t\in [0,T]$ but now take any sequence of times $(t_n)$ such that $t_n \rightarrow t$. For $\mathbbm{P}-a.e.$ $\omega$ and any given $\phi \in H$ we must justify that 
\begin{equation} \label{next thing to justify}
   \lim_{n \rightarrow \infty} \inner{\sy_{t_n}(\omega)-\sy_t(\omega)}{\phi}_H = 0.
\end{equation}
For any $\varepsilon > 0$, from the density of $V$ in $H$ there exists $\phi^k \in V$ such that  $$\norm{\phi - \phi^k}_H < \frac{\varepsilon}{4\norm{\sy(\omega)}_{L^\infty([0,T];H)}}$$
and then from the continuity in $U$ there exists an $N \in \N$ sufficiently large such that for all $n \geq N$,
$$\norm{\sy_{t_n}(\omega)-\sy_t(\omega)}_U <  \frac{\varepsilon}{2\norm{\phi^k}_V}.$$
Putting this together, 
\begin{align*}
    \abs{\inner{\sy_{t_n}(\omega)-\sy_t(\omega)}{\phi}_H} & \leq \abs{\inner{\sy_{t_n}(\omega)-\sy_t(\omega)}{\phi - \phi^k}_H} + \abs{\inner{\sy_{t_n}(\omega)-\sy_t(\omega)}{\phi^k}_H}\\
    &\leq \norm{\sy_{t_n}(\omega)-\sy_t(\omega)}_{H}\norm{\phi - \phi^k}_H + \abs{\inner{\sy_{t_n}(\omega)-\sy_t(\omega)}{\phi^k}_{U \times V}}\\
    &\leq 2\norm{\sy(\omega)}_{L^\infty([0,T];H)}\norm{\phi - \phi^k}_H + \norm{\sy_{t_n}(\omega)-\sy_t(\omega)}_U\norm{\phi^k}_{V}\\
    &< \varepsilon
\end{align*}
as required.

%with the goal to show that $\sy_{t_n}(\omega) \rightarrow \sy_t(\omega)$ in $H$. From the density of $V$ in $H$ there exists an orthogonal basis $(a_k)$ in $H$ such that each $a_k \in V$ (such a basis can be deduced by taking a basis in $V$ and applying the Gram-Schmidt orthonormalisation procedure in $H$). Observe that
%\begin{align*}
%    \norm{\sy_{t_n}(\omega)-\sy_t(\omega)}_H^2 &= \sum_{k=1}^\infty \inner{\sy_{t_n}(\omega)-\sy_t(\omega)}{a_k}_H^2\\
 %   &= \sum_{k=1}^\infty \inner{\sy_{t_n}(\omega)-\sy_t(\omega)}{a_k}_{U \times V}^2. 
%\end{align*}
%Again from the continuity in $U$ we have that for each fixed $k$, $$\lim_{n \rightarrow \infty}\inner{\sy_{t_n}(\omega)-\sy_t(\omega)}{a_k}_{U \times V}^2 = 0$$ so it is sufficient to justify the interchange of limit and summation. To do this we invoke the Dominated Convergence Theorem, where for any $n\in \N$ we have that
%\begin{align*}
 %   \sum_{k=1}^\infty \inner{\sy_{t_n}(\omega)-\sy_t(\omega)}{a_k}_{U \times V}^2 &\leq 2\sum_{k=1}^\infty \inner{\sy_{t_n}(\omega)}{a_k}_{U \times V}^2 + 2\sum_{k=1}^\infty \inner{\sy_t(\omega)}{a_k}_{U \times V}^2\\
  %  &= 2\sum_{k=1}^\infty\inner{\sy_{t_n}(\omega)}{a_k}_{H}^2 + 2\sum_{k=1}^\infty \inner{\sy_t(\omega)}{a_k}_{H}^2\\
 %   &=2\norm{\sy_{t_n}(\omega)}_H^2 + 2\norm{\sy_{t}(\omega)}_H^2\\
 %   &\leq 4c
%\end{align*}
%which is a bound

\end{proof}

\begin{proposition} \label{first statement}
    The equality 
  \begin{align} \label{ito big dog}\norm{\sy_t}^2_{H} = \norm{\sy_0}^2_{H} + \int_0^{t\wedge \tau} \bigg( 2\inner{\eta_s}{\sy_s}_{U \times V} + \norm{B_s}^2_{\mathscr{L}^2(\mathcal{U};H)}\bigg)ds + 2\int_0^{t \wedge \tau}\inner{B_s}{\sy_s}_{H}d\mathcal{W}_s\end{align}
  holds $\mathbbm{P}-a.s.$ in $\R$ for all $t \in [0,T]$. Moreover for $\mathbbm{P}-a.e.$ $\omega$, $\sy_{\cdot}(\omega) \in C([0,T];H)$. 
\end{proposition}

\begin{proof}
 The proof now directly follows from claim $c$ onwards in the proof of Theorem 4.2.5 of [\cite{prevot2007concise}].

\end{proof}

We wish to relax the integrability constraints over $\Omega$, in accordance with Definition \ref{definitionofregularsolution}. In the same setting of the Hilbert Spaces $V,H,U$, we impose the new assumptions for some $T>0$ and stopping time $\tau$:
\begin{enumerate}
        \item $\sy_0:\Omega \rightarrow H$ is $\mathcal{F}_0-$measurable;
        \item For $\mathbbm{P}-a.e.$ $\omega$, $\eta(\omega) \in L^2([0,T];U)$;
        \item $B \in \bar{\mathcal{I}}^H(\mathcal{W})$;
        \item  \label{4*} For $\mathbbm{P}-a.e.$ $\omega$, $\sy_{\cdot}(\omega)\mathbbm{1}_{\cdot \leq \tau(\omega)} \in L^2([0,T];V)$ and $\sy_{\cdot}\mathbbm{1}_{\cdot \leq \tau}$ is progressively measurable in $V$;
        \item \label{item 5 again*} The identity
        \begin{equation} \label{newest identity*}
            \sy_t = \sy_0 + \int_0^{t \wedge \tau}\eta_sds + \int_0^{t \wedge \tau}B_s d\mathcal{W}_s
        \end{equation}
        holds $\mathbbm{P}-a.s.$ in $U$ for all $t \in [0,T]$.
    \end{enumerate}

We restate Proposition \ref{first statement} for the new setting.

\begin{proposition} \label{a new first statement}
    The equality 
  \begin{align} \label{ito big dog*}\norm{\sy_t}^2_{H} = \norm{\sy_0}^2_{H} + \int_0^{t\wedge \tau} \bigg( 2\inner{\eta_s}{\sy_s}_{U \times V} + \norm{B_s}^2_{\mathscr{L}^2(\mathcal{U};H)}\bigg)ds + 2\int_0^{t \wedge \tau}\inner{B_s}{\sy_s}_{H}d\mathcal{W}_s\end{align}
  holds $\mathbbm{P}-a.s.$ in $\R$ for any $t \in [0,T]$. Moreover for $\mathbbm{P}-a.e.$ $\omega$, $\sy_{\cdot}(\omega) \in C([0,T];H)$. 
\end{proposition}

\begin{proof}
The idea is simply to apply Proposition \ref{first statement} for some truncated versions of the processes. We consider the stopping times \begin{align*}
 \tau^1_n&:= n \wedge \inf\{0 \leq t < \infty: \int_0^t\norm{\eta_s}_{U}^2ds \geq n\}\\
 \tau^2_n&:= n \wedge \inf\{0 \leq t < \infty: \int_0^t\norm{B_s}_{\mathscr{L}^2(\mathfrak{U};H)}^2ds \geq n\}\\
    \tau^3_n&:= n \wedge \inf\{0 \leq t < \infty: \int_0^t\norm{\sy_s\mathbbm{1}_{s \leq \tau}}_{V}^2ds \geq n\}\\
    \tau_n &:= \tau^1_n \wedge \tau^2_n \wedge \tau^3_n.
\end{align*}
Then for every $n$, we have that
\begin{align*}
    \sy_0\mathbbm{1}_{\norm{\sy_0}_H \leq n}, \qquad 
     \eta_{\cdot}\mathbbm{1}_{\cdot \leq \tau_n}\mathbbm{1}_{\norm{\sy_0}_H \leq n}\\
    B_{\cdot}\mathbbm{1}_{\cdot \leq \tau_n}\mathbbm{1}_{\norm{\sy_0}_H \leq n}, \qquad
    \sy_{\cdot \wedge \tau_n}\mathbbm{1}_{\norm{\sy_0}_H \leq n}
\end{align*}
satisfy the previous assumptions of \ref{11111}, \ref{22222}, \ref{33333}, \ref{4}. Moreover from (\ref{newest identity*}) we have that for any $t \in [0,T]$
\begin{align*}
    \sy_{t\wedge\tau_n} &= \sy_0 + \int_0^{t \wedge \tau \wedge \tau_n}\eta_sds + \int_0^{t \wedge \tau \wedge \tau_n}B_s d\mathcal{W}_s\\
    &= \sy_0 + \int_0^{t \wedge \tau}\eta_s\mathbbm{1}_{s \leq \tau_n}ds + \int_0^{t \wedge \tau}B_s\mathbbm{1}_{s \leq \tau_n} d\mathcal{W}_s
\end{align*}
$\mathbbm{P}-a.s.$ and moreover 
\begin{align*}
    \sy_{t\wedge\tau_n}\mathbbm{1}_{\norm{\sy_0}_H \leq n} &= \sy_0\mathbbm{1}_{\norm{\sy_0}_H \leq n} + \mathbbm{1}_{\norm{\sy_0}_H \leq n}\int_0^{t \wedge \tau}\eta_s\mathbbm{1}_{s \leq \tau_n}ds + \mathbbm{1}_{\norm{\sy_0}_H \leq n}\int_0^{t \wedge \tau}B_s\mathbbm{1}_{s \leq \tau_n} d\mathcal{W}_s\\
    &= \sy_0\mathbbm{1}_{\norm{\sy_0}_H \leq n} + \int_0^{t \wedge \tau}\eta_s\mathbbm{1}_{s \leq \tau_n}\mathbbm{1}_{\norm{\sy_0}_H \leq n}ds + \int_0^{t \wedge \tau}B_s\mathbbm{1}_{s \leq \tau_n}\mathbbm{1}_{\norm{\sy_0}_H \leq n} d\mathcal{W}_s
\end{align*}
having applied Proposition \ref{real valued bounded take it in cylindrical}. Therefore we can apply Proposition \ref{first statement} to see that the equality
\begin{align*}
    \norm{\sy_{t\wedge\tau_n}\mathbbm{1}_{\norm{\sy_0}_H \leq n}}^2_{H} &= \norm{\sy_0\mathbbm{1}_{\norm{\sy_0}_H \leq n}}^2_{H} + \int_0^{t\wedge \tau} 2\inner{\eta_s\mathbbm{1}_{s \leq \tau_n}\mathbbm{1}_{\norm{\sy_0}_H \leq n}}{\sy_{s \wedge \tau_n}\mathbbm{1}_{\norm{\sy_0}_H \leq n}}_{U \times V}ds \\ &+ \int_0^{t \wedge \tau}\norm{B_s\mathbbm{1}_{s \leq \tau_n}\mathbbm{1}_{\norm{\sy_0}_H \leq n}}^2_{\mathscr{L}^2(\mathcal{U};H)}ds\\ &+ 2\int_0^{t \wedge \tau}\inner{B_s\mathbbm{1}_{s \leq \tau_n}\mathbbm{1}_{\norm{\sy_0}_H \leq n}}{\sy_{s \wedge \tau_n}\mathbbm{1}_{\norm{\sy_0}_H}}_{H}d\mathcal{W}_s
\end{align*}
holds for all $t \in [0,T]$ $\mathbbm{P}-a.s.$. We rewrite this as 
\begin{align*}
    \mathbbm{1}_{\norm{\sy_0}_H \leq n}\norm{\sy_{t\wedge\tau_n}}^2_{H} &= \mathbbm{1}_{\norm{\sy_0}_H \leq n}\norm{\sy_0}^2_{H} + \mathbbm{1}_{\norm{\sy_0}_H \leq n}\int_0^{t\wedge \tau \wedge \tau_n} 2\inner{\eta_s}{\sy_{s}}_{U \times V}ds \\ &+ \mathbbm{1}_{\norm{\sy_0}_H \leq n}\int_0^{t \wedge \tau \wedge \tau_n}\norm{B_s}^2_{\mathscr{L}^2(\mathcal{U};H)}ds\\ &+ \mathbbm{1}_{\norm{\sy_0}_H \leq n}2\int_0^{t \wedge \tau \wedge \tau_n}\inner{B_s}{\sy_{s}}_{H}d\mathcal{W}_s.
\end{align*}
Therefore for any $t \in [0,T]$, $\mathbbm{P}-a.e.$ $\omega$ with $n$ sufficiently large so that  $\norm{\sy_0(\omega)}_H \leq n$ and $t \leq \tau_n(\omega)$, the identity (\ref{newest identity*}) holds. We can always find such a large enough $n$, which completes the justification of this identity. The continuity then follows identically as we have again from Proposition \ref{first statement} that for every $n$ and $\mathbbm{P}-a.e.$ $\omega$, $\sy_{\cdot \wedge \tau_n}(\omega) \in C([0,T];H)$, and we conclude the proof.
\end{proof}

\subsection{The General It\^{o} Formula}
In this subsection we state only a general result from the literature and do not prove it ourselves. The purpose of this brief inclusion is for the reader to understand the approach, and have direction to further study along these lines elsewhere. The key consideration is the notion of derivatives, for which we use the standard functional analytic choice of the Fr\'{e}chet Derivative. For a Fr\'{e}chet differentiable function between Banach Spaces $\mathcal{K}:\mathcal{Y}\rightarrow \mathcal{Z}$, recall that its derivative is defined as a mapping $\mathcal{K}':\mathcal{Y} \rightarrow \mathscr{L}(\mathcal{Y};\mathcal{Z})$ and subsequently its second derivative as $\mathcal{K}'':\mathcal{Y}\rightarrow \mathscr{L}\big(\mathcal{Y}; \mathscr{L}(\mathcal{Y};\mathcal{Z})\big)$. In the case where $\mathcal{Y}=\mathcal{H}$ is a Hilbert Space and $\mathcal{Z}=\R$, then \begin{align*}
    \mathcal{K}'&:\mathcal{H}\rightarrow\mathcal{H}^* \\
    \mathcal{K}''&:\mathcal{H}\rightarrow \mathscr{L}(\mathcal{H};\mathcal{H}^*)
\end{align*} so it is commonplace to apply the Riesz Representation and make the identification
\begin{align}
    \mathcal{K}'&:\mathcal{H}\rightarrow\mathcal{H} \label{frech1}\\
    \mathcal{K}''&:\mathcal{H}\rightarrow \mathscr{L}(\mathcal{H}) \label{frech2}.
\end{align}
We will be considering functions $F:[0,\infty) \times U \rightarrow \R$ with partial derivatives $F_t, F_x$ and $F_{xx}$ where the latter two are understood in the sense of (\ref{frech1}) and (\ref{frech2}).

\begin{theorem} \label{Ito formula}
    Suppose $(\sy,\tau)$ is a local strong solution of (\ref{thespde*}), and $F:[0,\infty) \times U \rightarrow \R$ is such that $F_t, F_x$ and $F_{xx}$ are uniformly continuous on bounded subsets of $[0,T] \times U$. Then $F(\cdot,\sy_{\cdot})$ satisfies the identity \begin{align*}F(t,\sy_t) = F(0,\sy_0) &+ \int_0^{t\wedge\tau} \bigg( F_s(s,\sy_s) + \inner{F_x(s,\sy_s)}{\mathcal{Q}(s,\sy_s)}_{U}\bigg)ds\\ &+ \frac{1}{2}\int_0^{t \wedge \tau}\textnormal{Tr}\Big(F_{xx}(s,\sy_s)(\mathcal{G}(s,\sy_s))(\mathcal{G}(s,\sy_s))^*\Big)ds\\ &+ \int_0^{t \wedge \tau}\inner{F_x(s,\sy_s)}{\mathcal{G}(s,\sy_s)}_{U}d\mathcal{W}_s\end{align*} $\mathbbm{P}-a.s.$ in $\R$ for all $t \in [0,T]$, where the trace term is simply a composition of operators and the stochastic integral is understood component wise as in Theorem \ref{operatorthroughstochasticintegral}.
\end{theorem}

We note that through the function $F:h \mapsto \norm{h}^2_{U}$, one can establish an alternative energy equality to Proposition \ref{first statement}.

\begin{proposition} \label{ogenergy}
    Suppose $(\sy,\tau)$ is a local strong solution of (\ref{thespde*}). Then $\sy$ satisfies the energy equality \begin{align*}\norm{\sy_t}^2_{U} = \norm{\sy_0}^2_{U} + \int_0^{t\wedge \tau} \bigg( 2\inner{\sy_s}{\mathcal{Q}(s,\sy_s)}_{U} + \norm{\mathcal{G}(s,\sy_s)}^2_{\mathscr{L}^2(\mathcal{U};U)}\bigg)ds + 2\int_0^{t \wedge \tau}\inner{\sy_s}{\mathcal{G}(s,\sy_s)}_{U}d\mathcal{W}_s\end{align*} $\mathbbm{P}-a.s.$ in $\R$ for all $t \geq 0$.
\end{proposition}

%We also have a very important result if we assume some additional structure to our framework. Suppose in addition that there is a continuous bilinear form $\inner{\cdot}{\cdot}_{V \times U}: V \times U \rightarrow \R$ such that for $\phi \in V$ and $\psi \in H$, \begin{equation} \label{bilinear form}
 %   \inner{\phi}{\psi}_{V \times U} =  \inner{\phi}{\psi}_{H}.
%\end{equation}
%In this setting we also have the following.
%\begin{proposition} \label{ogenergy with bilinear}
 %   Suppose $(\sy,\tau)$ is a local strong solution of (\ref{thespde*}). Then $\sy$ satisfies the energy equality \begin{align*}\norm{\sy_t}^2_{H} = \norm{\sy_0}^2_{H} &+ \int_0^{t\wedge \tau} \bigg( 2\inner{\sy_s}{\mathcal{Q}(s,\sy_s)}_{V \times U} + \norm{\mathcal{G}(s,\sy_s)}^2_{\mathscr{L}^2(\mathcal{U};H)}\bigg)ds\\ &+ 2\int_0^{t \wedge \tau}\inner{\sy_s}{\mathcal{G}(s,\sy_s)}_{H}d\mathcal{W}_s\end{align*} $a.e.$ in $\R$ for all $t \geq 0$.
%\end{proposition}

We do not prove these results here, but instead refer to [\cite{da2014stochastic}] Theorem 4.18 and related discussions therein. Also shown there is the It\^{o} Formula for processes defined by a more general evolution equation in a Hilbert Space $\mathcal{H}$, \begin{equation} \label{the one for phi}\py_t = \py_0 + \int_0^t\phi_sds + \int_0^t B(s)d\mathcal{W}_s\end{equation} for any given $\py_0:\Omega \rightarrow \mathcal{H}$ $\mathcal{F}_0-$measurable, $\phi:\Omega \times [0,t] \rightarrow \mathcal{H}$ progressively measurable and $\mathbbm{P}-a.s.$ Bochner Integrable and $B \in \bar{\mathcal{I}}^{\mathcal{H}}(\mathcal{W})$. Then we have the corresponding result to Theoerem \ref{Ito formula}.

\begin{theorem} \label{the ito that we use}
Suppose that $F$ is as in Theorem \ref{Ito formula} for $\mathcal{H}$ replacing $U$ and $\py$ is defined by (\ref{the one for phi}). Then $F(\cdot,\py_{\cdot})$ satisfies the identity \begin{align*}F(t,\py_t) = F(0,\py_0) &+ \int_0^{t\wedge\tau} \bigg( F_s(s,\py_s) + \inner{F_x(s,\py_s)}{\phi_s}_{\mathcal{H}}\bigg)ds\\ &+ \frac{1}{2}\int_0^{t \wedge \tau}\textnormal{Tr}\Big(F_{xx}(s,\py_s)(B(s))(B(s))^*\Big)ds\\ &+ \int_0^{t \wedge \tau}\inner{F_x(s,\py_s)}{B(s)}_{\mathcal{H}}d\mathcal{W}_s\end{align*} $\mathbbm{P}-a.e.$ in $\R$ for every $t \in [0,T]$.
\end{theorem}

\subsection{The Case of Constant Multiplicative Noise}

Many techniques in proving existence and uniqueness of an SPDE in this framework rely on simplifying the equation to one where we can apply the standard theory, and then constructing solutions in the original framework via some appropriate limit of solutions to the simplified equations. As such we shall briefly considered a special type of equation in this framework which reduces the driving noise from something infinite dimensional to one dimensional. For this we work again with an arbitrary Hilbert Space $\mathcal{H}$.

\begin{proposition} \label{hjkl}
    Suppose that $\sy \in \bar{\mathcal{I}}^{\mathcal{H}}$ and that the operator $\mathcal{G}:\mathfrak{U} \times \mathcal{H} \rightarrow \mathcal{H}$ is such that $$\mathcal{G}_i: \phi \mapsto \lambda_i\phi$$ for each $i$ with $\lambda_i \in \R$, and \begin{equation}\label{squaresummable}\sum_{i=1}^\infty \lambda_i^2 < \infty.\end{equation} Then $\mathcal{G}\sy \in \bar{\mathcal{I}}^{\mathcal{H}}(\mathcal{W})$ and there exists a real valued Brownian Motion $W$ such that  \begin{equation}\label{then}\int_0^t \mathcal{G}\sy_s d\mathcal{W}_s = \bigg(\sum_{i=1}^\infty \lambda_i^2\bigg)^{1/2} \int_0^t\sy_sdW_s\end{equation} $\mathbbm{P}-a.s.$ for all $t \geq 0$.
\end{proposition}

In order to prove the above, we use an intermediary lemma.

\begin{lemma} \label{infsumbrownian}
    In the setting of Proposition \ref{hjkl} the infinite sum \begin{equation}\label{anotherinfsum} M_s := \sum_{i=1}^\infty \lambda_iW^i_s\end{equation} is convergent in $L^2\big(\Omega;\R)$ at every $s$, and the limiting martingale has the representation \begin{equation} \label{requiredform} M = \bigg(\sum_{i=1}^\infty \lambda_i^2\bigg)^{1/2}W\end{equation} for some real valued Brownian Motion $W$, $\mathbbm{P}-a.s.$ for all $t \geq 0$. 
\end{lemma}

\begin{proof}[Proof of Proposition \ref{infsumbrownian}]
Firstly let's verify that the convergence in (\ref{anotherinfsum}) does indeed hold, which is immediate from observing that $$\sum_{i=1}^\infty \lambda_iW^i_s = \sum_{i=1}^\infty \int_0^s\lambda_idW^i_r$$ which is simply the stochastic integral $$\int_0^sP(r)d\mathcal{W}_r$$ for the process $P \in \mathcal{I}^{\R}(\mathcal{W})$ defined by $P_{e_i}(r) = \lambda_i$. So $M$ is a continuous genuine martingale, which we show is of the form (\ref{requiredform}) through Levy's Characterisation of Brownian Motion, e.g. [\cite{karatzas1991brownian}] Theorem 3.16 pp.157. Indeed the quadratic variation process $[M]$ is deduced from Lemma \ref{martingale quadratic variation limit}, where our approximating sequence of martingales $$M^n = \sum_{i=1}^n \lambda_i W^i_s$$ have quadratic variation $$[M^n]_s = \sum_{i=1}^n\lambda_i^2s$$ which of course converges in $L^1\big(\Omega;\R\big)$ to the infinite sum, from which we conclude
    $$[M]_s = \sum_{i=1}^\infty \lambda_i^2s.$$
    Therefore $$\Bigg[\frac{M}{\big(\sum_{i=1}^\infty \lambda_i^2\big)^{1/2}}\Bigg]_s = s$$ and we immediately deduce the representation (\ref{requiredform}) from Levy's Characterisation.

\end{proof}

\begin{proof} [Proof of \ref{hjkl}]
    As $\mathcal{G}:\mathcal{H} \rightarrow \mathcal{L}^2(\mathcal{U};\mathcal{H})$ is linear and also bounded from the observation that $$\norm{\mathcal{G}\phi}_{\mathcal{L}^2(\mathcal{U};\mathcal{H})}^2 = \sum_{i=1}^\infty\norm{\lambda_i\phi}_{\mathcal{H}}^2 = \sum_{i=1}^\infty\lambda_i^2\norm{\phi}_{\mathcal{H}}^2$$ then it is continuous as a mapping between these spaces so preserves the progressive measurability, and evidently the required boundedness to deduce that $\mathcal{G}\sy \in \bar{\mathcal{I}}^{\mathcal{H}}(\mathcal{W})$. To show the identity (\ref{then}) let's rewrite $$\int_0^t \mathcal{G}\sy_s d\mathcal{W}_s = \lim_{n \rightarrow \infty}\sum_{i=1}^n\int_0^t\lambda_i\sy_sdW^i_s = \lim_{n \rightarrow \infty} \int_0^t\sy_sdM^n_s$$ with notation $M^n$ as in Propostion \ref{infsumbrownian}, understanding once more that this limit is taken in $L^2\big(\Omega;\mathcal{H}\big)$ for the stopped integrals. The localising stopping times $(\tau_m)$ defined by $$\tau_m = m \wedge \inf\{ 0 \leq t < \infty: \int_0^t\sum_{i=1}^\infty \lambda_i^2\norm{\sy_s}_{\mathcal{H}}^2ds \geq m\}$$ are precisely as in Definition \ref{cylindricalintlocal} and (\ref{R_n}) to define the left and right sides of (\ref{then}) respectively. It is therefore sufficient to show that for any $m$, $$\lim_{n \rightarrow \infty} \int_0^t\sy_s\mathbbm{1}_{s \leq \tau_m}dM^n_s = \int_0^t\sy_s\mathbbm{1}_{s \leq \tau_m}dM_s$$ having simply inserted the representation (\ref{requiredform}) into our required identity. In other words we want that
    
    $$\mathbbm{E}\Big \vert \Big \vert \int_0^t\sy_s\mathbbm{1}_{s \leq \tau_m}dM_s - \int_0^t\sy_s\mathbbm{1}_{s \leq \tau_m}dM_s^n \Big \vert \Big \vert_{\mathcal{H}}^2 \longrightarrow 0$$ as $n \rightarrow \infty$ which is equivalent to the statement 
     $$\mathbbm{E}\Big \vert \Big \vert \int_0^t\sy_s\mathbbm{1}_{s \leq \tau_m}d(M-M^n)_s\Big \vert \Big \vert_{\mathcal{H}}^2 \longrightarrow 0.$$
     The same arguments of Proposition $\ref{infsumbrownian}$ afford us that the martingale $M-M^n$ which is given at each time $s$ by $$(M-M_n)_s = \sum_{i=n+1}^\infty\lambda_iW^i_s$$ has the representation $$M-M^n = \bigg(\sum_{i=n+1}^\infty \lambda_i^2\bigg)^{1/2} V$$ for a standard Brownian Motion $V$. So we have that
     
     \begin{align*}
         \mathbbm{E}\bigg \vert \bigg \vert  \int_0^t\sy_s\mathbbm{1}_{s \leq \tau_m}d(M-M^n)_s\bigg \vert \bigg \vert_{\mathcal{H}}^2 &= \mathbbm{E}\bigg \vert \bigg \vert \bigg(\sum_{i=n+1}^\infty \lambda_i^2\bigg)^{1/2} \int_0^t\sy_s\mathbbm{1}_{s \leq \tau_m}dV_s\bigg \vert \bigg \vert_{\mathcal{H}}^2\\ &= \bigg(\sum_{i=n+1}^\infty \lambda_i^2\bigg)\mathbbm{E}\bigg \vert \bigg \vert  \int_0^t\sy_s\mathbbm{1}_{s \leq \tau_m}dV_s\bigg \vert \bigg \vert_{\mathcal{H}}^2\\
         &= \bigg(\sum_{i=n+1}^\infty \lambda_i^2\bigg) \mathbbm{E} \int_0^t \norm{\sy_s\mathbbm{1}_{s \leq \tau_m}}_{\mathcal{H}}^2ds
     \end{align*}
     having used the It\^{o} Isometry \ref{realItoIsom}. By definition of the stopping time the integral is bounded uniformly in $\omega$ ($\mathbbm{P}-a.s.$) hence has finite expectation, so we conclude that this approaches zero in the limit from the fact that $\sum_{i=1}^\infty \lambda_i^2 < \infty$.
    
    \end{proof}

%\subsection{Applications}

%We have alluded quite heavily to the application of this framework for SALT [\cite{holm2015variational}] derived SPDEs, which is done for now in [\cite{goodair2023existence}] and to be expanded upon in [\cite{Goodair1},\cite{Goodair2}]. In [\cite{goodair2023existence}] we establish an abstract solution method in the context of the Hilbert Spaces $$V \xhookrightarrow{} H \xhookrightarrow{} U \xhookrightarrow{} X$$ posing additional assumptions on the operators to the ones given in \ref{Qassumpt2}, \ref{Gassumpt2}. In particular this is applied to the Stratonovich SPDE
%\begin{equation} \label{number2equation}
 %   u_t - u_0 + \int_0^t\mathcal{L}_{u_s}u_s\,ds - \int_0^t \Delta u_s\, ds + \int_0^t Bu_s \circ d\mathcal{W}_s + \nabla \rho_t = 0
%\end{equation}
%where $B$ here is a differential (transport) operator and $\mathcal{L}$ is the usual fluids non-linear term. Clearly this is a highly non-trivial SPDE, but we show that this fits our framework for (\ref{stratoSPDE}) to swiftly convert into the It\^{o} Form for which we then apply the general existence and uniqueness arguments in the paper. The notion of $V-$valued solution in [\cite{goodair2023existence}] is thus what corresponds precisely and rigorously to a solution of the Stratonovich form as laid out in this text. Indeed the general existence argument also used the result \ref{Skorotheorem}.\\

%We hope that this very general framework can facilitate a much easier analysis of many SPDEs in the future. 

\section{Existence Theory for Nonlinear Stochastic Partial Differential Equations} \label{section 3}

\subsection{An Existence and Uniqueness Result in Finite Dimensions}

As suggested in the introduction, our techniques in the existence theory centre around taking finite dimensional approximations and using somewhat familiar theory. This scheme is referred to as a \textit{Galerkin Approximation}, used traditionally in the analysis for highly non-trivial PDEs such as the Euler and Navier-Stokes Equations (see [\cite{temam1974euler}, \cite{robinson2016three}] for example) so offers a very reasonable first suggestion for the study of their stochastic counterparts. This approach will only work if we can quickly deduce the existence and uniqueness of solutions of the finite dimensional system, which we do in this subsection. It should be noted however that we still work with the Cylindrical Brownian Motion $\mathcal{W}$, over the same infinite dimensional Hilbert Space $\mathfrak{U}$; it is only the space in which the equation satisfies its identity that is assumed finite dimensional.

\begin{theorem} \label{Skorotheorem}
Fix a finite-dimensional Hilbert Space $\mathcal{H}$. Suppose the following:
\begin{itemize}
    %\item[1:] $\mathscr{A}: [0,\infty) \times \mathcal{H} \rightarrow \mathcal{H}$, and for every $T>0$ we have that $\mathscr{A} \in C\left([0,T] \times \mathcal{H};\mathcal{H}\right)$. Similarly $\mathscr{G}: [0,\infty) \times \mathcal{H} \rightarrow \mathscr{L}^2\left(\mathfrak{U};\mathcal{H}\right)$, and for every $T>0$ we have that $\mathscr{G} \in C\left([0,T] \times \mathcal{H};\mathscr{L}^2\left(\mathfrak{U};\mathcal{H}\right)\right)$;\\
    \item[1:] For any $T>0$, the operators $\mathscr{A}:[0,T] \times \mathcal{H} \rightarrow \mathcal{H}$ and $\mathscr{G}:[0,T] \times \mathcal{H} \rightarrow \mathscr{L}^2(\mathfrak{U};\mathcal{H})$ are measurable;\\
    \item[2:] There exists a $C_{\cdot}:[0,\infty) \rightarrow \R$ bounded on $[0,T]$ for every $T$, and constants $c_i$ such that for every $\boldsymbol{\phi}, \boldsymbol{\psi} \in \mathcal{H}$ and $t \in [0,\infty)$, \begin{align*}\norm{\mathscr{A}(t,\boldsymbol{\phi})}^2_{\mathcal{H}}  &\leq C_t\left[1 + \norm{\boldsymbol{\phi}}_{\mathcal{H}}^2\right]\\
     \norm{\mathscr{G}_i(t,\boldsymbol{\phi})}^2_{\mathcal{H}} &\leq C_tc_i\left[1 + \norm{\boldsymbol{\phi}}_{\mathcal{H}}^2\right]\\
     \sum_{i=1}^\infty c_i &< \infty\\
    \norm{\mathscr{A}(t,\boldsymbol{\phi}) - \mathscr{A}(t,\boldsymbol{\psi})}^2_{\mathcal{H}} &+\sum_{i=1}^\infty \norm{\mathscr{G}_i(t,\boldsymbol{\phi}) - \mathscr{G}_i(t,\boldsymbol{\psi})}^2_{\mathcal{H}} \leq C_t \norm{\boldsymbol{\phi}-\boldsymbol{\psi}}_{\mathcal{H}}^2
    \end{align*}
    \item[3:] $\py_0 \in L^2(\Omega;\mathcal{H})$.\\
\end{itemize}
Then there exists a process $\py:[0,\infty) \times \Omega \rightarrow \mathcal{H}$ such that for $\mathbbm{P}-a.e.$ $\omega$, $\py_\cdot(\omega) \in C\left([0,T];\mathcal{H}\right)$ for every $T>0$, $\py$ is progressively measurable in $\mathcal{H}$ and the identity \begin{equation}\label{identityingalerkinsolution}\py_t = \py_0 + \int_0^t \mathscr{A}(s,\py_s)ds + \int_0^t \mathscr{G}(s,\py_s)d\mathcal{W}_s\end{equation} holds $\mathbbm{P}-a.s.$ in $\mathcal{H}$ for every $t \geq 0$.
\end{theorem}

We remark that the operators satisfy the assumptions of \ref{Qassumpt2}, \ref{Gassumpt2} for the spaces $V=H=U:= \mathcal{H}$ and that the conclusion of this theorem is the existence of a strong solution of (\ref{identityingalerkinsolution}) in the sense of Definition \ref{definitionofregularsolutionglobaltime}.

\begin{proof}
With the finite-dimesnionality of $\mathcal{H}$ in place, we first restrict ourselves to finitely many Brownian Motions in our stochastic integral to make things classical. Fixing any $T > 0$, let's define the operator $\mathscr{G}^k: [0,T]  \times \mathcal{H} \times \mathfrak{U} \rightarrow \mathcal{H}$ on the basis vectors $(e_i)$ of $\mathfrak{U}$ by $\mathscr{G}^k(\cdot,\cdot,e_i) = \mathscr{G}(\cdot,\cdot,e_i)$ for $i \leq k$, and zero otherwise. We consider at first the equation
$$\py^k_t = \py^k_0 + \int_0^t \mathscr{A}(s,\py^k_s)ds + \int_0^t \mathscr{G}^k(s,\py^k_s)d\mathcal{W}_s$$ or equivalently, $$\py^k_t = \py^k_0 + \int_0^t \mathscr{A}(s,\py^k_s)ds + \sum_{i=1}^k\int_0^t \mathscr{G}_i(s,\py^k_s)dW^i_s$$ for $\py^k_0:=\py_0$. The existence and uniqueness of solutions to this finite-dimensional system is then classical (for solutions defined as in the theorem). Consider now solutions $\py^j, \py^k$ for $j < k$ arbitrary, which therefore satisfy the difference equation
\begin{align*}
    \py^k_r - \py^j_r = \int_0^r\mathscr{A}(s,\py^k_s) - \mathscr{A}(s,\py^j_s) ds + \int_0^r\mathscr{G}^k(s,\py^k_s) -\mathscr{G}^j(s,\py^j_s) d\mathcal{W}_s
\end{align*}
for any $r\in[0,\infty)$. By applying the energy identity Proposition \ref{a new first statement}, for the spaces all taken to be $\mathcal{H}$, we see further that the identity 
\begin{align*}
    \norm{\py^k_r - \py^j_r}_{\mathcal{H}}^2 &= 2\int_0^r\inner{\mathscr{A}(s,\py^k_s) - \mathscr{A}(s,\py^j_s)}{\py^k_s - \py^j_s}_{\mathcal{H}} ds\\ &+ \int_0^r\norm{\mathscr{G}^k(s,\py^k_s) -\mathscr{G}^j(s,\py^j_s)}^2_{\mathscr{L}^2(\mathfrak{U};\mathcal{H})}ds +  2\int_0^r\inner{\mathscr{G}^k(s,\py^k_s) -\mathscr{G}^j(s,\py^j_s)}{\py^k_s - \py^j_s}_{\mathcal{H}} d\mathcal{W}_s
\end{align*}
holds $\mathbbm{P}-a.s.$. We use Cauchy-Schwarz to move to an inequality, and rewrite the quadratic variation term to give us the bound
\begin{align*}
    \norm{\py^k_r - \py^j_r}_{\mathcal{H}}^2 &\leq  \int_0^r\left(2\norm{\mathscr{A}(s,\py^k_s) - \mathscr{A}(s,\py^j_s)}_{\mathcal{H}}\norm{\py^k_s - \py^j_s}_{\mathcal{H}} + \sum_{i=1}^j\norm{\mathscr{G}_i(s,\py^k_s) -\mathscr{G}_i(s,\py^j_s)}^2\right)ds\\& +\int_0^r \sum_{i=j+1}^k\norm{\mathscr{G}_i(s,\py^k_s)}_{\mathcal{H}}^2ds  +  2\int_0^r\inner{\mathscr{G}^k(s,\py^k_s) -\mathscr{G}^j(s,\py^j_s)}{\py^k_s - \py^j_s}_{\mathcal{H}} d\mathcal{W}_s.
\end{align*}
In one step now we bound the stochastic integral by its absolute value, take the supremum over all such $r$ up to any arbitrary time $t\in[0,\infty)$ and employ the Lipschitz assumption to see that
\begin{align*}
    \sup_{r\in[0,t]}\norm{\py^k_r - \py^j_r}_{\mathcal{H}}^2 &\leq  c\int_0^t\norm{\py^k_s - \py^j_s}_{\mathcal{H}}^2ds\\& +\int_0^t \sum_{i=j+1}^k\norm{\mathscr{G}_i(s,\py^k_s)}_{\mathcal{H}}^2ds  +  2\sup_{r\in[0,t]}\left\vert\int_0^r\inner{\mathscr{G}^k(s,\py^k_s) -\mathscr{G}^j(s,\py^j_s)}{\py^k_s - \py^j_s}_{\mathcal{H}} d\mathcal{W}_s\right\vert
\end{align*}
for a generic constant $c$, allowed to depend on $t$. We want to take the expectation here but have to be slightly careful in ensuring that the expectation is finite; to this end we consider the stopping times $$\tau_R:= R \wedge \inf\{s \geq 0: \max\{\norm{\py^k_s}_{\mathcal{H}},\norm{\py^j_s}_{\mathcal{H}}\} \geq R \} $$ and the process defined for any fixed $R$ by $$\tilde{\py}^k_s:=\py^k_s\mathbbm{1}_{s \leq \tau_R}, \qquad \tilde{\py}^j_s:=\py^j_s\mathbbm{1}_{s \leq \tau_R}.$$
From the continuity of the processes $\py^k,\py^j$ then $(\tau_R)$ is a $\mathbbm{P}-a.s.$ monotone increasing sequence convergent to infinity and $\max\{\norm{\tilde{\py}^k_s}_{\mathcal{H}},\norm{\tilde{\py}^j_s}_{\mathcal{H}}\} \leq R$ for any $s \geq 0$. It is trivial that these processes satisfy the same inequality
\begin{align*}
    \sup_{r\in[0,t]}\norm{\tilde{\py}^k_r - \tilde{\py}^j_r}_{\mathcal{H}}^2 &\leq  c\int_0^t\norm{\tilde{\py}^k_s - \tilde{\py}^j_s}_{\mathcal{H}}^2ds\\& +\int_0^t \sum_{i=j+1}^k\norm{\mathscr{G}_i(s,\tilde{\py}^k_s)}_{\mathcal{H}}^2ds  +  2\sup_{r\in[0,t]}\left\vert\int_0^r\inner{\mathscr{G}^k(s,\tilde{\py}^k_s) -\mathscr{G}^j(s,\tilde{\py}^j_s)}{\tilde{\py}^k_s - \tilde{\py}^j_s}_{\mathcal{H}} d\mathcal{W}_s\right\vert
\end{align*}
and justify that the expectation of all terms involved is finite (indeed for the stochastic integral, using the Lipschitz assumption and the boundedness of $\tilde{\py}^k_\cdot - \tilde{\py}^j_\cdot$ then $\inner{\mathscr{G}^k(\cdot,\tilde{\py}^k_\cdot) -\mathscr{G}^j(\cdot,\tilde{\py}^j_\cdot)}{\tilde{\py}^k_\cdot - \tilde{\py}^j_\cdot}_{\mathcal{H}} \in \mathcal{I}^\R(\mathcal{W})$ so the expectation of this term is finite). We do now take the expectation and apply the classical Burkholder-Davis-Gundy Inequality (recall, again, [\cite{burkholder1972integral}]) to give us the bound
\begin{align*}
    \mathbbm{E}\sup_{r\in[0,t]}\norm{\tilde{\py}^k_r - \tilde{\py}^j_r}_{\mathcal{H}}^2 &\leq  c\mathbbm{E}\int_0^t\norm{\tilde{\py}^k_s - \tilde{\py}^j_s}_{\mathcal{H}}^2ds +\mathbbm{E}\int_0^t \sum_{i=j+1}^k\norm{\mathscr{G}_i(s,\tilde{\py}^k_s)}_{\mathcal{H}}^2ds\\  &+  c\mathbbm{E}\left(\int_0^t\sum_{i=1}^\infty\inner{\mathscr{G}_i^k(s,\tilde{\py}^k_s) -\mathscr{G}_i^j(s,\tilde{\py}^j_s)}{\tilde{\py}^k_s - \tilde{\py}^j_s}_{\mathcal{H}}^2 ds\right)^{\frac{1}{2}}
\end{align*}
which we promptly reduce to 
\begin{align*}
    \mathbbm{E}\sup_{r\in[0,t]}\norm{\tilde{\py}^k_r - \tilde{\py}^j_r}_{\mathcal{H}}^2 &\leq  c\mathbbm{E}\int_0^t\norm{\tilde{\py}^k_s - \tilde{\py}^j_s}_{\mathcal{H}}^2ds +\mathbbm{E}\int_0^t \sum_{i=j+1}^k\norm{\mathscr{G}_i(s,\tilde{\py}^k_s)}_{\mathcal{H}}^2ds\\&  +  c\mathbbm{E}\left(\int_0^t\left[\sum_{i=1}^j\norm{\mathscr{G}_i(s,\tilde{\py}^k_s) -\mathscr{G}_i(s,\tilde{\py}^j_s)}^2 + \sum_{i=j+1}^k\norm{\mathscr{G}_i(s,\tilde{\py}^k_s)}_{\mathcal{H}}^2 \right]\norm{\tilde{\py}^k_s - \tilde{\py}^j_s}_{\mathcal{H}}^2 ds\right)^{\frac{1}{2}}.
\end{align*}
Employing our Lipschitz assumption once more, followed by an application of Young's Inequality, we have that
\begin{align*}
    &c\left(\int_0^t\left[\sum_{i=1}^j\norm{\mathscr{G}_i(s,\tilde{\py}^k_s) -\mathscr{G}_i(s,\tilde{\py}^j_s)}^2 + \sum_{i=j+1}^k\norm{\mathscr{G}_i(s,\tilde{\py}^k_s)}_{\mathcal{H}}^2 \right]\norm{\tilde{\py}^k_s - \tilde{\py}^j_s}_{\mathcal{H}}^2 ds\right)^{\frac{1}{2}}\\
    & \qquad \qquad \leq c\left(\int_0^t\left[\norm{\tilde{\py}^k_s - \tilde{\py}^j_s}_{\mathcal{H}}^2 + \sum_{i=j+1}^k\norm{\mathscr{G}_i(s,\tilde{\py}^k_s)}_{\mathcal{H}}^2 \right]\norm{\tilde{\py}^k_s - \tilde{\py}^j_s}_{\mathcal{H}}^2 ds\right)^{\frac{1}{2}}\\
    &\qquad \qquad \leq c\left(\sup_{r\in[0,t]}\norm{\tilde{\py}^k_r - \tilde{\py}^j_r}_{\mathcal{H}}^2\int_0^t \norm{\tilde{\py}^k_s - \tilde{\py}^j_s}_{\mathcal{H}}^2 + \sum_{i=j+1}^k\norm{\mathscr{G}_i(s,\tilde{\py}^k_s)}_{\mathcal{H}}^2  ds\right)^{\frac{1}{2}}\\
    &\qquad \qquad = c\left(\sup_{r\in[0,t]}\norm{\tilde{\py}^k_r - \tilde{\py}^j_r}_{\mathcal{H}}^2\right)^{\frac{1}{2}}\left(\int_0^t \norm{\tilde{\py}^k_s - \tilde{\py}^j_s}_{\mathcal{H}}^2 + \sum_{i=j+1}^k\norm{\mathscr{G}_i(s,\tilde{\py}^k_s)}_{\mathcal{H}}^2  ds\right)^{\frac{1}{2}}\\
    & \qquad \qquad \leq \frac{1}{2}\sup_{r\in[0,t]}\norm{\tilde{\py}^k_r - \tilde{\py}^j_r}_{\mathcal{H}}^2 + \frac{c^2}{2}\int_0^t\left(\norm{\tilde{\py}^k_s - \tilde{\py}^j_s}_{\mathcal{H}}^2 + \sum_{i=j+1}^k\norm{\mathscr{G}_i(s,\tilde{\py}^k_s)}_{\mathcal{H}}^2\right)ds
\end{align*}
and furthermore
\begin{align*}
    \mathbbm{E}\sup_{r\in[0,t]}\norm{\tilde{\py}^k_r - \tilde{\py}^j_r}_{\mathcal{H}}^2 &\leq  c\mathbbm{E}\int_0^t\norm{\tilde{\py}^k_s - \tilde{\py}^j_s}_{\mathcal{H}}^2ds +c\mathbbm{E}\int_0^t \sum_{i=j+1}^k\norm{\mathscr{G}_i(s,\tilde{\py}^k_s)}_{\mathcal{H}}^2ds.
\end{align*}
It is then a standard application of the Gr\"{o}nwall Inequality that \begin{equation}\label{abcdefg} \mathbbm{E}\sup_{r\in[0,t]}\norm{\tilde{\py}^k_r - \tilde{\py}^j_r}_{\mathcal{H}}^2 \leq c\mathbbm{E}\int_0^t \sum_{i=j+1}^k\norm{\mathscr{G}_i(s,\tilde{\py}^k_s)}_{\mathcal{H}}^2ds\end{equation} where the $c$ incorporates $e^{ct}$. Observe also that, through very similar arguments just using the linear growth property instead of the Lipschitz one, we have that
\begin{align*}
    \mathbbm{E}\sup_{r\in[0,t]}\norm{\tilde{\py}^k_r}_{\mathcal{H}}^2 &\leq \mathbbm{E}\norm{\tilde{\py}_0}_{\mathcal{H}}^2 + \mathbbm{E}\int_0^t\left(2\norm{\mathscr{A}(s,\tilde{\py}^k_s)}_{\mathcal{H}}\norm{\tilde{\py}^k_s}_{
\mathcal{H}} + \sum_{i=1}^k\norm{\mathscr{G}_i(s,\tilde{\py}^k_s)}_{\mathcal{H}}^2\right)ds\\ & \qquad \qquad \qquad \qquad +2\mathbbm{E}\sup_{r \in [0,t]}\left\vert \sum_{i=1}^k\int_0^r\inner{\mathscr{G}_i(s,\tilde{\py}^k_s)}{\tilde{\py}^k_s}_{\mathcal{H}} dW^i_s\right\vert\\
&\leq \mathbbm{E}\norm{\tilde{\py}_0}_{\mathcal{H}}^2 +c\mathbbm{E}\int_0^t(1 + \norm{\tilde{\py}^k_s}_{\mathcal{H}}) \norm{\tilde{\py}^k_s}_{\mathcal{H}} + 1 +  \norm{\tilde{\py}^k_s}_{\mathcal{H}}^2 ds\\
& \qquad \qquad \qquad \qquad + c\mathbbm{E}\left(\int_0^t\left(1 + \norm{\tilde{\py}^k_s}_{\mathcal{H}}^2  \right)\norm{\tilde{\py}^k_s}_{\mathcal{H}}^2ds\right)^{\frac{1}{2}}
\end{align*}
to which we use that $\norm{\tilde{\py}^k_s}_{\mathcal{H}} \leq 1 + \norm{\tilde{\py}^k_s}_{\mathcal{H}}^2$ to see that
\begin{align*}
    \mathbbm{E}\sup_{r\in[0,t]}\norm{\tilde{\py}^k_r}_{\mathcal{H}}^2 \leq \mathbbm{E}\norm{\tilde{\py}_0}_{\mathcal{H}}^2+ c\mathbbm{E}\int_0^t 1 + \norm{\tilde{\py}^k_s}_{\mathcal{H}}^2 ds + c\mathbbm{E}\left(\int_0^t\left(1 + \norm{\tilde{\py}^k_s}_{\mathcal{H}}^2  \right)\norm{\tilde{\py}^k_s}_{\mathcal{H}}^2ds\right)^{\frac{1}{2}}
\end{align*}
and further
\begin{align*}
    \mathbbm{E}\sup_{r\in[0,t]}\norm{\tilde{\py}^k_r}_{\mathcal{H}}^2 \leq c\left[\mathbbm{E}\norm{\tilde{\py}_0}_{\mathcal{H}}^2+ 1\right] + c\mathbbm{E}\int_0^t \norm{\tilde{\py}^k_s}_{\mathcal{H}}^2 ds
\end{align*}
as above, integrating the $1$ and adding it as a constant. Thus we have $$  \mathbbm{E}\sup_{r\in[0,t]}\norm{\tilde{\py}^k_r}_{\mathcal{H}}^2 \leq c\left[\mathbbm{E}\norm{\tilde{\py}_0}_{\mathcal{H}}^2+ 1\right] = c\left[\mathbbm{E}\norm{\py_0}_{\mathcal{H}}^2+ 1\right] $$ which is a bound uniform in $k$ and independent of $\tau_R$. Moreover for each fixed $k$ we appreciate that the sequence of random variables $$\sup_{r\in[0,t]}\norm{\tilde{\py}^k_r}_{\mathcal{H}}^2$$ is monotone increasing (indexed by $R$) and convergent to $\sup_{r\in[0,t]}\norm{\py^k_r}_{\mathcal{H}}^2$, $\mathbbm{P}-a.s.$. Thus we may apply the Monotone Convergence Theorem to this sequence of random variables to see that
\begin{align*}
    \mathbbm{E}\sup_{r\in[0,t]}\norm{\py^k_r}_{\mathcal{H}}^2 &= \lim_{R \rightarrow \infty}\mathbbm{E}\sup_{r\in[0,t]}\norm{\tilde{\py}^k_r}_{\mathcal{H}}^2\\ &\leq c\left[\mathbbm{E}\norm{\py_0}_{\mathcal{H}}^2+ 1\right].
\end{align*}
With this bound established we can revert back to (\ref{abcdefg}), combining with the boundedness of the $\mathscr{G}_i$ to deduce that $$\mathbbm{E}\int_0^t\sum_{i=j+1}^{k}\norm{\mathscr{G}_i(s,\py^k_s)}_{\mathcal{H}}^2ds < \infty$$ and clearly $$\mathbbm{E}\int_0^t\sum_{i=j+1}^{k}\norm{\mathscr{G}_i(s,\tilde{\py}^k_s)}_{\mathcal{H}}^2ds \leq \mathbbm{E}\int_0^t\sum_{i=j+1}^{k}\norm{\mathscr{G}_i(s,\py^k_s)}_{\mathcal{H}}^2ds$$
so we can update (\ref{abcdefg}) with the bound 
\begin{equation}\nonumber \mathbbm{E}\sup_{r\in[0,t]}\norm{\tilde{\py}^k_r - \tilde{\py}^j_r}_{\mathcal{H}}^2 \leq c\mathbbm{E}\int_0^t \sum_{i=j+1}^k\norm{\mathscr{G}_i(s,\py^k_s)}_{\mathcal{H}}^2ds\end{equation}
to which we apply the same monotone convergence argument to deduce that 
\begin{equation}\label{abcdefgh} \mathbbm{E}\sup_{r\in[0,t]}\norm{\py^k_r - \py^j_r}_{\mathcal{H}}^2 \leq c\mathbbm{E}\int_0^t \sum_{i=j+1}^k\norm{\mathscr{G}_i(s,\py^k_s)}_{\mathcal{H}}^2ds.\end{equation}
Moreover
\begin{align*}
    \sup_{k > j}\mathbbm{E}\int_0^t \sum_{i=j+1}^k\norm{\mathscr{G}_i(s,\py^k_s)}_{\mathcal{H}}^2ds &\leq t\sup_{k >j}\mathbbm{E}\sup_{r\in[0,t]} \sum_{i=j+1}^k\norm{\mathscr{G}_i(r,\py^k_r)}_{\mathcal{H}}^2\\
    &\leq t\sup_{k >j}\mathbbm{E}\sup_{r\in[0,t]} \sum_{i=j+1}^kC_tc_i(1 + \norm{\py^k_r}_{\mathcal{H}}^2)\\
    &:= c\sup_{k >j}\mathbbm{E}\sup_{r\in[0,t]} \sum_{i=j+1}^kc_i(1 + \norm{\py^k_r}_{\mathcal{H}}^2)\\
    &\leq c\sup_{k >j}\sum_{i=j+1}^kc_i\mathbbm{E}\sup_{r\in[0,t]} (1 + \norm{\py^k_r}_{\mathcal{H}}^2)\\
    &\leq c \sum_{i=j+1}^\infty c_i \sup_{k > j}\mathbbm{E}\sup_{r\in[0,t]} (1 + \norm{\py^k_r}_{\mathcal{H}}^2)\\
    &\leq c \sum_{i=j+1}^\infty c_i\left(c\left[\mathbbm{E}\norm{\py_0}_{\mathcal{H}}^2+ 1\right] \right)
\end{align*}
which is a sequence in $j$ monotone decreasing to zero. Thus in view of (\ref{abcdefgh}),
$$\lim_{j \rightarrow \infty} \sup_{k > j}\left[\mathbbm{E}\sup_{r\in[0,t]}\norm{\py^k_r - \py^j_r}_{\mathcal{H}}^2\right] = 0$$ so the sequence $(\py^k)$ is Cauchy in $L^2\left(\Omega;C\left([0,t];\mathcal{H}\right) \right)$ and as such we can deduce the existence of a $\py$ such that $\py^k \rightarrow \py$ in this space (and hence, in $L^2\left(\Omega;L^2\left([0,t];\mathcal{H}\right) \right))$ for every $t \in [0,T]$, and thus $\py$ is also the $\mathbbm{P}-a.s.$ limit of a subsequence of the $(\py^k)$ in $C\left([0,T];\mathcal{H}\right)$. This limit process inherits the progressive measurability (indeed it is adapted and has continuous paths in $\mathcal{H}$). It simply remains to show that $\py$ satisfies the identity (\ref{identityingalerkinsolution}), so we first consider the $\mathbbm{P}-a.s.$ convergent subsequence $(\py^{k_l})$. Looking at the stochastic integral and employing Proposition \ref{ito isom for cylindrical}, we have that
\begin{align*}
   &\mathbbm{E}\left \Vert \int_0^t \mathscr{G}(s,\py_s)d\mathcal{W}_s - \int_0^t \mathscr{G}^{k_l}(s,\py^{k_l}_s)d\mathcal{W}_s \right\Vert_{\mathcal{H}}^2\\ & \qquad \qquad \qquad \qquad =  \mathbbm{E}\left \Vert  \int_0^t \mathscr{G}(s,\py_s) - \mathscr{G}^{k_l}(s,\py^{k_l}_s) d\mathcal{W}_s \right\Vert_{\mathcal{H}}^2\\
   & \qquad \qquad \qquad \qquad= \mathbbm{E}\int_0^t\sum_{i=1}^\infty \norm{\mathscr{G}_i(s,\py_s) - \mathscr{G}^{k_l}_i(s,\py^{k_l}_s)}_{\mathcal{H}}^2ds\\
& \qquad \qquad \qquad \qquad= \mathbbm{E}\int_0^t\left(\sum_{i=1}^{k_l}\norm{\mathscr{G}_i(s,\py_s) - \mathscr{G}_i(s,\py^{k_l}_s)}_{\mathcal{H}}^2ds + \sum_{i=k_l + 1}^{\infty}\norm{\mathscr{G}_i(s,\py_s)}_{\mathcal{H}}^2\right)ds\\
   & \qquad \qquad \qquad \qquad\leq \sum_{i=1}^{k_l} \mathbbm{E}\int_0^t c_i \norm{\py_s - \py^{k_l}_s}_{\mathcal{H}}^2ds + \sum_{i=k_l+1}^\infty \mathbbm{E}\int_0^t c_i \norm{\py_s}_{\mathcal{H}}^2ds\\
   & \qquad \qquad \qquad \qquad\leq c\mathbbm{E}\int_0^t \norm{\py_s - \py^{k_l}_s}_{\mathcal{H}}^2ds + \sum_{i=k_l+1}^\infty \mathbbm{E}\int_0^t c_i \norm{\py_s}_{\mathcal{H}}^2ds.
\end{align*}
so from the known $L^2\left(\Omega;L^2\left([0,t];\mathcal{H}\right) \right)$ convergence, we have that $$\lim_{k_l \rightarrow \infty} \int_0^t \mathscr{G}^{k_l}(s,\py^{k_l}_s)d\mathcal{W}_s = \int_0^t \mathscr{G}(s,\py_s)d\mathcal{W}_s $$ with the limit in $L^2\left(\Omega;\mathcal{H} \right)$. We can thus extract a further subsequence which we denote $(\py^{k_m})$ such that this limit holds $\mathbbm{P}-a.s.$ in $\mathcal{H}$, and is of course still such that $(\py^{k_m}) \rightarrow \py$ $\mathbbm{P}-a.s.$ in $L^2\left([0,t];\mathcal{H}\right)$. Therefore
\begin{align*}
   \left\Vert \int_0^t \mathscr{A}(s,\py_s)ds- \int_0^t \mathscr{A}(s,\py^{k_m}_s)ds\right\Vert_{\mathcal{H}}^2 &\leq t\int_0^t\norm{ \mathscr{A}(s,\py_s)-  \mathscr{A}(s,\py^{k_m}_s)}_{\mathcal{H}}^2ds\\ &\leq ct\int_0^t\norm{\py_s - \py^{k_m}_s}_{\mathcal{H}}^2ds
\end{align*}
and so $$\lim_{k_m \rightarrow \infty}\int_0^t \mathscr{A}(s,\py^{k_m}_s)ds = \int_0^t \mathscr{A}(s,\py_s)ds$$
with the limit $\mathbbm{P}-a.s.$ in $\mathcal{H}$. Thus by taking the $\mathbbm{P}-a.s.$ limit in $\mathcal{H}$ of the identity satisfied by $\py^{k_m}$, we reach (\ref{identityingalerkinsolution}) as required.

\end{proof}

\begin{theorem} \label{uniquey theorem}
Suppose $\sy$ is another strong solution of (\ref{identityingalerkinsolution}). Then for every $T \geq 0$, $$\mathbbm{P}\left(\left\{\omega \in \Omega: \py_t(\omega) = \sy_t(\omega) \quad \forall t \in [0,T] \right\}\right)=1.$$
\end{theorem}

\begin{proof}
The method of proof here is entirely contained in that for the existence just seen. Indeed we look at the energy equality satisfied by the difference of the solutions, which is 
\begin{align*}
    \norm{\py_r - \sy_r}_{\mathcal{H}}^2 &= 2\int_0^r\inner{\mathscr{A}(s,\py_s) - \mathscr{A}(s,\sy_s)}{\py_s - \sy_s}_{\mathcal{H}} ds\\ &+ \int_0^r\norm{\mathscr{G}(s,\py_s) -\mathscr{G}(s,\sy_s)}^2_{\mathscr{L}^2(\mathfrak{U};\mathcal{H})}ds +  2\int_0^r\inner{\mathscr{G}(s,\py_s) -\mathscr{G}(s,\sy_s)}{\py_s - \sy_s}_{\mathcal{H}} d\mathcal{W}_s.
\end{align*}
Following along the proof, we introduce
$$\tau_R:= R \wedge \inf\{s \geq 0: \max\{\norm{\py_s}_{\mathcal{H}},\norm{\sy_s}_{\mathcal{H}}\} \geq R \} $$ and the process defined for any fixed $R$ by $$\tilde{\py}_s:=\py_s\mathbbm{1}_{s \leq \tau_R}, \qquad \tilde{\sy}_s:=\sy_s\mathbbm{1}_{s \leq \tau_R}.$$
In this case we have the inequality 
\begin{align*}
    \sup_{r\in[0,t]}\norm{\tilde{\py}_r - \tilde{\sy}_r}_{\mathcal{H}}^2 \leq  c\int_0^t\norm{\tilde{\py}_s - \tilde{\sy}_s}_{\mathcal{H}}^2ds  +  2\sup_{r\in[0,t]}\left\vert\int_0^r\inner{\mathscr{G}(s,\tilde{\py}_s) -\mathscr{G}(s,\tilde{\sy}_s)}{\tilde{\py}_s - \tilde{\sy}_s}_{\mathcal{H}} d\mathcal{W}_s\right\vert
\end{align*}
so following all of the same steps, simply now without the $\sum_{i=j+1}^k\norm{\mathscr{G}_i(s,\tilde{\py}^k_s)}^2_{\mathcal{H}}$ term, we deduce again that
$$\mathbbm{E}\sup_{r\in[0,t]}\norm{\tilde{\py}_r - \tilde{\sy}_r}_{\mathcal{H}}^2 \leq 0$$ in analogy with (\ref{abcdefg}). By the same monotone convergence argument, we have that $$\mathbbm{E}\sup_{r\in[0,t]}\norm{\py_r - \sy_r}_{\mathcal{H}}^2 = 0$$ which gives the result.

\end{proof}

\subsection{Stochastic Gr\"{o}nwall Lemma}

Continuing to look at techniques from PDE theory, a Stochastic Gr\"{o}nwall Lemma will prove of great significance in applications. Whilst in some situations we can apply the classical Gr\"{o}nwall Lemma to the expectation of the process, this is complicated when we have control by the expectation of a product of processes. To combat this Glatt-Holtz and Ziane proved the following, see [\cite{glatt2009strong}] Lemma 5.3.

\begin{lemma} \label{gronny}
Fix $t>0$ and suppose that $\boldsymbol{\phi},\boldsymbol{\psi}, \boldsymbol{\eta}$ are real-valued, non-negative stochastic processes. Assume, moreover, that there exists constants $c',\hat{c}$ (allowed to depend on $t$) such that for $\mathbbm{P}-a.e.$ $\omega$, \begin{equation} \label{boundingronny} \int_0^t\boldsymbol{\eta}_s(\omega) ds \leq c'\end{equation} and for all stopping times $0 \leq \theta_j < \theta_k \leq t$,
$$\mathbbm{E}\sup_{r \in [\theta_j,\theta_k]}\boldsymbol{\phi}_r + \mathbbm{E}\int_{\theta_j}^{\theta_k}\boldsymbol{\psi}_sds \leq \hat{c}\mathbbm{E}\left(\left(\boldsymbol{\phi}_{\theta_j} + 1 \right) + \int_{\theta_j}^{\theta_k} \boldsymbol{\eta}_s\boldsymbol{\phi}_sds\right) < \infty. $$Then there exists a constant $C$ dependent only on $c',\hat{c},t$ such that $$\mathbbm{E}\sup_{r \in [0,t]}\boldsymbol{\phi}_r + \mathbbm{E}\int_{0}^{t}\boldsymbol{\psi}_sds \leq C\left[\mathbbm{E}(\boldsymbol{\phi}_{0}) + 1\right].$$
\end{lemma}

\begin{proof}
 We shall make explicit reference to this constant $\hat{c}$, in defining a sequence of stopping times $$\theta_j(\omega):= t \wedge \inf\left\{r \geq 0: \int_0^r \boldsymbol{\eta}_s(\omega) ds \leq \frac{j}{2\hat{c}}\right\}$$ for $j = 0, 1, \dots$. Clearly $\theta_0 = 0$ ($\mathbbm{P}-a.s.$). From the boundedness (\ref{boundingronny}) uniformly over $\mathbbm{P}-a.e.$ $\omega$, there exists some finite $N$ such that $\theta_N=t$ ($\mathbbm{P}-a.s.$). Moreover for $0 \leq j < j+1 \leq N$, from the time continuity of the integral and characterisation of the first hitting times we have that $$ \int_{\theta_{j}}^{\theta_{j+1}} \boldsymbol{\eta}_s ds \leq \frac{1}{2\hat{c}}$$ $\mathbbm{P}-a.s.$ (and is in fact an equality for $j+1 < N$). From the assumed inequality we see that for any such $j$,
\begin{align*}
     \mathbbm{E}\sup_{r\in [{\theta_j},\theta_{j+1}]}\boldsymbol{\phi}_r +\mathbbm{E}\int_{\theta_j}^{\theta_{j+1}}\boldsymbol{\psi}_sds  &\leq \hat{c}\mathbbm{E}\left(\left(\boldsymbol{\phi}_{\theta_j} + 1\right)+ \int_{\theta_j}^{\theta_{j+1}}\boldsymbol{\eta}_s\boldsymbol{\phi}_s ds\right)\\
    &\leq \hat{c}\mathbbm{E}\left(\left(\boldsymbol{\phi}_{\theta_j} + 1\right)+ \frac{1}{2\hat{c}}\sup_{r\in[\theta_j,\theta_{j+1}]}\boldsymbol{\phi}_r\right)
\end{align*}
and therefore
\begin{equation}\label{abc}
\mathbbm{E}\sup_{r\in [{\theta_j},\theta_{j+1}]}\boldsymbol{\phi}_r +\mathbbm{E}\int_{\theta_j}^{\theta_{j+1}}\boldsymbol{\psi}_sds \leq 2\hat{c}\mathbbm{E}\left(\boldsymbol{\phi}_{\theta_j} + 1\right).
\end{equation}
For $j=0$ then $$ \mathbbm{E}\sup_{r\in [0,\theta_{1}]}\boldsymbol{\phi}_r +\mathbbm{E}\int_{0}^{\theta_{1}}\boldsymbol{\psi}_sds \leq 2\hat{c}\mathbbm{E}\left(\boldsymbol{\phi}_{0} + 1\right)$$
and we use this along with (\ref{abc}) to make an inductive argument. Suppose that for some such $j$,
\begin{equation} \label{bcd}
   \mathbbm{E}\sup_{r\in [0,\theta_{j}]}\boldsymbol{\phi}_r +\mathbbm{E}\int_{0}^{\theta_{j}}\boldsymbol{\psi}_sds \leq c\mathbbm{E}\left(\boldsymbol{\phi}_{0} + 1\right)
\end{equation}
for a general constant $c$ as seen throughout this proof, dependent on $c',\hat{c},t$. Then
\begin{align*}
    &\mathbbm{E}\sup_{r\in [0,\theta_{j+1}]}\boldsymbol{\phi}_r +\mathbbm{E}\int_{0}^{\theta_{j+1}}\boldsymbol{\psi}_sds\\ & \qquad \qquad \leq \left(\mathbbm{E}\sup_{r\in [0,\theta_{j}]}\boldsymbol{\phi}_r +\mathbbm{E}\int_{0}^{\theta_{j}}\boldsymbol{\psi}_sds\right) + \left(\mathbbm{E}\sup_{r\in [\theta_{j},\theta_{j+1}]}\boldsymbol{\phi}_r +\mathbbm{E}\int_{\theta_j}^{\theta_{j+1}}\boldsymbol{\psi}_sds \right)\\
    & \qquad \qquad \leq c\mathbbm{E}\left(\boldsymbol{\phi}_{0} + 1\right) + 2\hat{c}\mathbbm{E}\left(\boldsymbol{\phi}_{\theta_j} + 1\right)\\
    &\qquad \qquad \leq c\mathbbm{E}\left(\boldsymbol{\phi}_{0} + 1\right) + 2\hat{c}\mathbbm{E}\sup_{r\in[0,\theta_j]}\left(\boldsymbol{\phi}_{r} + 1\right)\\
    &\qquad \qquad \leq
     c\mathbbm{E}\left(\boldsymbol{\phi}_{0} + 1\right) + 2\hat{c}c\mathbbm{E}\left(\boldsymbol{\phi}_{0} + 1\right)\\
     &\qquad \qquad = c\mathbbm{E}\left(\boldsymbol{\phi}_{0} + 1\right)
\end{align*}
thanks to (\ref{abc}) and two applications of (\ref{bcd}). Hence, by induction, we can conclude that (\ref{bcd}) holds for all $j=0, \dots, N$ and in particular for $\theta_N$ which is $\mathbbm{P}-a.s.$ equal to $t$. This completes the proof.
\end{proof}

\subsection{Tightness Criteria} \label{subby tight}

Unsurprisingly our route into the relative compactness methods of PDE theory is through tightness, owing to Prokhorov's Theorem. Poignantly we can connect this weak limit of measures with a genuine process through Skorohod's Representation Theorem, see for example [\cite{billingsley2013convergence}] pp.70. Simple criteria through which we can establish tightness in the space of solutions to SPDEs will prove useful. Recalling our notions of solution, for example Definitions \ref{definitionofregularsolutionglobal} and \ref{analytically weak sol def}, we consider tightness criteria in both the spaces $L^2\left([0,T];\mathcal{H}\right)$ and $\mathcal{D}\left([0,T];\mathcal{H} \right)$ for suitably chosen $\mathcal{H}$ (recall the definition of $\mathcal{D}\left([0,T];\mathcal{H} \right)$ from \ref{definition of spaces}). We note, for example [\cite{billingsley2013convergence}] pp.124, that the Skorohod Topology is equivalent to the uniform topology when restricted to continuous functions. It is necessary to use due to the separability of the associated metric space, allowing us to invoke Prokhorov's Theorem. Our first result is due to R\"{o}ckner, Shang and Zhang.

\begin{lemma} \label{Lemma 5.2}
    Let $\mathcal{H}_1, \mathcal{H}_2$ be Hilbert Spaces such that $\mathcal{H}_1$ is compactly embedded into $\mathcal{H}_2$, and for some fixed $T>0$ let $(\sy^n): \Omega \times [0,T] \rightarrow \mathcal{H}_1$ be a sequence of measurable processes such that \begin{equation} \label{first condition} \sup_{n\in \N}\mathbbm{E}\int_0^T\norm{\sy^n_s}^2_{\mathcal{H}_1}ds < \infty\end{equation} and for any $\varepsilon > 0$, 
    \begin{equation}\label{second condition} \lim_{\delta \rightarrow 0^+}\sup_{n \in \N}\mathbbm{P}\left(\left\{\omega \in \Omega:\int_0^{T-\delta}\norm{\sy^n_{s + \delta}(\omega) - \sy^n_s(\omega)}^2_{\mathcal{H}_2}ds > \varepsilon\right\} \right) =0.\end{equation}
    Then the sequence of the laws of $(\sy^n)$ is tight in the space of probability measures over $L^2\left([0,T];\mathcal{H}_2\right)$.
\end{lemma}

\begin{proof}
    See [\cite{rockner2022well}] Lemma 5.2.
\end{proof}

We now give two results for tightness in $\mathcal{D}\left([0,T];\mathcal{H} \right)$. These ideas were also present in [\cite{rockner2022well}] but were not established into a result, so we give a full proof here.

\begin{lemma} \label{lemma for D tight}
    Let $\mathcal{Y}$ be a reflexive separable Banach Space and $\mathcal{H}$ a separable Hilbert Space such that $\mathcal{Y}$ is compactly embedded into $\mathcal{\mathcal{H}}$, and consider the induced Gelfand Triple
    $$\mathcal{Y} \xhookrightarrow{} \mathcal{H} \xhookrightarrow{} \mathcal{Y}^*. $$ For some fixed $T>0$ let $\sy^n: \Omega \rightarrow C\left([0,T];\mathcal{H}\right)$ be a sequence of measurable processes such that for every $t\in[0,T]$, \begin{equation} \label{first condition primed}
        \sup_{n \in \N}\mathbbm{E}\left(\sup_{t\in[0,T]}\norm{\sy^n_t}_{\mathcal{H}}\right) < \infty
    \end{equation}
    and for any sequence of stopping times $(\gamma_n)$ with $\gamma_n: \Omega \rightarrow [0,T]$, and any $\varepsilon > 0$, $y \in \mathcal{Y}$,
    \begin{equation} \label{second condition primed}
        \lim_{\delta \rightarrow 0^+}\sup_{n \in \N}\mathbbm{P}\left(\left\{
    \omega \in \Omega: \left\vert \left\langle \sy^n_{(\gamma_n + \delta) \wedge T} -\sy^n_{\gamma_n }   , y     \right\rangle_{\mathcal{H}} \right\vert > \varepsilon \right\}\right)  = 0.
    \end{equation}
    Then the sequence of the laws of $(\sy^n)$ is tight in the space of probability measures over $\mathcal{D}\left([0,T];\mathcal{Y}^*\right)$.
\end{lemma}

\begin{proof}
    We essentially combine the tightness criteria of [\cite{jakubowski1986skorokhod}] Theorem 3.1 and [\cite{aldous1978stopping}] Theorem 1, in the specific case outlined here. Firstly in reference to [\cite{jakubowski1986skorokhod}] Theorem 3.1 we may take $E$ to be $\mathcal{Y}^*$ (which is separable from the reflexivity and separability of $\mathcal{Y})$ and $\mathbbm{F}$ to be $\left(\mathcal{Y}^*\right)^*$, which is well known to separate points in $\mathcal{Y}^*$ from a corollary of the Hahn-Banach Theorem which asserts that for every $\phi \in \mathcal{Y}^*$ there exists a $\psi \in \left(\mathcal{Y}^*\right)^*$ such that $\inner{\phi}{\psi}_{\mathcal{Y}^* \times \left(\mathcal{Y}^*\right)^*} = \norm{\phi}_{\mathcal{Y}^*}$. We also note that condition $(3.3)$ in [\cite{jakubowski1986skorokhod}] is satisfied for $(\mu_n)$ taken to be the sequence of laws of $(\sy^n)$ over $\mathcal{D}\left([0,T];\mathcal{Y}^*\right)$, owing to the property (\ref{first condition primed}). Indeed as $\mathcal{Y}$ is compactly embedded into $\mathcal{H}$ then $\mathcal{H}$ is compactly embedded into $\mathcal{Y}^*$, so one only needs to take a bounded subset of $\mathcal{H}$ for this property (3.3). Considering the closed ball of radius $M$ in $\mathcal{H}$, $\tilde{B}_M$, we have that 
\begin{align*}
    \mathbbm{P}\left(\left\{\omega \in \Omega: \sy^n(\omega) \notin D\left([0,T] ;\tilde{B}_M\right) \right\} \right) &\leq \mathbbm{P}\left(\left\{\omega \in \Omega: \sy^n(\omega) \notin C\left([0,T] ;\tilde{B}_M\right) \right\} \right)\\ &\leq 
    \mathbbm{P}\left(\left\{\omega \in \Omega: \sup_{t\in[0,T]}\norm{\sy^n_t(\omega)}_{\mathcal{H}} > M \right\} \right)\\
    &\leq \frac{1}{M}\mathbbm{E}(\sup_{t\in[0,T]}\norm{\sy^n_t}_{\mathcal{H}})\\
    &\leq  \frac{1}{M}\sup_{n \in \N}\mathbbm{E}(\sup_{t\in[0,T]}\norm{\sy^n_t}_{\mathcal{H}})
\end{align*}
from which we see an arbitrarily large choice of $M$ will justify (3.3). Therefore by Theorem 3.1 it only remains to show that for every $\psi \in \left(\mathcal{Y}^*\right)^*$ the sequence of the laws of $\inner{\sy^n}{\psi}_{\mathcal{Y}^* \times \left(\mathcal{Y}^*\right)^*}$ is tight in the space of probability measures over $\mathcal{D}\left([0,T];\R\right)$. By the reflexivity of $\mathcal{Y}$ for every $\psi \in \left(\mathcal{Y}^*\right)^*$ there exists a $y \in \mathcal{Y}$ such that $\inner{\sy^n}{\psi}_{\mathcal{Y}^* \times \left(\mathcal{Y}^*\right)^*} = \inner{\sy^n}{y}_{\mathcal{Y}^* \times \mathcal{Y}}$ and as $\sy^n_t \in \mathcal{H}$ $\mathbbm{P}-a.s.$, then this is furthermore just $\inner{\sy^n}{y}_{\mathcal{H}}$. The problem is now reduced to showing tightness in $\mathcal{D}\left([0,T];\R\right)$, which by Theorem 1 of [\cite{aldous1978stopping}] is satisfied if we can show that for for any sequence of stopping times $(\gamma_n)$, $\gamma_n: \Omega \rightarrow [0,T]$, and constants $(\delta_n)$, $\delta_n \geq 0$ and $\delta_n \rightarrow 0$ as $n \rightarrow \infty$:
    \begin{enumerate}
        \item For every $t \in [0,T]$, the sequence of the laws of $\inner{\sy^n_t}{y}_{\mathcal{H}}$ is tight in the space of probability measures over $\R$, \label{item 1}

        \item For every $\varepsilon > 0$, $\lim_{n \rightarrow \infty}\mathbbm{P}\left( \left\{ \omega \in \Omega: \left\vert \left\langle \sy^n_{(\gamma_n + \delta_n) \wedge T} -\sy^n_{\gamma_n }   , y     \right\rangle_{\mathcal{H}} \right\vert > \varepsilon \right\} \right) = 0.$ \label{item 2}
    \end{enumerate}
We address each item in turn: as for \ref{item 1}, we are required to show that for every $\varepsilon > 0$ and $t\in[0,T]$, there exists a compact $K_{\varepsilon} \subset \R$ such that for every $n \in \N$, $$\mathbbm{P}\left(\left\{\omega \in \Omega: \inner{\sy^n_t(\omega)}{y}_{\mathcal{H}} \notin K_{\varepsilon} \right\} \right) < \varepsilon.$$ To this end define $B_M$ as the closed ball of radius $M$ in $\R$, then
\begin{align*}
    \mathbbm{P}\left(\left\{\omega \in \Omega: \inner{\sy^n_t(\omega)}{y}_{\mathcal{H}} \notin B_M \right\} \right) &= \mathbbm{P}\left(\left\{\omega \in \Omega: \left\vert\inner{\sy^n_t(\omega)}{y}_{\mathcal{H}}\right\vert > M \right\} \right)\\
    &\leq \frac{1}{M}\mathbbm{E}(\left\vert\inner{\sy^n_t}{y}_{\mathcal{H}}\right\vert)\\
    &\leq \frac{\norm{y}_{\mathcal{H}}}{M}\sup_{n \in \N}\mathbbm{E}\left(\norm{\sy^n_t}_{\mathcal{H}}\right)
\end{align*}
so setting $$M:= \frac{\varepsilon}{2\norm{y}_{\mathcal{H}}\sup_{n \in \N}\mathbbm{E}\left(\norm{\sy^n_t}_{\mathcal{H}}\right)} $$ justifies item \ref{item 1}. As for \ref{item 2}, note that for each fixed $j \in \N$ we have that $$\left\vert \left\langle \sy^j_{(\gamma_j + \delta_j) \wedge T} -\sy^j_{\gamma_j }   , y     \right\rangle_{\mathcal{H}} \right\vert  \leq \sup_{n \in \N}\left\vert \left\langle \sy^n_{(\gamma_n + \delta_j) \wedge T} -\sy^n_{\gamma_n }   , y     \right\rangle_{\mathcal{H}} \right\vert $$ so in particular $$\lim_{j \rightarrow \infty}\mathbbm{P}\left(\left\{\left\vert \left\langle \sy^j_{(\gamma_j + \delta_j) \wedge T} -\sy^j_{\gamma_j }   , y     \right\rangle_{\mathcal{H}} \right\vert > \varepsilon\right\}\right)  \leq \lim_{j \rightarrow \infty}\sup_{n \in \N}\mathbbm{P}\left(\left\{\left\vert \left\langle \sy^n_{(\gamma_n + \delta_j) \wedge T} -\sy^n_{\gamma_n }   , y     \right\rangle_{\mathcal{H}} \right\vert > \varepsilon \right\}\right).$$ As $(\delta_j)$ was an arbitrary sequence of non-negative constants approaching zero, we can generically take $\delta \rightarrow 0^+$ and \ref{item 2} is implied by (\ref{second condition primed}). The proof is complete.

\end{proof}

We do not need to rely on this duality structure to obtain such a criteria. 

\begin{lemma} \label{lemma for D tight 2}
    Let $\mathcal{H}_1, \mathcal{H}_2$ be separable Hilbert Spaces with $\mathcal{H}_1$ compactly embedded into $\mathcal{H}_2$, and $V$ any dense set in $\mathcal{H}_2$. For some fixed $T>0$ let $\sy^n: \Omega \rightarrow C\left([0,T];\mathcal{H}_1\right)$ be a sequence of measurable processes such that \begin{equation} \label{first condition primed 2}
        \sup_{n \in \N}\mathbbm{E}\left(\sup_{t\in[0,T]}\norm{\sy^n_t}_{\mathcal{H}_1}\right) < \infty
    \end{equation}
    and for any sequence of stopping times $(\gamma_n)$ with $\gamma_n: \Omega \rightarrow [0,T]$, and any $\varepsilon > 0$, $v \in V$,
    \begin{equation} \label{second condition primed 2}
        \lim_{\delta \rightarrow 0^+}\sup_{n \in \N}\mathbbm{P}\left(\left\{
    \omega \in \Omega: \left\vert \left\langle \sy^n_{(\gamma_n + \delta) \wedge T} -\sy^n_{\gamma_n }   , v     \right\rangle_{\mathcal{H}_2} \right\vert > \varepsilon \right\}\right)  = 0.
    \end{equation}
    Then the sequence of the laws of $(\sy^n)$ is tight in the space of probability measures over $\mathcal{D}\left([0,T];\mathcal{H}_2\right)$.
\end{lemma}

The proof is mechanically near identical to that of Lemma \ref{lemma for D tight}, so we highlight only the slight technical differences.

\begin{proof}
   In reference to [\cite{jakubowski1986skorokhod}] Theorem 3.1 we may take $E$ to be $\mathcal{H}_2$ and $\mathbbm{F}$ to be the collection of functions defined for each $v\in V$ by $\inner{\cdot}{v}_{\mathcal{H}_2}$, which separates points in $\mathcal{H}_2$ from the density of $V$. We also note that condition $(3.3)$ in [\cite{jakubowski1986skorokhod}] is satisfied for $(\mu_n)$ taken to be the sequence of laws of $(\sy^n)$ over $\mathcal{D}\left([0,T];\mathcal{H}_2\right)$, owing to the property (\ref{first condition primed 2}). Indeed as $\mathcal{H}_1$ is compactly embedded into $\mathcal{H}_2$ one only needs to take a bounded subset of $\mathcal{H}_1$, hence considering the closed ball of radius $M$ in $\mathcal{H}_1$, $\tilde{B}_M$, we have that 
\begin{align*}
    \mathbbm{P}\left(\left\{\omega \in \Omega: \sy^n(\omega) \notin D\left([0,T] ;\tilde{B}_M\right) \right\} \right)
    \leq  \frac{1}{M}\sup_{n \in \N}\mathbbm{E}(\sup_{t\in[0,T]}\norm{\sy^n_t}_{\mathcal{H}_1})
\end{align*}
precisely as in Lemma \ref{lemma for D tight}, from which we see an arbitrarily large choice of $M$ will justify (3.3). Therefore by Theorem 3.1 it only remains to show that for every $v \in V$ the sequence of the laws of $\inner{\sy^n}{v}_{\mathcal{H}_2}$ is tight in the space of probability measures over $\mathcal{D}\left([0,T];\R\right)$. The remainder of the proof now follows exactly as in Lemma \ref{lemma for D tight}.
 
\end{proof}

\subsection{Cauchy Criteria}

A more direct way to deduce the existence of a limiting process from the Galerkin Approximations comes from the Cauchy Property, but in practice due to potential nonlinearities one requires some truncation to get sufficient control on the approximating sequence. More precisely for the $n^{\textnormal{th}}$ term of the sequence one must work up to a first hitting time of this process, giving a stopping time $\tau_n$. The question is then whether we can deduce a limiting process up to some time $\tau$, where $0 < \tau \leq \tau_n$ for all $n$. Such a result was proven by Glatt-Holtz and Ziane in [\cite{glatt2009strong}] Lemma 5.1.\\

Our result is a slight extension of this, but has important applications in the deduction of \textit{maximal} and \textit{global} solutions. The result of Glatt-Holtz and Ziane asserts that, under assumptions of a Cauchy property of the sequence of processes up until their first hitting times and some weak equicontinuity at the \textit{initial} time, then a limiting process and positive stopping time exist (which are then argued to be a local strong solution, as a limit of the Galerkin Approximation). No characterisation of this stopping time is given though, hence completely separate arguments are required to consider what interval the solution exists upon. We demonstrate that if instead one imposes that the processes satisfy a weak equicontinuity assumption at \textit{all} times then the limiting stopping time can be taken as a first hitting time of the limiting process for an arbitrarily large hitting parameter. Application of this result \textit{immediately} yields that solutions exist up until they blow up, removing the need for further analysis towards the interval on which solutions exist.

\begin{lemma} \label{amazing cauchy lemma}
    Fix $T>0$. For $t\in[0,T]$ let $X_t$ denote a Banach Space with norm $\norm{\cdot}_{X,t}$ such that for all $s > t$, $X_s \xhookrightarrow{}X_t$ and $\norm{\cdot}_{X,t} \leq \norm{\cdot}_{X,s}$. Suppose that $(\sy^n)$ is a sequence of processes $\sy^n:\Omega \mapsto X_T$, $\norm{\sy^n}_{X,\cdot}$ is adapted and $\mathbb{P}-a.s.$ continuous, $\sy^n \in L^2\left(\Omega;X_T\right)$, and such that $\sup_{n}\norm{\sy^n}_{X,0} \in L^\infty\left(\Omega;\R\right)$. For any given $M >1$ define the stopping times
    \begin{equation} \label{another taumt}
        \tau^{M,T}_n := T \wedge \inf\left\{s \geq 0: \norm{\sy^n}_{X,s}^2 \geq M + \norm{\sy^n}_{X,0}^2 \right\}.
    \end{equation}
Furthermore suppose \begin{equation} \label{supposition 1}
    \lim_{m \rightarrow \infty}\sup_{n \geq m}\mathbbm{E}\left[\norm{\sy^n-\sy^m}^2_{X,\tau
_{m}^{M,t}\wedge \tau_{n}^{M,t}} \right] =0
\end{equation}
and that for any stopping time $\gamma$ and sequence of stopping times $(\delta_j)$ which converge to $0$ $\mathbb{P}-a.s.$, \begin{equation} \label{supposition 2} \lim_{j \rightarrow \infty}\sup_{n\in\N}\mathbbm{E}\left(\norm{\sy^n}_{X,(\gamma + \delta_j) \wedge \tau^{M,T}_n}^2 - \norm{\sy^n}_{X,\gamma \wedge \tau^{M,T}_n}^2 \right) =0.
\end{equation}
Then there exists a stopping time $\tau^{M,T}_{\infty}$, a process $\sy:\Omega \mapsto X_{\tau^{M,T}_{\infty}}$ whereby $\norm{\sy}_{X,\cdot \wedge \tau^{M,T}_{\infty}}$ is adapted and $\mathbb{P}-a.s.$ continuous, and a subsequence indexed by $(m_j)$ such that 
\begin{itemize}
    \item $\tau^{M,T}_{\infty} \leq \tau^{M,T}_{m_j}$ $\mathbb{P}-a.s.$,
    \item $\lim_{j \rightarrow \infty}\norm{\sy - \sy^{m_j}}_{X,\tau^{M,T}_{\infty}} = 0$ $\mathbb{P}-a.s.$.
\end{itemize}
Moreover for any $R>0$ we can choose $M$ to be such that the stopping time \begin{equation} \label{another tauR}
        \tau^{R,T} := T \wedge \inf\left\{s \geq 0: \norm{\sy}_{X,s\wedge\tau^{M,T}_{\infty}}^2 \geq R \right\}
    \end{equation}
satisfies $\tau^{R,T} \leq \tau^{M,T}_{\infty}$ $\mathbb{P}-a.s.$. Thus $\tau^{R,T}$ is simply $T \wedge \inf\left\{s \geq 0: \norm{\sy}_{X,s}^2 \geq R \right\}$.

\end{lemma}

    \begin{proof}
   Property (\ref{supposition 1}) implies that for any given $j \in \N$ we can choose an $n_j\in \N$ such that for all $k \geq n_j$, \begin{equation}\label{a good property}\mathbbm{E}\left(\norm{\sy^k - \sy^{n_j}}^2_{X,\tau
_{n_{j}}^{M,t}\wedge \tau_{k}^{M,t}} \right) \leq 2^{-4j}.\end{equation}
We shall make use of highly sensitive manipulations of the subsequence indexed by $(n_j)$, and for this we introduce a new sequence of stopping times. We now impose that $$M > 2 + \left\Vert\sup_{n \in \N}\norm{\sy^n}_{X,0}^2\right\Vert_{L^\infty(\Omega;\R)}$$ and define $$\tilde{M}^2:= \frac{M -\left\Vert\sup_{n }\norm{\sy^n}_{X,0}^2\right\Vert_{L^\infty(\Omega;\R)}}{2}.$$
The purpose of this is to define $$\sigma^M_j:= T \wedge \inf\left\{s > 0: \norm{\sy^{n_{j}}}_{X,s}\geq (\tilde{M} - 1 +2^{-j}) + \norm{\sy^{n_{j}}}_{X,0} \right\}$$
and ensure that $\sigma^M_j \leq \tau^{M,T}_{n_j}$ at every $\omega$. Note the key difference in not squaring the norm, and also that $\tilde{M}>1$ so each $\sigma^M_j$ is necessarily positive. To demonstrate the inequality, it is sufficient to show that, $\mathbb{P}-a.s.$, \begin{equation} \label{or more easily}
    \left((\tilde{M} - 1 +2^{-j}) + \norm{\sy^{n_{j}}}_{X,0}\right)^2 \leq M +  \norm{\sy^{n_{j}}}_{X,0}^2
\end{equation}
or more easily 
$$  \left(\tilde{M} + \norm{\sy^{n_{j}}}_{X,0}\right)^2 \leq M +  \norm{\sy^{n_{j}}}_{X,0}^2.$$
This is possible as 
\begin{align*}
    \left(\tilde{M} + \norm{\sy^{n_{j}}}_{X,0}\right)^2 \leq 2\tilde{M}^2 + 2\norm{\sy^{n_{j}}}_{X,0}^2 &= M -\left\Vert\sup_{n }\norm{\sy^n}_{X,0}^2\right\Vert_{L^\infty(\Omega;\R)} + 2\norm{\sy^{n_{j}}}_{X,0}^2\\ &\leq M + \norm{\sy^{n_{j}}}_{X,0}^2.
\end{align*}
The property (\ref{or more easily}) is thus verified, so $\sigma^M_j \leq \tau^{M,T}_{n_j}$ and hence the subsequence $(\sy^{n_j})$ enjoys the same properties up until the corresponding $\sigma^M_j$. In particular from (\ref{a good property}), \begin{equation}\label{a good property 2}
    \mathbbm{E}\left(\norm{\sy^{n_{j+1}} - \sy^{n_j}}_{X,\sigma^M
_j\wedge \sigma_{j+1}^{M}} \right) \leq \left[\mathbbm{E}\left(\norm{\sy^{n_{j+1}} - \sy^{n_j}}_{X,\sigma^M
_j\wedge \sigma_{j+1}^{M}}^2 \right)\right]^{\frac{1}{2}}\leq 2^{-2j}
\end{equation}
hence in defining the sets
\begin{equation}\label{defined omega j}\Omega_j:=\left\{\omega \in \Omega: \norm{\sy^{n_{j+1}}(\omega)-\sy^{n_j}(\omega)}_{X,\sigma^{M}_{j}(\omega)\wedge \sigma^{M}_{j+1}(\omega)} < 2^{-(j+2)} \right\}\end{equation}
we have, by Chebyshev's Inequality and (\ref{a good property 2}),
\begin{align*}
    \mathbb{P}\left(\Omega_j^C \right) \leq 2^{j+2}\mathbbm{E}\left(\norm{\sy^{n_{j+1}}-\sy^{n_j}}_{X,\sigma^{M}_{j}\wedge \sigma^{M}_{j+1}} \right) \leq 2^{-j + 2}.
\end{align*}
We have, therefore, that $$\sum_{j=1}^\infty \mathbb{P}\left( \Omega_j^C   \right) < \infty$$ from which we see $$\mathbb{P} \left( \bigcap_{K=1}^\infty \bigcup_{j=K}^\infty \Omega_j^C \right) = 0$$ courtesy of Borel-Cantelli. It then follows that the set $$\hat{\Omega} := \bigcup_{K=1}^\infty \bigcap_{j=K}^\infty \Omega_j$$ is such that $\mathbb{P}( \hat{\Omega}) = 1$ so that in verifying $\mathbb{P}-a.e.$ properties, we can in fact simply show that they hold everywhere on $\hat{\Omega}$. More precisely, we also take $\hat{\Omega}$ to be such that every $\norm{\sy^{n_j}}_{X,\cdot}$ is continuous on $\hat{\Omega}$, which is only a further countable intersection of full measure sets. We proceed by considering the sets $$\hat{\Omega}_K:= \bigcap_{j=K}^\infty \Omega_j$$ with the idea to just show such properties on $\hat{\Omega}_K$ for all $K$ (as their union makes up $\hat{\Omega}$). We look to construct a new stopping time $\sigma^M_{\infty}$ (which will prove to be the desired $\tau^{M,T}_{\infty}$) given as the $\mathbb{P}-a.e.$ limit of $(\sigma^M_j)$, built from demonstrating that $(\sigma^M_j)$ is monotone decreasing everywhere on $\hat{\Omega}_K$ for all $K$. In other words we show that for sufficiently large $j$ (in fact, just $j\geq K$) that the set \begin{equation} \label{def of set} \{ \sigma^M_j < \sigma^M_{j+1} \} \cap \hat{\Omega}_K\end{equation} is empty. Firstly we observe from the strict inequality $\sigma^M_j < \sigma^M_{j+1}$ on this set that $\sigma^M_j < T$, implying that 
$$\sigma^M_j= \inf\left\{s > 0: \norm{\sy^{n_{j}}}_{X,s}\geq (\tilde{M} - 1 +2^{-j}) + \norm{\sy^{n_{j}}}_{X,0} \right\}$$
so by the continuity of $\norm{\sy^{n_j}}_{X,\cdot}$, \begin{equation}\label{by the continuity}\norm{\sy^{n_{j}}}_{X,\sigma^M_j} =(\tilde{M} - 1 +2^{-j}) + \norm{\sy^{n_{j}}}_{X,0}. \end{equation} Using the definition of $\Omega_j$, (\ref{defined omega j}), for $j \geq K$, we have that \begin{equation}\label{it is useful}\norm{\sy^{n_j}}_{X,\sigma^M_j \wedge \sigma^M_{j+1}} - \norm{\sy^{n_{j+1}}}_{X,\sigma^M_j \wedge \sigma^M_{j+1}} \leq \norm{\sy^{n_{j+1}} - \sy^{n_j} }_{X,\sigma^M_j \wedge \sigma^M_{j+1}} < 2^{-\left(j+2\right)}\end{equation}
and also
\begin{equation} \label{useful ic}
    \norm{\sy^{n_{j+1}}}_{X,0} - \norm{\sy^{n_j}}_{X,0}  \leq \norm{\sy^{n_{j+1}} - \sy^{n_j} }_{X,0} < 2^{-\left(j+2\right)}.
\end{equation}
Combining (\ref{by the continuity}), (\ref{it is useful}) and (\ref{useful ic}), whilst using that $\sigma^M_j < \sigma^M_{j+1}$, we see that \begin{align}
    \nonumber \norm{\sy^{n_{j+1}}}_{X,\sigma^M_j \wedge \sigma^M_{j+1}} &> \norm{\sy^{n_j}}_{X,\sigma^M_j \wedge \sigma^M_{j+1}} - 2^{-\left(j+2\right)}\\ \nonumber
    &= \norm{\sy^{n_j}}_{X,\sigma^M_j} - 2^{-\left(j+2\right)}\\ \nonumber
    &=  \norm{\sy^{n_{j}}}_{X,0} + (\tilde{M} - 1 +2^{-j}) - 2^{-\left(j+2\right)}\\  \nonumber
    &> \norm{\sy^{n_{j+1}}}_{X,0} - 2^{-(j+2)} + (\tilde{M} - 1 +2^{-j}) - 2^{-\left(j+2\right)}\\
    &= \norm{\sy^{n_{j+1}}}_{X,0} + (\tilde{M} - 1 +2^{-(j+1)}) \label{end of the align}
\end{align}
where in the last line we have used the manipulation
$$ - 2^{-(j+2)} - 2^{-(j+2)} + 2^{-j} = -2^{-j-1} + 2^{-j} = 2^{-j}\left( -2^{-1} + 1\right) = 2^{-j}\left( 2^{-1} \right) = 2^{-(j+1)}.$$ The hard work is done in showing that the set (\ref{def of set}) is empty, as on this set note that $$\norm{\sy^{n_{j+1}}}_{X,\sigma^M_j \wedge \sigma^M_{j+1}} \leq \norm{\sy^{n_{j+1}}}_{X, \sigma^M_{j+1}} \leq \norm{\sy^{n_{j+1}}}_{X,0} + (\tilde{M} - 1 +2^{-(j+1)})$$ which contradicts (\ref{end of the align}), hence (\ref{def of set}) must be empty. Thus on every $\hat{\Omega}_K$, and furthermore the whole of $\hat{\Omega}$, the sequence $(\sigma^M_j)$ is eventually monotone decreasing (and bounded below by $0$). Furthermore we define $\sigma^M_{\infty}$ as the pointwise limit $\lim_{j \rightarrow \infty}\sigma^M_j$ on $\hat{\Omega}$, which must itself be a stopping time as the $\mathbb{P}-a.s.$ limit of stopping times. As mentioned this shall prove to be our $\tau^{M,T}_\infty$, and for the existence of $\sy$ we show that on $\hat{\Omega}$ the subsequence $(\sy^{n_j})$ is Cauchy in $X_{\sigma^M_{\infty}}$. Every $\omega \in \hat{\Omega}$ belongs to $\hat{\Omega}_K$ for some $K$, and furthermore to $\hat{\Omega}_L$ for all $L > K$. We fix arbitrary $\omega$ and select an associated $K$. At this $\omega$, for any $j > k \geq K$, observe that
\begin{align*}
    \norm{\sy^{n_j} - \sy^{n_k}}_{X, \sigma^M_{\infty}} &= \norm{\sy^{n_{j}} -\sy^{n_{k+1}} + \sy^{n_{k+1}} - \sy^{n_k}}_{X, \sigma^M_{\infty}}\\ &\leq \norm{\sy^{n_{j}} -\sy^{n_{k+1}}}_{X, \sigma^M_{\infty}} + \norm{\sy^{n_{k+1}} - \sy^{n_k}}_{X, \sigma^M_{\infty}}\\ &\leq \norm{\sy^{n_{j}} -\sy^{n_{k+1}}}_{X, \sigma^M_{\infty}} + 2^{-(k+2)}\\& \leq \sum_{l=k}^{j}2^{-(l+2)}\\
    &\leq 2^{-(k+1)}
\end{align*}
having carried out an inductive argument in the penultimate step. We are thus free to take $K$ large enough so that this difference is arbitrarily small; therefore there exists a limit in the Banach Space $X_{\sigma^M_{\infty}}$, which we call $\sy$. The process $\norm{\sy}_{X,\cdot \wedge \sigma^M_{\infty}}$ is adapted and $\mathbb{P}-a.s.$ continuous, as $$\sup_{r \in [0,T]}\left\vert \norm{\sy}_{X,r \wedge \sigma^M_{\infty}} - \norm{\sy^{n_j}}_{X,r \wedge \sigma^M_{\infty}}\right\vert \leq \sup_{r \in [0,T]}\left\vert \norm{\sy - \sy^{n_j}}_{X,r \wedge \sigma^M_{\infty}}\right\vert = \norm{\sy - \sy^{n_j}}_{X,\sigma^M_{\infty}}$$ 
which has $\mathbb{P}-a.s.$ limit as $j \rightarrow \infty$ equal to zero. Thus $\norm{\sy}_{X,\cdot \wedge \sigma^M_{\infty}}$ is given, $\mathbb{P}-a.s.$, as the uniform in time limit of adapted and continuous processes, verifying the result. Moving on, it is now that we make use of (\ref{supposition 2}) much in the same way as we did for (\ref{supposition 1}). This will be done in the context of $\gamma:= \sigma^M_{\infty}$ and $\delta_j:= \sigma^M_j - \sigma^M_{\infty}$. Indeed for any $j \in \N$ we can choose an $m_j \in \N$ (where $m_j=n_l$ some $l$) such that for all $k \geq m_j$, $$\sup_{n\in\N}\mathbbm{E}\left(\norm{\sy^n}_{X,\sigma^M_k \wedge \tau^{M,T}_{n}}^2 - \norm{\sy^n}_{X,\sigma^M_{\infty} \wedge \tau^{M,T}_n}^2\right) \leq 2^{-2j}.$$ In particular, through a relabelling of $\sigma^M_{m_j}=\sigma^M_l$, $$\mathbbm{E}\left(\norm{\sy^{m_j}}_{X,\sigma^M_{m_j}}^2 - \norm{\sy^{m_j}}_{X,\sigma^M_{\infty}}^2\right) \leq 2^{-2j}$$
by choosing $n$ as $m_j$ and using that $\sigma^M_{\infty} \leq \sigma^M_k \leq \sigma^M_{m_j} \leq \tau^{M,T}_{m_j}$. In a familiar way we define $$\Omega'_j:= \left\{\norm{\sy^{m_j}}_{X,\sigma^M_{m_j}}^2 - \norm{\sy^{m_j}}_{X,\sigma^M_{\infty}}^2 < 2^{-(j+2)} \right\} $$
so that, just as we showed for (\ref{defined omega j}), $$\check{\Omega}_K:=\bigcap_{j=K}^\infty \Omega'_j, \qquad \check{\Omega}:=\bigcup_{K=1}^\infty \check{\Omega}_K, \qquad  \mathbb{P}\left(\check{\Omega}\right)=1.$$ For arbitrary given $R>0$, the plan now is to find a constant $M$ such that at every $\omega \in \hat{\Omega}\cap\check{\Omega}$, either $\sigma^M_{\infty} = T$ or $\norm{\sy}^2_{X,\sigma^M_{\infty}} \geq R$. In both instances it is clear that $\tau^{R,T} \leq \sigma^M_{\infty}$, thus proving the proposition. To this end we fix an $\omega \in \hat{\Omega}\cap\check{\Omega}$ such that $\sigma^M_{\infty} < T$. As $\sigma^M_{\infty}$ is the decreasing limit of $(\sigma^M_{m_j})$ then for sufficiently large $m_j$ we must also have that $\sigma^M_{m_j} < T$. Exactly as in (\ref{by the continuity}), \begin{equation}\label{newcontinuitything} \norm{\sy^{m_{j}}}_{X,\sigma^M_{m_j}} =(\tilde{M} - 1 +2^{-j}) + \norm{\sy^{m_{j}}}_{X,0}. \end{equation}
From the proven convergence we also have that for sufficiently large $m_j$, \begin{equation}\label{applynewestcauchy}
    \norm{\sy - \sy^{m_j}}_{X,\sigma^M_{\infty}} < 1
\end{equation}
which implies that $\norm{\sy}_{X,\sigma^M_{\infty}} > \norm{\sy^{m_j}}_{X,\sigma^M_{\infty}} -1$, and likewise as $\omega \in \check{\Omega}_K$ for some $K$, \begin{equation} \label{lkjh}
    \norm{\sy^{m_j}}_{X,\sigma^M_{m_j}}^2 - \norm{\sy^{m_j}}_{X,\sigma^M_{\infty}}^2 < 1.
\end{equation}
We fix an $m_j$ large enough so that (\ref{newcontinuitything}), (\ref{applynewestcauchy}) and (\ref{lkjh}) all hold. Substituting (\ref{newcontinuitything}) into (\ref{lkjh}) gives that $$\norm{\sy^{m_j}}_{X,\sigma^M_{\infty}}^2 > \left((\tilde{M} - 1 +2^{-j}) + \norm{\sy^{m_{j}}}_{X,0}\right)^2 -1 > (\tilde{M}-1)^2 -1.$$ If $\tilde{M} > 2$ then the expression on the right is positive and $$ \norm{\sy^{m_j}}_{X,\sigma^M_{\infty}} > \left((\tilde{M}-1)^2 -1 \right)^{\frac{1}{2}}.$$
Furthermore $$ \norm{\sy}_{X,\sigma^M_{\infty}} > \left((\tilde{M}-1)^2 -1 \right)^{\frac{1}{2}} -1$$ where the right hand side is of course monotone increasing and unbounded in $\tilde{M}$ and hence $M$. By choosing $M$ large enough such that $$ \left[\left((\tilde{M}-1)^2 -1 \right)^{\frac{1}{2}} -1\right]^2 > R$$ we complete the proof.
\end{proof}

\subsection{An Existence and Uniqueness Result in Infinite Dimensions} \label{general criterion sub}

For completeness we present a series of assumptions necessary to deduce the existence and uniqueness of local strong solutions (Definition \ref{definitionofregularsolution}) in infinite dimensions, for the It\^{o} SPDE (\ref{thespde}), which is $$\sy_t = \sy_0 + \int_0^t \mathcal{Q}\sy_sds + \int_0^t\mathcal{G}\sy_sd\mathcal{W}_s.$$ 
The result which we state was proven with Dan Crisan and Oana Lang in [\cite{goodair2022existence1}], Theorem 3.15. The proof is highly involved so we give only the statement here. This is far from the elegant result of the classical variational framework developed by Pardoux [\cite{pardoux1975equations}, \cite{pardoux2021stochastic}], but it gives an indication as to the sort of assumptions required when dealing with highly non-trivial SPDEs such as the Navier-Stokes Equation with transport noise. We state the assumptions for a triplet of embedded Hilbert Spaces $$V \hookrightarrow H \hookrightarrow U$$ and ask that there is a continuous bilinear form $\inner{\cdot}{\cdot}_{U \times V}: U \times V \rightarrow \R$ such that for $\phi \in H$ and $\psi \in V$, \begin{equation} \label{bilinear formog}
    \inner{\phi}{\psi}_{U \times V} =  \inner{\phi}{\psi}_{H}.
\end{equation}
The mappings $\mathcal{Q},\mathcal{G}$ are such that
    $\mathcal{Q}:[0,T] \times V \rightarrow U,
    \mathcal{G}:[0,T] \times V \rightarrow \mathscr{L}^2(\mathfrak{U};H)$ are measurable. We assume that $V$ is dense in $H$ which is dense in $U$. 

\begin{assumption} \label{assumption fin dim spaces}
There exists a system $(a_n)$ of elements of $V$ such that, defining the spaces $V_n:= \textnormal{span}\left\{a_1, \dots, a_n \right\}$ and $\mathcal{P}_n$ as the orthogonal projection to $V_n$ in $U$, then:
\begin{enumerate}
    \item There exists some constant $c$ independent of $n$ such that for all $\phi\in H$,
\begin{equation} \label{projectionsboundedonH}
    \norm{\mathcal{P}_n \phi}_H^2 \leq c\norm{\phi}_H^2.
\end{equation}
\item There exists a real valued sequence $(\mu_n)$ with $\mu_n \rightarrow \infty$ such that for any $\phi \in H$, \begin{align}
     \label{mu2}
    \norm{(I - \mathcal{P}_n)\phi}_U \leq \frac{1}{\mu_n}\norm{\phi}_H
\end{align}
where $I$ represents the identity operator in $U$.
\end{enumerate}
\end{assumption}

 These conditions are supplemented by a series of assumptions on the mappings. We shall use general notation $c_t$ to represent a function $c_\cdot:[0,\infty) \rightarrow \R$ bounded on $[0,T]$, evaluated at the time $t$. Moreover we define functions $K$, $\tilde{K}$ relative to some non-negative constants $p,\tilde{p},q,\tilde{q}$. We use a generic notation to define the functions $K: U \rightarrow \R$, $K: U \times U \rightarrow \R$, $\tilde{K}: H \rightarrow \R$ and $\tilde{K}: H \times H \rightarrow \R$ by
\begin{align*}
    K(\phi)&:= 1 + \norm{\phi}_U^{p}, \qquad
    K(\phi,\psi):= 1+\norm{\phi}_U^{p} + \norm{\psi}_U^{q},\\
    \tilde{K}(\phi) &:= K(\phi) + \norm{\phi}_H^{\tilde{p}}, \qquad
    \tilde{K}(\phi,\psi) := K(\phi,\psi) + \norm{\phi}_H^{\tilde{p}} + \norm{\psi}_H^{\tilde{q}}.
\end{align*}
 Distinct use of the function $K$ will depend on different constants but in no meaningful way in our applications, hence no explicit reference to them shall be made. In the case of $\tilde{K}$, when $\tilde{p}, \tilde{q} = 2$ then we shall denote the general $\tilde{K}$ by $\tilde{K}_2$. In this case no further assumptions are made on the $p,q$. That is, $\tilde{K}_2$ has the general representation \begin{equation}\label{Ktilde2}\tilde{K}_2(\phi,\psi) = K(\phi,\psi) + \norm{\phi}_H^2 + \norm{\psi}_H^2\end{equation} and similarly as a function of one variable.\\
 
 We state the subsequent assumptions for arbitrary elements $\phi,\psi \in V$, $\phi^n \in V_n$, $\eta \in H$ and $t \in [0,T]$, and a fixed $\gamma > 0$. Understanding $\mathcal{G}$ as a mapping $\mathcal{G}: [0,\infty) \times V \times \mathfrak{U} \rightarrow H$, we introduce the notation $\mathcal{G}_i(\cdot,\cdot):= \mathcal{G}(\cdot,\cdot,e_i)$.
 
% \begin{assumption} \label{measurabilityassumption}
%For any $T>0$, $\mathcal{Q}:[0,T] \times V \rightarrow U$ and  $\mathcal{G}:[0,T] \times V \rightarrow \mathscr{L}^2(\mathfrak{U};H)$ are measurable. 
%\end{assumption}
 
  \begin{assumption} \label{new assumption 1} \begin{align}
     \label{111} \norm{\mathcal{Q}(t,\boldsymbol{\phi})}^2_U +\sum_{i=1}^\infty \norm{\mathcal{G}_i(t,\boldsymbol{\phi})}^2_H &\leq c_t K(\boldsymbol{\phi})\left[1 + \norm{\boldsymbol{\phi}}_V^2\right],\\ \label{222}
     \norm{\mathcal{Q}(t,\boldsymbol{\phi}) - \mathcal{Q}(t,\boldsymbol{\psi})}_U^2 &\leq  c_t\left[K(\phi,\psi) + \norm{\phi}_V^2 + \norm{\psi}_V^2\right]\norm{\phi-\psi}_V^2,\\ \label{333}
    \sum_{i=1}^\infty \norm{\mathcal{G}_i(t,\boldsymbol{\phi}) - \mathcal{G}_i(t,\boldsymbol{\psi})}_U^2 &\leq c_tK(\phi,\psi)\norm{\phi-\psi}_H^2.
 \end{align}
 \end{assumption}

\begin{assumption} \label{assumptions for uniform bounds2}
 \begin{align}
   \label{uniformboundsassumpt1}  2\inner{\mathcal{P}_n\mathcal{Q}(t,\boldsymbol{\phi}^n)}{\boldsymbol{\phi}^n}_H + \sum_{i=1}^\infty\norm{\mathcal{P}_n\mathcal{G}_i(t,\boldsymbol{\phi}^n)}_H^2 &\leq c_t\tilde{K}_2(\boldsymbol{\phi}^n)\left[1 + \norm{\boldsymbol{\phi}^n}_H^2\right] - \gamma\norm{\boldsymbol{\phi}^n}_V^2,\\  \label{uniformboundsassumpt2}
    \sum_{i=1}^\infty \inner{\mathcal{P}_n\mathcal{G}_i(t,\boldsymbol{\phi}^n)}{\boldsymbol{\phi}^n}^2_H &\leq c_t\tilde{K}_2(\boldsymbol{\phi}^n)\left[1 + \norm{\boldsymbol{\phi}^n}_H^4\right].
\end{align}
\end{assumption}

\begin{assumption} \label{therealcauchy assumptions}
\begin{align}
  \nonumber 2\inner{\mathcal{Q}(t,\boldsymbol{\phi}) - \mathcal{Q}(t,\boldsymbol{\psi})}{\boldsymbol{\phi} - \boldsymbol{\psi}}_U &+ \sum_{i=1}^\infty\norm{\mathcal{G}_i(t,\boldsymbol{\phi}) - \mathcal{G}_i(t,\boldsymbol{\psi})}_U^2\\ \label{therealcauchy1} &\leq  c_{t}\tilde{K}_2(\boldsymbol{\phi},\boldsymbol{\psi}) \norm{\boldsymbol{\phi}-\boldsymbol{\psi}}_U^2 - \gamma\norm{\boldsymbol{\phi}-\boldsymbol{\psi}}_H^2,\\ \label{therealcauchy2}
    \sum_{i=1}^\infty \inner{\mathcal{G}_i(t,\boldsymbol{\phi}) - \mathcal{G}_i(t,\boldsymbol{\psi})}{\boldsymbol{\phi}-\boldsymbol{\psi}}^2_U & \leq c_{t} \tilde{K}_2(\boldsymbol{\phi},\boldsymbol{\psi}) \norm{\boldsymbol{\phi}-\boldsymbol{\psi}}_U^4.
\end{align}
\end{assumption}

\begin{assumption} \label{assumption for prob in V}
\begin{align}
   \label{probability first} 2\inner{\mathcal{Q}(t,\boldsymbol{\phi})}{\boldsymbol{\phi}}_U + \sum_{i=1}^\infty\norm{\mathcal{G}_i(t,\boldsymbol{\phi})}_U^2 &\leq c_tK(\boldsymbol{\phi})\left[1 +  \norm{\boldsymbol{\phi}}_H^2\right],\\\label{probability second}
    \sum_{i=1}^\infty \inner{\mathcal{G}_i(t,\boldsymbol{\phi})}{\boldsymbol{\phi}}^2_U &\leq c_tK(\boldsymbol{\phi})\left[1 + \norm{\boldsymbol{\phi}}_H^4\right].
\end{align}
\end{assumption}

\begin{assumption} \label{finally the last assumption}
 \begin{equation} \label{lastlast assumption}
    \inner{\mathcal{Q}(t,\phi)-\mathcal{Q}(t,\psi)}{\eta}_U \leq c_t(1+\norm{\eta}_H)\left[K(\phi,\psi) + \norm{\phi}_V + \norm{\psi}_V\right]\norm{\phi-\psi}_H.
    \end{equation}
\end{assumption}

Under these assumptions we have the existence of a local strong solution of the equation (\ref{thespde}), in the sense of Definition \ref{definitionofregularsolution}, and uniqueness in the sense of Theorem \ref{uniquey theorem}. This is proven in [\cite{goodair2022existence1}] Theorem 3.15.

\subsection{Applications}

We conclude these notes by considering a concrete application of the framework and results developed here. Our motivating example is the Navier-Stokes Equation with Stochastic Lie Transport, derived through the principle of Stochastic Advection by Lie Transport (SALT) introduced in [\cite{holm2015variational}]. A complete introduction to this equation and its physical relevance is given in [\cite{goodair2022navier}]. The equation is
\begin{equation} \label{number2equation}
    u_t - u_0 + \int_0^t\mathcal{L}_{u_s}u_s\,ds - \nu\int_0^t \Delta u_s\, ds + \int_0^t Bu_s \circ d\mathcal{W}_s + \nabla \rho_t= 0
\end{equation}
where $u$ represents the fluid velocity, $\nu$ the viscosity, $\rho$ the pressure\footnote{The pressure term is a semimartingale, and an explicit form for the SALT Euler Equation is given in [\cite{street2021semi}] Subsection 3.3} and $\mathcal{W}$ a Cylindrical Brownian Motion. We pose the equation on the three dimensional torus $\T$. The mapping $\mathcal{L}$ is defined for sufficiently regular functions $f,g:\mathcal{O} \rightarrow \R^3$ by $\mathcal{L}_fg:= \sum_{j=1}^3f^j\partial_jg.$ As in Subsection \ref{subs 2.2}, the operator $B$ is defined by its action on the basis vectors $(e_i)$ of $\mathfrak{U}$, relative to functions $(\xi_i):\T \rightarrow \R^3$, for sufficiently regular $f$ by the mapping $$B_i:f \mapsto \mathcal{L}_{\xi_i}f + \mathcal{T}_{\xi_i}f$$ where $\mathcal{L}$ is as before, and $\mathcal{T}$ is a new operator that we introduce defined by $$\mathcal{T}_{g}f := \sum_{j=1}^3 f^j\nabla g^j.$$
We require some functional analytic set up in order to frame the equation. 
Recall that any function $f \in L^2(\T;\R^3)$ admits the representation \begin{equation} \label{fourier rep}f(x) = \sum_{k \in \mathbb{Z}^3}f_ke^{ik\cdot x}\end{equation} whereby each $f_k \in \mathbb{C}^3$ is such that $f_k = \overbar{f_{-k}}$ and the infinite sum is defined as a limit in $L^2(\T;\R^3)$, see e.g. [\cite{robinson2016three}] Subsection 1.5 for details.

\begin{definition}
We define $\dot{L}^2(\T;\R^3)$ as the subset of $L^2(\T;\R^3)$ of functions $\phi$ such that $$\int_{\T}\phi \ d\lambda = 0.$$
$L^2_{\sigma}(\T;\R^3)$ is defined as the subset of $\dot{L}^2(\T;\R^3)$ of functions $f$ whereby for all $k \in \mathbbm{Z}^3$, $k \cdot f_k = 0$ with $f_k$ as in (\ref{fourier rep}). For general $m \in \N$ we introduce $W^{m,2}_{\sigma}(\T;\R^3)$ as the intersection of $W^{m,2}(\T;\R^3)$ respectively with $L^2_{\sigma}(\T;\R^3)$.
\end{definition}

Note that $W^{1,2}_{\sigma}(\T;\R^3)$ is precisely the subspace of $W^{1,2}(\T;\R^3)$ consisting of zero-average divergence free functions. We introduce the Leray Projector $\mathcal{P}$ as the orthogonal projection in $L^2(\T;\R^3)$ onto $L^2_{\sigma}(\T;\R^3)$. It is well known (see e.g. [\cite{temam2001navier}] Remark 1.6.) that for any $m \in \N$, $\mathcal{P}$ is continuous as a mapping $\mathcal{P}: W^{m,2}(\T;\R^3) \rightarrow W^{m,2}(\T;\R^3)$. To impose zero-average and divergence-free constraints on $u$, we instead consider a projected version of (\ref{number2equation}),
\begin{equation} \label{projected strato}
    u_t = u_0 - \int_0^t\mathcal{P}\mathcal{L}_{u_s}u_s\,ds - \nu\int_0^t A u_s\, ds - \int_0^t \mathcal{P}Bu_s \circ d\mathcal{W}_s 
\end{equation}
where we have introduced the Stokes Operator $A: W^{2,2}(\mathcal{O};\R^3) \rightarrow L^2_{\sigma}(\mathcal{O};\R^3)$ by $A:= -\mathcal{P}\Delta$. We can now understand this equation, (\ref{projected strato}), in the sense of Subsection \ref{subs 2.3} for the spaces \begin{align*}
    V&:= W^{3,2}_{\sigma}(\T;\R^3), \qquad H:= W^{2,2}_{\sigma}(\T;\R^3),\\
    U&:= W^{1,2}_{\sigma}(\T;\R^3), \qquad X:= L^2_{\sigma}(\T;\R^3).
\end{align*}
The following is then proven in [\cite{goodair2022navier}] Theorem 2.1.

\begin{theorem} \label{existence for NS Strat}
    For any given $\mathcal{F}_0-$measurable $u_0: \Omega \rightarrow W^{2,2}_{\sigma}(\T;\R^3)$ there exists a pair $(u,\tau)$ such that: $\tau$ is a $\mathbbm{P}-a.s.$ positive stopping time and $u$ is a process whereby for $\mathbbm{P}-a.e.$ $\omega$, $u_{\cdot}(\omega) \in C\left([0,T];W^{2,2}_{\sigma}(\T;\R^3)\right)$ and $u_{\cdot}(\omega)\mathbbm{1}_{\cdot \leq \tau(\omega)} \in L^2\left([0,T];W^{3,2}_{\sigma}(\T;\R^3)\right)$ for all $T>0$ with $u_{\cdot}\mathbbm{1}_{\cdot \leq \tau}$ progressively measurable in $W^{3,2}_{\sigma}(\T;\R^3)$, and moreover satisfying the identity
\begin{equation} \nonumber
  u_t = u_0 - \int_0^{t\wedge \tau}\mathcal{P}\mathcal{L}_{u_s}u_s\,ds -\nu\int_0^{t\wedge \tau} A u_s\, ds - \int_0^{t \wedge \tau} \mathcal{P}Bu_s \circ d\mathcal{W}_s 
\end{equation}
$\mathbbm{P}-a.s.$ in $L^2_{\sigma}(\T;\R^3)$ for all $t \geq 0$.
\end{theorem}

This result uses Theorem \ref{the conversion}, where a local strong solution of the It\^{o} form is proven as an application of Subsection \ref{general criterion sub}. We note that the framework developed in these notes has been used in considering other types of solutions of (\ref{number2equation}), particularly with boundary conditions embedded into the Hilbert Spaces. Such solutions include analytically and probabilistically weak in 3D [\cite{goodair2023zero}], analytically weak in 2D [\cite{goodair2023zero}], and globally strong in 2D [\cite{goodair2023navier}]. The weak solutions relied on the tightness results of Subsection \ref{subby tight}, whilst the strong solutions used Lemma \ref{amazing cauchy lemma}.\\

\textbf{Thanks:} I would like to give my sincerest thanks to Dan Crisan for the regular and extended discussions around these notes, his feedback throughout, and overall guidance during this process. They most certainly would not have been possible without him.

\newpage

\bibliographystyle{spmpsci}
\bibliography{mybibnewest2}

\end{document}